\documentclass[10pt, a4paper]{amsart}
\usepackage{amsmath, amssymb,amsthm}
\usepackage[extra]{tipa}
\usepackage{lettrine}
\usepackage[Sonny]{fncychap}
\usepackage[english]{babel}
\usepackage{amsmath, amssymb,amsthm}
\usepackage[all]{xy}
\usepackage{tabularx}
\usepackage[applemac]{inputenc}
\usepackage{graphicx}
\usepackage[dvipsnames]{xcolor}
\usepackage{hyperref}
\usepackage[applemac]{inputenc}
\usepackage{enumitem}
\usepackage{geometry}
\usepackage{prerex}
\usetikzlibrary{fit}
\usepackage{pdflscape}
\usepackage{tikz-cd} 
\usepackage{mathtools}

\DeclarePairedDelimiter\abs{\lvert}{\rvert}%
\makeatletter
\def\oversortoftilde#1{\mathop{\vbox{\m@th\ialign{##\crcr\noalign{\kern3\p@}%
      \sortoftildefill\crcr\noalign{\kern3\p@\nointerlineskip}%
      $\hfil\displaystyle{#1}\hfil$\crcr}}}\limits}

\def\sortoftildefill{$\m@th \setbox\z@\hbox{$\braceld$}%
  \braceld\leaders\vrule \@height\ht\z@ \@depth\z@\hfill\braceru$}

\makeatother

\usepackage{mathrsfs}

\DeclareMathOperator{\B}{\overline{B}}

\DeclareMathOperator{\Id}{Id}

\DeclareMathOperator{\Ker}{Ker}
\DeclareMathOperator{\AJ}{AJ}
\DeclareMathOperator{\lcm}{lcm}

\DeclareMathOperator{\Def}{Def}
\DeclareMathOperator{\alg}{alg}
\DeclareMathOperator{\an}{an}
\DeclareMathOperator{\ord}{ord}
\DeclareMathOperator{\Gal}{Gal}

\DeclareMathOperator{\Res}{Res}
\DeclareMathOperator{\Div}{Div}
\DeclareMathOperator{\Cusp}{Cusp}
\DeclareMathOperator{\Eis}{Eis}

\DeclareMathOperator{\Hom}{Hom}

\DeclareMathOperator{\Card}{Card}

\DeclareMathOperator{\Frob}{Frob}
\DeclareMathOperator{\SL}{SL}
\DeclareMathOperator{\GL}{GL}
\DeclareMathOperator{\PSL}{PSL}
\DeclareMathOperator{\rk}{rk}

\newcommand{\cf}{\textit{cf. }}
\newcommand{\ie}{\textit{i.e. }}

\newcommand{\eg}{\textit{e.g. }}
\theoremstyle{definition}

\newtheorem{rem}{Remark}[section]
\newtheorem{rems}{Remarks}
\theoremstyle{plain}
\newtheorem{thm}{Theorem}[section]

\newtheorem{lem}[thm]{Lemma}

\newtheorem{prop}[thm]{Proposition}

\title{Mixed modular symbols and the generalized cuspidal $1$-motive}
\author{Emmanuel Lecouturier}

\begin{document}
\maketitle
\begin{abstract}
We define and study the space of \textit{mixed modular symbols} for a given finite index subgroup $\Gamma$ of $\SL_2(\mathbf{Z})$. This is an extension of the usual space of modular symbols, which in some cases carries more information about Eisenstein series. We make use of mixed modular symbols to construct some $1$-motives related to the generalized Jacobian of modular curves. In the case $\Gamma = \Gamma_0(p)$ for some prime $p$, we relate our construction to the work of Ehud de Shalit on $p$-adic periods of $X_0(p)$.
\end{abstract}

\section{Introduction}\label{section_introduction}
\subsection*{Overview and motivations}\label{intro_overview}
Consider the torsion free abelian group $\mathcal{M}$ generated by the symbols $\{\alpha, \beta\}$ where $(\alpha, \beta) \in \mathbf{P}^1(\mathbf{Q})^2$ with the following relations
$$\{\alpha, \beta\}+\{\beta, \gamma\}+\{\gamma, \alpha\}=0$$
for $(\alpha, \beta, \gamma) \in \mathbf{P}^1(\mathbf{Q})^3$. We denote by $\GL_2^+(\mathbf{Q})$ the subgroup of $\GL_2(\mathbf{Q})$ consisting of positive determinant matrices. We let $\GL_2^+(\mathbf{Q})$ act on the left on $\mathbf{P}^1(\mathbf{Q})$. If $\Gamma$ is a finite index subgroup of $\SL_2(\mathbf{Z})$, we denote by $\mathcal{M}_{\Gamma}$ the largest torsion-free quotient of $\mathcal{M}$ on which $\Gamma$ acts trivially. This group is usually called the group of modular symbols of level $\Gamma$ (and weight $2$). The group $\mathcal{M}_{\Gamma}$ is canonically isomorphic to the relative homology group $H_1(X_{\Gamma}, \mathcal{C}_{\Gamma}, \mathbf{Z})$, where $X_{\Gamma}$ is the compact modular curve associated to $\Gamma$ and $\mathcal{C}_{\Gamma}$ is the set of cusps of $X_{\Gamma}$. By intersection duality, we have a canonical isomorphism $\mathcal{M}_{\Gamma} \xrightarrow{\sim} \Hom_{\mathbf{Z}}(H_1(Y_{\Gamma}, \mathbf{Z}), \mathbf{Z})$ where $Y_{\Gamma} = X_{\Gamma} - \mathcal{C}_{\Gamma}$ is the open modular curve. 

If $\Gamma$ is a congruence subgroup, there is a Hecke-equivariant injective map of $\mathbf{R}$-vector spaces
$$p_{\Gamma} \otimes \mathbf{R}: H_1(Y_{\Gamma}, \mathbf{Z}) \otimes_{\mathbf{Z}} \mathbf{R} \hookrightarrow \Hom_{\mathbf{C}}(M_2(\Gamma), \mathbf{C})$$
given by $c \otimes 1 \mapsto \left(f \mapsto \int_{c} 2i\pi f(z)dz\right)$. Here, $M_2(\Gamma)$ is the complex vector space of modular forms of weight $2$ and level $\Gamma$ (this includes Eisenstein series). The map $p_{\Gamma}$ is not surjective in general, since the dimension of the left hand side is $2g(\Gamma)+c(\Gamma)-1$ whereas the dimension of the right hand side is $2g(\Gamma)+2\cdot(c(\Gamma)-1)$, where $g(\Gamma)$ (resp. $c(\Gamma)$) is the genus (resp. the number of cusps) of $X_{\Gamma}$. Dually one copy of Eisenstein series is missing to the space $\mathcal{M}_{\Gamma} \otimes_{\mathbf{Z}} \mathbf{C}$, which is a well-known Hodge theoretic phenomenon for non-projective smooth curves.

In this note, we define an extension $\tilde{\mathcal{M}}_{\Gamma}$ of $\mathcal{M}_{\Gamma}$, called the space of \textit{mixed modular symbols} of level $\Gamma$, with rank equal to $2\dim_{\mathbf{C}} M_2(\Gamma)$. The definition does not require any assumption on $\Gamma$ (in particular, it could be a non-congruence subgroup of $\SL_2(\mathbf{Z})$). The map $p_{\Gamma}$ extends to a Hecke-equivariant map $\tilde{p}_{\Gamma} \otimes \mathbf{R} : \tilde{\mathcal{M}}_{\Gamma} \otimes_{\mathbf{Z}} \mathbf{R} \rightarrow  \Hom_{\mathbf{C}}(M_2(\Gamma), \mathbf{C})$. The construction makes use of a cocycle construction by Glenn Stevens \cite{Stevens_book}, and involves special values of $L$-functions. 

We define and study various objects related to the space of mixed modular symbols, namely Hecke operators, the complex conjugation, Manin symbols, intersection duality and its relation with the extended Petersson pairing of Don Zagier \cite{Zagier_Rankin} and Vinsentiu Pasol and Alexandru A. Popa \cite{Popa_Haberland}, relation with generalized Jacobians and the associated $\ell$-adic Galois representations. We also study in more details the particular case $\Gamma = \Gamma_0(p)$ were $p$ is prime. In this case, we relate our construction to the one of Ehud de Shalit on generalized $p$-adic periods \cite{deShalit}. This will be used in our forthcoming work on the Mazur-Tate conjecture in conductor $p$ \cite{Lecouturier_MT}. We now describe in more details our main results.

\subsection{Definition of mixed modular symbols}\label{intro_def_mixed_modsymb}
Consider the torsion-free abelian group $\tilde{\mathcal{M}}$ generated by the symbols $\{g,g'\}$ where $(g,g') \in \SL_2(\mathbf{Z})^2$ with the following relations:
\begin{enumerate}
\item $\{g,g'\}+\{g',g''\}+\{g'',g\}=0$ for all $(g,g',g'') \in \SL_2(\mathbf{Z})^3$;
\item $\{g,g'\}-\{\epsilon_1g, \epsilon_2 g'\}=0$ for all $(g,g') \in \SL_2(\mathbf{Z})^2$ and $(\epsilon_1,\epsilon_2)\in \{1,-1\}^2$; 
\item $\{g,gT^n\}-n\cdot \{g,gT\}=0$ for all $g\in \SL_2(\mathbf{Z})$, $n\in \mathbf{Z}$. In this note, we let $T = \begin{pmatrix} 1 & 1 \\ 0 & 1 \end{pmatrix}$.
\end{enumerate}
There is a left action of $\SL_2(\mathbf{Z})$ on $\tilde{\mathcal{M}}$, given by $g \cdot \{g',g''\} = \{gg',gg''\}$ for all $(g,g',g'') \in \SL_2(\mathbf{Z})^3$. We let $\tilde{\mathcal{M}}_{\Gamma}$ be the largest torsion-free quotient of $\tilde{\mathcal{M}}$ on which $\Gamma$ acts trivially. There is a surjective group homomorphism 
$$\pi_{\Gamma} : \tilde{\mathcal{M}}_{\Gamma} \rightarrow \mathcal{M}_{\Gamma}$$
given by $\pi_{\Gamma}(\{g,g'\}) = \{g\infty, g'\infty\}$. We show in Proposition \ref{rank_computation} that $\Ker(\pi_{\Gamma})$ is isomorphic to the cokernel of the group homomorphism $\mathbf{Z}\rightarrow \mathbf{Z}[\mathcal{C}_{\Gamma}]$ given by $1 \mapsto \frac{1}{d_{\Gamma}} \sum_{c \in \mathcal{C}_{\Gamma}} e_c\cdot [c]$, where $e_c$ is the width of the cusp $c$ and $d_{\Gamma} = \gcd(e_c)_{c \in \mathcal{C}_{\Gamma}}$. In particular, $\tilde{\mathcal{M}}_{\Gamma}$ has the required rank. There is a boundary map $\partial_{\Gamma} : \tilde{\mathcal{M}}_{\Gamma} \rightarrow \mathbf{Z}[\mathcal{C}_{\Gamma}]^0$ (the upper $0$ meaning the augmentation subgroup), whose kernel contains $H_1(Y_{\Gamma}, \mathbf{Z}) \hookrightarrow \tilde{\mathcal{M}}_{\Gamma}$ with finite index equal to $\frac{1}{d_{\Gamma}}\prod_{c \in \mathcal{C}_{\Gamma}} e_c$.

\subsection{The generalized period map $\tilde{p}_{\Gamma}$}\label{intro_generalized_periods}

In \S \ref{section_period_iso}, we extend the period map $p_{\Gamma}$ to a group homomorphism $$\tilde{p}_{\Gamma} : \tilde{\mathcal{M}}_{\Gamma} \rightarrow \Hom_{\mathbf{C}}(M_2(\Gamma), \mathbf{C}) \text{ .}$$
One of the main properties of $\tilde{p}_{\Gamma}$ is that it encodes the special values of $L$-functions. More precisely, for any $g \in \SL_2(\mathbf{Z})$ and any $f \in M_2(\Gamma)$, we have $\tilde{p}_{\Gamma}(\{g,gS\}) = L(f \mid g,1)$ where $S = \begin{pmatrix} 0 & -1 \\ 1 & 0 \end{pmatrix}$ and as usual $(f \mid g)(z) = (cz+d)^{-2}\cdot f(\frac{az+b}{cz+d})$. The map $p_{\Gamma} \otimes \mathbf{R}$ is not always an isomorphism (equivalently injective), for instance if $\Gamma$ is a principal congruence subgroup of level divisible by $6$. Nevertheless, we prove the following result.
\begin{thm}\label{intro_thm_period}
The map $\tilde{p}_{\Gamma} \otimes \mathbf{R}$ is an isomorphism if $\Gamma$ is a congruence subgroup of level $p^n$ for some prime $p$ and integer $n \geq 1$.
\end{thm}

If $\Gamma=\Gamma_0(N)$ or $\Gamma = \Gamma_1(N)$ for some integer $N \geq 1$, let $\mathbb{T}$ be the Hecke algebra over $\mathbf{Z}$ acting faithfully on $M_2(\Gamma)$. We define in \S \ref{paragraph_Hecke} an action of $\mathbb{T}$ on $\tilde{\mathcal{M}}_{\Gamma} \otimes_{\mathbf{Z}} \mathbf{Z}[\frac{1}{2N}]$. More precisely, if $n \geq 1$ is an integer prime to $2N$, then the Hecke operator $T_n$ stabilises $\tilde{\mathcal{M}}_{\Gamma}$. However, it turns out that if $p$ is a prime dividing $N$ (resp. $2$) then $U_p$ (resp. $T_2$ or $U_2$) sends $\tilde{\mathcal{M}}_{\Gamma}$ into $\frac{1}{p}\cdot \tilde{\mathcal{M}}_{\Gamma}$, so it may not stabilises $\tilde{\mathcal{M}}_{\Gamma}$ in general. The reason for this non-integrality phenomenon is that there is a priori no direct way to define a general double coset action on $\tilde{\mathcal{M}}_{\Gamma}$: while $\GL_2^+(\mathbf{Q})$ acts naturally on $\mathbf{P}^1(\mathbf{Q})$, it does not acts naturally on $\SL_2(\mathbf{Z})$. To resolve this issue, we embed $\tilde{\mathcal{M}}_{\Gamma}$ as a lattice inside a $\mathbf{Q}$-vector space on which there is a natural action of the double coset operators (this vector space is defined in a similar way as $\tilde{\mathcal{M}}_{\Gamma}$, replacing $\SL_2(\mathbf{Z})$ by $\GL_2^+(\mathbf{Q})$). We also define similarly an action of the Atkin--Lehner involution $W_N$ on $\tilde{\mathcal{M}}_{\Gamma} \otimes_{\mathbf{Z}} \mathbf{Z}[\frac{1}{N}]$. In \S \ref{paragraph_complex_conjugation}, we define an action of the complex conjugation on $\tilde{\mathcal{M}}_{\Gamma}$. The action of Hecke operators and the complex conjugation are compatible with the projection $\pi_{\Gamma}$ and with the embedding $H_1(Y_{\Gamma}, \mathbf{Z}) \hookrightarrow \tilde{\mathcal{M}}_{\Gamma}$.

\subsection{Manin symbols}\label{intro_Manin_symbols}
In \S \ref{paragraph_Manin}, we define Manin symbols in $\tilde{\mathcal{M}}_{\Gamma}$. Let $$\tilde{\xi}_{\Gamma} : \mathbf{Z}[\Gamma\backslash \PSL_2(\mathbf{Z})]\rightarrow \tilde{\mathcal{M}}_{\Gamma} $$ be the map defined by $\tilde{\xi}_{\Gamma}(\Gamma g) = \{g,gS\}$ for all $g\in \PSL_2(\mathbf{Z})$. This is the Manin map, and the element $\tilde{\xi}_{\Gamma}(\Gamma g)$ is the Manin symbol in $\tilde{\mathcal{M}}_{\Gamma}$ associated to $\Gamma g$. We give a simple and concrete description of the image of $\tilde{\xi}_{\Gamma}$. In particular, we prove the following.
\begin{thm}\label{intro_thm_Manin}
 Let $p\geq 5$ be a prime and $n \in \mathbf{N}$. If $\Gamma=\Gamma_1(p^n)$ or $\Gamma = \Gamma_0(p^n)$, then the image of $\tilde{\xi}_{\Gamma}$ has index dividing $3$, the divisibility being strict of and only if $p \equiv 1 \text{ (modulo 3}\text{)}$ and $\Gamma = \Gamma_0(p^n)$.
\end{thm}
If $\Gamma = \Gamma_1(N)$ or $\Gamma = \Gamma_0(N)$ for some odd $N \geq 1$, we were not able to determine when the image of $\tilde{\xi}_{\Gamma}$ has finite index in $\tilde{\mathcal{M}}_{\Gamma}$ (we know it is not of finite index when $N$ is even). The question seems to be related to additive number theory in $(\mathbf{Z}/N\mathbf{Z})^{\times}$. See Remark \ref{Manin_rmk} for more details. 

We also give a description of the Manin relations, \ie of the group $\Ker(\tilde{\xi}_{\Gamma})$. It turns out that the $2$-terms Manin relations $\tilde{\xi}(\Gamma g)+\tilde{\xi}(\Gamma gS)=0$ are satisfied, but the $3$-terms Manin relations $\tilde{\xi}(\Gamma g)+\tilde{\xi}(\Gamma gU)+\tilde{\xi}(\Gamma gU^2)=0$, where $U = \begin{pmatrix} 1 & -1 \\ 1 & 0 \end{pmatrix}$, are not satisfied in general. Instead, we have to replace them with a subgroup of relations which is described in Theorem \ref{generalized_Manin_thm}.

\subsection{Duality theory and the generalized Petersson product}\label{intro_duality}
In \S \ref{paragraph_Petersson}, we study an anti-symmetric bilinear pairing $\langle \cdot , \cdot \rangle : \tilde{\mathcal{M}}_{\Gamma}^* \times \tilde{\mathcal{M}}_{\Gamma}^* \rightarrow \frac{1}{6}\cdot \mathbf{Z}$ defined by the formula
\begin{align*}
\langle \varphi_1, \varphi_2 \rangle =  \frac{1}{6}\cdot \sum_{g \in \Gamma \backslash \SL_2(\mathbf{Z})} & \varphi_1(\{gS,g\}) \cdot \varphi_2(\{gTS,gT\}) - \varphi_1(\{gTS,gT\})\cdot \varphi_2(\{gS,g\})   \\&-4\cdot \varphi_1(\{g,gT\})\cdot\varphi_2(\{g,gS\})+4\cdot\varphi_1(\{g,gS\})\cdot \varphi_2(\{g,gT\}) \text{ ,}
\end{align*}
where  $\tilde{\mathcal{M}}_{\Gamma}^* = \Hom_{\mathbf{Z}}(\tilde{\mathcal{M}}_{\Gamma}, \mathbf{Z})$. 
By tensoring with $\mathbf{C}$, $\langle \cdot , \cdot \rangle$ extends to an anti-symmetric bilinear pairing on $\Hom_{\mathbf{Z}}(\tilde{\mathcal{M}}_{\Gamma}, \mathbf{C})$. The map $\tilde{p}_{\Gamma}$ induces a map $ \tilde{p}_{\Gamma}^* : M_2(\Gamma) \rightarrow \Hom_{\mathbf{Z}}(\tilde{\mathcal{M}}_{\Gamma}, \mathbf{C})$. Following Zagier, Pasol and Popa \cite{Popa_Haberland} extended the Pertersson pairing on $M_2(\Gamma)$, using a procedure of renormalization of divergent integrals. They generalize a formula of Klaus Haberland and Lo\"ic Merel as follows: for any $f_1$, $f_2$ $\in M_2(\Gamma)$, we have
$$
-8 i \pi^2 (f_1, f_2) = \frac{1}{[\SL_2(\mathbf{Z}):\Gamma]}\cdot  \langle \tilde{p}_{\Gamma}^*(f_1) , \overline{\tilde{p}_{\Gamma}^*(f_2)} \rangle \text{ .}
$$

We expect that $\langle \cdot , \cdot \rangle $ is $\mathbf{Z}$-valued and non-degenerate. More precisely, we expect that the determinant of the pairing $\langle \cdot , \cdot \rangle $ is $\frac{1}{d_{\Gamma}}\cdot \prod_{c \in \mathcal{C}_{\Gamma}} e_c$. In this direction, we prove the following result.

\begin{thm}\label{intro_thm_Petersson}
Assume that the cusps of $X_{\Gamma}$ are fixed by the complex conjugation (if $\Gamma = \Gamma_1(N)$, this is the case if and only if $N$ divides $2p$ for some prime $p$). Then $\langle \cdot , \cdot \rangle $ is perfect after inverting $2$ and the $\lcm$ of the widths of the cusps of $X_{\Gamma}$.
\end{thm}

See also Remark \ref{rem_generality_conjecture_G} and Proposition \ref{expression_G^+_rho_i} for a partial result in the general case where the cusps are not assumed to be real. 

\subsection{Relation with the generalized Jacobian}\label{intro_generalized_cuspidal_motive}
Assume in this paragraph that $\Gamma$ is a congruence subgroup of $\SL_2(\mathbf{Z})$, whose level is denoted by $N$. Recall that for each cusp $c \in \mathcal{C}_{\Gamma}$, we have denoted by $e_c$ the width of $c$. Fix an algebraic closure $\overline{\mathbf{Q}}$ of $\mathbf{Q}$ and an embedding $\overline{\mathbf{Q}} \hookrightarrow \mathbf{C}$. We let $\zeta_N = e^{\frac{2i\pi}{N}} \in \overline{\mathbf{Q}}$.  Let $k$ be the number field of definition of $X_{\Gamma}$ in $\overline{\mathbf{Q}}$; this is a subfield of $\mathbf{Q}(\zeta_N)$. 

Let $J_{\Gamma}$ be the Jacobian variety of $X_{\Gamma}$ over $k$. Let $J_{\Gamma}^{\#}$ be the generalized Jacobian variety of $X_{\Gamma}$ over $k$ with respect to the cuspidal divisor, \ie the sum of the closed points of $X_{\Gamma}\backslash Y_{\Gamma}$ over $k$. By definition, $J_{\Gamma}^{\#}$ parametrizes degree zero divisors supported on $Y_{\Gamma}$ modulo the divisors of functions which are constant ($\neq 0, \infty$) on the cusps. 

We have a  canonical exact sequence of group schemes over $k$:
\begin{equation}
1 \rightarrow \mathbf{G}_m \rightarrow \prod_{c \in X_{\Gamma}\backslash Y_{\Gamma}} \text{Res}_{k(c)/k}(\mathbf{G}_m)\rightarrow J_{\Gamma}^{\#} \rightarrow J_{\Gamma} \rightarrow 1 \text{ ,}
\end{equation}
where $k(c)\subset \mathbf{C}$ is the field of definition of the cusp $c$ and $\text{Res}_{k(c)/k}$ denotes the Weil restriction. 
Let $C_{\Gamma} \subset J_{\Gamma}$ be the cuspidal subgroup, \ie the subgroup generated by the image of the difference of the cusps of $X_{\Gamma}$. This is a finite group by the Manin--Drinfeld theorem \cite{Manin_Drinfeld}. Our goal is to try to define a reasonable ``lift'' $C^{\#}_{\Gamma}$ of $C_{\Gamma}$ in $J_{\Gamma}^{\#}$. For technical (but seemingly necessary) reasons, we only define a lift $C_{\Gamma}^{\natural}$ of $C_{\Gamma}$ in $J_{\Gamma}^{\natural} := J_{\Gamma}^{\#}/\prod_{c \in X_{\Gamma}\backslash Y_{\Gamma}} \text{Res}_{k(c)/k}(\mu_{N})$, where $\mu_n \subset \mathbf{G}_m$ is the group of $n$th roots of unity. Note that there is a canonical projection $$\beta_{\Gamma}^{\natural} : J_{\Gamma}^{\natural} \rightarrow J_{\Gamma}$$ induced by the projection $$\beta_{\Gamma}^{\#} : J_{\Gamma}^{\#} \rightarrow J_{\Gamma} \text{ .}$$

Let $\delta_{\Gamma} :  \mathbf{Z}[\mathcal{C}_{\Gamma}]^0 \rightarrow J_{\Gamma}$ be the map sending a divisor to its class in $J_{\Gamma}$. 
We define two maps $\delta_{\Gamma}^{\natural, \alg} : \mathbf{Z}[\mathcal{C}_{\Gamma}]^0 \rightarrow J_{\Gamma}^{\natural}(\mathbf{C})$ and $\delta_{\Gamma}^{\natural, \an} : \mathbf{Z}[\mathcal{C}_{\Gamma}]^0 \rightarrow J_{\Gamma}^{\natural}(\mathbf{C})$  such that $$\beta_{\Gamma}^{\natural} \circ \delta_{\Gamma}^{\natural, \alg} = \beta_{\Gamma}^{\natural} \circ \delta_{\Gamma}^{\natural, \an} = \delta_{\Gamma} \text{ .}$$ 
The map $\delta_{\Gamma}^{\natural, \alg}$ is defined in an algebraic way using uniformizers at cusps, whereas the map $\delta_{\Gamma}^{\natural, \an}$ is defined in an analytic way (via an Abel-Jacobi map) using mixed modular symbols and the generalized period map $\tilde{p}_{\Gamma}$. The fact that we were only able to define $\delta_{\Gamma}^{\natural, \alg}$ as a map valued in $J_{\Gamma}^{\natural}(\mathbf{C})$ and not in $J_{\Gamma}^{\#}(\mathbf{C})$  is explained by the fact that our choice of a uniformizer at a cusp $c$ is canonical only up to a $e_c$-th root of unity. Similarly, the fact that we were only able to define $\delta_{\Gamma}^{\natural, \an}$ as a map valued in $J_{\Gamma}^{\natural}(\mathbf{C})$ and not in $J_{\Gamma}^{\#}(\mathbf{C})$  is explained by the fact that the index of $H_1(Y_{\Gamma}, \mathbf{Z})$ in $\tilde{\mathcal{M}}_{\Gamma}$ is $\frac{1}{d_{\Gamma}}\cdot \prod_{c \in \mathcal{C}_{\Gamma}} e_c$.
For convenience, we refer to section \ref{generalized_cuspidal} for the precise definitions of $\delta_{\Gamma}^{\natural, \alg}$ and $\delta_{\Gamma}^{\natural, \an}$. 

While the map $\delta_{\Gamma}^{\natural, \alg}$ is easily seen to take values in $J_{\Gamma}^{\natural}(\mathbf{Q}(\zeta_N))$, it is unclear \textit{a priori} whether the map $\delta_{\Gamma}^{\natural, \an}$ takes values in $J_{\Gamma}^{\natural}(\overline{\mathbf{Q}})$. This is in fact the case, as we show in the following comparison result between $\delta_{\Gamma}^{\natural, \alg}$ and $\delta_{\Gamma}^{\natural, \an}$.
\begin{thm}\label{intro_main_thm_comparison}
\begin{enumerate}
\item \label{comparison_i} Let $n$ be the order of the cuspidal subgroup $C_{\Gamma}$ of $J_{\Gamma}$. Then we have
$$n^2\cdot \delta_{\Gamma}^{\natural, \alg} = n^2\cdot  \delta_{\Gamma}^{\natural, \an} \text{ .}$$
In particular, $\delta_{\Gamma}^{\natural, \an}$ takes values in $J_{\Gamma}^{\natural}(\mathbf{Q}(\zeta_N, \zeta_{n^2}))$. 
\item \label{comparison_ii} Assume that $N$ is odd, that $\begin{pmatrix} -1 & 0 \\ 0 & 1 \end{pmatrix}$ normalizes $\Gamma$, and that all the cusps in $\mathcal{C}_{\Gamma}$ are fixed by the complex conjugation, \ie for all $\frac{p}{q} \in \mathbf{P}^1(\mathbf{Q})$ we have $\Gamma \cdot (-\frac{p}{q}) = \Gamma \cdot \frac{p}{q}$. Then we have $\delta_{\Gamma}^{\natural, \alg} \equiv  \delta_{\Gamma}^{\natural, \an}$ up to some element in the image of $\prod_{c \in \mathcal{C}_{\Gamma}} \{\pm 1\}$ in $J_{\Gamma}^{\natural}(\mathbf{C})$. 
 \end{enumerate}
\end{thm}
\subsection{Additional results in the case $\Gamma = \Gamma_0(p)$}\label{intro_Gamma_0(p)}

Assume in this paragraph that $\Gamma = \Gamma_0(p)$ for some prime $p$. The modular curve $X_{\Gamma_0(p)}$ has two cusps, namely $\Gamma_0(p) \infty$ and $\Gamma_0(p) 0$. Let $n$ be the order of the cuspidal subgroup of $J_{\Gamma_0(p)}$; Barry Mazur proved that $n = \frac{p-1}{d}$ where $d = \gcd(p-1,12)$. Let $j : X_{\Gamma_0(p)} \rightarrow \mathbf{P}^1$ be the usual $j$-invariant map. Let $\mathbb{T}$ be the Hecke algebra over $\mathbf{Z}$ acting faithfully on $M_2(\Gamma_0(p))$.

\subsubsection{The generalized cuspidal $1$-motive}\label{intro_refined_Gamma_0(p)}
We define maps $$\delta_{\Gamma_0(p)}^{\#, \alg} : \mathbf{Z}[\mathcal{C}_{\Gamma_0(p)}]^0 \rightarrow J_{\Gamma_0(p)}^{\#}(\mathbf{Q})$$ and  $$\delta_{\Gamma_0(p)}^{\#, \an} : \mathbf{Z}[\mathcal{C}_{\Gamma_0(p)}]^0 \rightarrow J_{\Gamma_0(p)}^{\#}(\mathbf{C})$$ lifting our previous maps $\delta_{\Gamma_0(p)}^{\natural, \alg}$ and $\delta_{\Gamma_0(p)}^{\natural, \an}$. The reason we were able to construct $\delta_{\Gamma_0(p)}^{\#, \alg}$ is that there is a canonical uniformizer at the cusp $\Gamma_0(p) \infty$ (resp. $\Gamma_0(p) 0$), namely $j^{-1}$ (resp. $(j\circ w_p)^{-1}$ where $w_p$ is the Atkin--Lehner involution). The reason we were able to construct $\delta_{\Gamma_0(p)}^{\#, \an}$ is that there is a canonical map $M_2(\Gamma_0(p)) \rightarrow \mathbf{C}$ given by $f \mapsto -L(f,1)$ where $L(f,s)$ is the complex $L$-function attached to $f$. This map, together with the (generalized) Abel--Jacobi isomorphism, gives a point in $\delta_{\Gamma_0(p)}^{\natural, \an}$ corresponding to $\delta_{\Gamma_0(p)}^{\#, \alg}((\Gamma_0(p) \infty) - (\Gamma_0(p) 0))$. For convenience, we refer to \S \ref{paragraph_applications_Gamma_0(p)} for the precise definitions of $\delta_{\Gamma_0(p)}^{\#, \alg}$ and $\delta_{\Gamma_0(p)}^{\#, \an}$.

We prove the following result, which is a refinement of Theorem \ref{intro_main_thm_comparison}.
\begin{thm}\label{intro_main_comparison_Gamma_0_case}
\begin{enumerate}
\item \label{main_comparison_Gamma_0_case_i} The maps $\delta_{\Gamma_0(p)}^{\#, \alg}$ and $\delta_{\Gamma_0(p)}^{\#, \an}$ are $\mathbb{T}$-equivariant.
\item \label{main_comparison_Gamma_0_case_ii} The element $n \cdot \delta_{\Gamma_0(p)}^{\#, \alg}((\Gamma_0(p) \infty)-(\Gamma_0(p) 0))$ of $J_{\Gamma_0(p)}^{\#}(\mathbf{Q})$ is the image of $(1,p^{\frac{12}{d}}) \in \mathbf{Q}^{\times} \times \mathbf{Q}^{\times}$. 
\item \label{main_comparison_Gamma_0_case_iii} We have $\delta_{\Gamma_0(p)}^{\#, \alg} \equiv \delta_{\Gamma_0(p)}^{\#, \an}$ modulo the image of $\mu_{\gcd(2,n)} \times \mu_{\gcd(2,n)}$ in $J_{\Gamma_0(p)}^{\#}(\mathbf{C})$. In particular, $\delta_{\Gamma_0(p)}^{\#, \an}$ takes values in $J_{\Gamma_0(p)}^{\#}(\mathbf{Q})$.
\end{enumerate}
\end{thm}

The map $\delta_{\Gamma_0(p)}^{\#, \alg}$ can be considered as a $1$-motive $\mathbf{Z} \rightarrow J_{\Gamma_0(p)}^{\#}$ over $\mathbf{Q}$, which we call the \textit{generalized cuspidal $1$-motive}. Theorem \ref{intro_main_comparison_Gamma_0_case} (\ref{main_comparison_Gamma_0_case_iii}) describes the Betti realization of this $1$-motive (up to a sign ambiguity).

\subsubsection{Relation with the generalized $p$-adic period pairing of de Shalit}\label{intro_deShalit}
There is a $p$-adic analogue of the constructions of \S \ref{intro_generalized_cuspidal_motive}, coming from the work of de Shalit \cite{deShalit} which we briefly recall below. Fix an algebraic closure $\overline{\mathbf{Q}}_p$ of $\mathbf{Q}_p$, and let $K \subset \overline{\mathbf{Q}}_p$ be the quadratic unramified extension of $\mathbf{Q}_p$. We also fix an algebraic closure $\mathbf{C}_p$ of the $p$-adic completion of $\overline{\mathbf{Q}}_p$. Let $S$ be the set of isomorphisms classes of supersingular elliptic curves over $\overline{\mathbf{F}}_p$. This is a finite set since the $j$-invariant of an element of $S$ is known to lie in $\mathbf{F}_{p^2}$. More precisely, we have $\Card(S)=g+1$ where $g$ is the genus of $X_0(p)$; we write $S  = \{e_0, ..., e_g\}$.  We denote by $\mathbf{Z}[S]$ the free $\mathbf{Z}$-module with basis the elements of $S$ (this is usually called the \textit{supersingular module}) and by $\mathbf{Z}[S]^0$ its augmentation subgroup (the degree zero elements). There is a canonical bilinear pairing called the \textit{$p$-adic period pairing} 
$$Q^0 : \mathbf{Z}[S]^0 \times \mathbf{Z}[S]^0 \rightarrow K^{\times}$$
inducing a map $q^0 : \mathbf{Z}[S]^0 \rightarrow \Hom(\mathbf{Z}[S]^0, K^{\times})$ via the formula $q^0(x)(y) = Q(x,y)$. The theory of $p$-adic uniformization gives a canonical $\Gal(\mathbf{C}_p/K)$-equivariant isomorphism
$$J_{\Gamma_0(p)}(\mathbf{C}_p) \xrightarrow{\sim} \Hom(\mathbf{Z}[S]^0, \mathbf{C}_p^{\times})/q^0(\mathbf{Z}[S]^0) \text{ .}$$
de Shalit constructed in \cite{deShalit} a bilinear pairing
$$Q : \mathbf{Z}[S] \times \mathbf{Z}[S] \rightarrow K^{\times}$$
extending $Q^0$ and 
with the property that there is a canonical  $\Gal(\mathbf{C}_p/K)$-equivariant isomorphism
\begin{equation}\label{p-adic_uniformization_generalized}
J_{\Gamma_0(p)}^{\#}(\mathbf{C}_p) \xrightarrow{\sim} \Hom(\mathbf{Z}[S], \mathbf{C}_p^{\times})/q(\mathbf{Z}[S]^0) \text{ ,}
\end{equation}
where $q : \mathbf{Z}[S] \rightarrow \Hom(\mathbf{Z}[S], K^{\times})$ is such that $q(x)(y) = Q(x,y)$. We apologize to the reader for not recalling the precise construction of $Q$, as it is quite involved and is beautifully done in \cite{deShalit}. As de Shalit notes, the pairing $Q^0$ is canonical, but the choice of $Q$ depends on a choice of a tangent vector at the cusp $\Gamma_0(p)\infty$ \cite[\S 1.1]{deShalit}. This corresponds to the choice of the uniformizer $j^{-1}$ at $\Gamma_0(p) \infty$. de Shalit also proved that $Q$ is non-degenerate (this is even true after composing with the $p$-adic valuation $K^{\times} \rightarrow \mathbf{Z}$ since we essentially get the Kronecker pairing, \cf \cite[\S 1.6 Main Theorem]{deShalit}). This means that $q$ is injective.

There is a group homomorphism 
$$\delta^{\#, p\text{-adic}} : \mathbf{Z}[\mathcal{C}_{\Gamma_0(p)}]^0 \rightarrow J_{\Gamma_0(p)}^{\#}(\mathbf{C}_p) $$
defined by 
$$\delta^{\#, p\text{-adic}}(\Gamma_0(p) 0 - \Gamma_0(p) \infty) = \text{image of }q(x) \text{ modulo }q(\mathbf{Z}[S]^0)$$
via (\ref{p-adic_uniformization_generalized}) for any $x \in \mathbf{Z}[S]$ of degree $1$ (this does not depend on $x$).

We prove the following comparison result in \S \ref{section_deShalit}.
\begin{thm}\label{intro_main_thm_comparison_p_adic}
We have $\delta^{\#, p\text{-adic}} = \delta^{\#, \alg}$.
\end{thm}

\subsubsection{The $\ell$-adic realization}\label{intro_Galois_representations}

As an application of the above results of \S \ref{intro_Gamma_0(p)}, we construct certain modular Galois representations in a ``geometric'' way.

Let $\ell \geq 2$ be a prime and $\mathfrak{m}$ be a maximal ideal of $\mathbb{T}$ of residue characteristic $\ell$. The $\ell$-adic Tate module of the generalized cuspidal $1$-motive $\delta_{\Gamma_0(p)}^{\#, \alg}  : \mathbf{Z}[\mathcal{C}_{\Gamma_0(p)}]^0 \rightarrow J_{\Gamma_0(p)}^{\#}(\mathbf{Q})$ is denoted by $\mathcal{V}_{\ell}$, and its $\mathfrak{m}$-adic completion is denoted by $\mathcal{V}_{\mathfrak{m}}$(\cf \S \ref{application_Galois_representations} for the precise definition). We prove in Proposition \ref{multiplicity_one_V} that $\mathcal{V}_{\mathfrak{m}}$ is a free $\mathbb{T}_{\mathfrak{m}}$-module of rank $2$, except possibly if $\ell=2$ and $\mathfrak{m}$ is supersingular (here, $\mathbb{T}_{\mathfrak{m}}$ is the $\mathfrak{m}$-adic completion of $\mathbb{T}$). If $\mathfrak{m}$ is non-Eisenstein (meaning it does not contain the Eisenstein ideal defined by Mazur), then $\mathcal{V}_{\mathfrak{m}}$ is canonically isomorphic to the usual Galois representation constructed from the $\mathfrak{m}$-adic Tate module of $J_{\Gamma_0(p)}$. We prove in \S \ref{application_Galois_representations} the following result if $\mathfrak{m}$ is Eisenstein.

\begin{thm}\label{intro_thm_galois_rep}
Assume that $\mathfrak{m}$ is Eisenstein (in particular, $\ell$ divides the numerator of $\frac{p-1}{12}$).
We can choose a basis of $\mathcal{V}_{\mathfrak{m}}$ as a $\mathbb{T}_{\mathfrak{m}}$-module such that the following properties hold for the associated representation $\rho : \Gal(\overline{\mathbf{Q}}/\mathbf{Q}) \rightarrow \GL_2(\mathbb{T}_{\mathfrak{m}})$. 
\begin{enumerate}
\item\label{intro_galois_prop1} The reduction of $\rho$ modulo $\mathfrak{m}$ is the residual representation $\overline{\rho} :  \Gal(\overline{\mathbf{Q}}/\mathbf{Q}) \rightarrow \GL_2(\mathbf{F}_{\ell})$ given by $ \overline{\rho} = \begin{pmatrix} \overline{\chi}_{\ell} & \overline{b} \\ 0 & 1 \end{pmatrix}$, where $\chi_{\ell} : \Gal(\overline{\mathbf{Q}}/\mathbf{Q}) \rightarrow \mathbf{Z}_{\ell}^{\times}$ is the $\ell$-adic cyclotomic character, $\overline{\chi}_{\ell}$ is the reduction of $\chi_{\ell}$ modulo $\ell$ and $\overline{b} : \Gal(\overline{\mathbf{Q}}/\mathbf{Q}) \rightarrow \mathbf{F}_{\ell}$ is a Kummer cocycle in $Z^1( \Gal(\overline{\mathbf{Q}}/\mathbf{Q}), \overline{\chi}_{\ell})$ whose kernel cut out a number field isomorphic to $\mathbf{Q}(p^{\frac{1}{\ell}})$.
\item \label{intro_galois_prop2} The representation $\rho$ is unramified outside $p$ and $\ell$, has determinant $\chi_{\ell}$ and is finite flat at $\ell$.
\item \label{intro_galois_prop3} There is a free $\mathbb{T}_{\mathfrak{m}}$-submodule of rank one in $\mathcal{V}_{\mathfrak{m}}$ which is pointwise fixed by the inertia subgroup at $p$. 
\item \label{intro_galois_prop4} For all primes $q \neq \ell, p$, the trace of $\rho(\Frob_q)$ is the Hecke operator $T_q$, where $\Frob_q$ is any (arithmetic) Frobenius element at $q$.
\item \label{intro_galois_prop5} If $\ell \geq 5$ then the representation $\rho$ is universal for the above properties, so we have an isomorphism $R \xrightarrow{\sim} \mathbb{T}_{\mathfrak{m}}$ where $R$ is the universal deformation rings with the prescribed above properties.
\end{enumerate}
\end{thm}

Theorem \ref{intro_thm_galois_rep} is similar to a result of Frank Calegari and Matthew Emerton \cite[Theorem 1.5]{Calegari_Emerton}, although the residual Galois representation they consider is $ \begin{pmatrix} \overline{\chi}_{\ell} & 0 \\ 0 & 1 \end{pmatrix}$. Property (\ref{intro_galois_prop5}) is an immediate consequence of a result of Preston Wake and Carl Wang-Erickson \cite[Corollary 7.1.3]{WWE} (the restriction $\ell \geq 5$ comes from there, but we expect that the result still holds for $\ell \in \{2,3\}$). Our contribution here is really to the construction of the Galois representation $\rho$ satisfying the above properties. While we could maybe prove the existence of $\rho$ using a gluing argument using Ribet's Lemma, the construction here is more geometric in nature since it is related to the generalized Jacobian of $X_0(p)$.

\begin{rem}\label{intro_rem_galois_rep_duality}
There exists a perfect pairing of Galois modules $\mathcal{V}_{\ell} \times \mathcal{V}_{\ell} \rightarrow \mathbf{Z}_{\ell}(1)$, where as usual $\mathbf{Z}_{\ell}(1)$ is $\mathbf{Z}_{\ell}$ with the Galois action given by $\chi_{\ell}$ (\cf Proposition \ref{galois_rep_duality}). It seems therefore reasonable to expect that the $1$-motive $\mathbf{Z} \rightarrow J_{\Gamma_0(p)}^{\#}$ itself is self-dual.

If $\Gamma$ is a congruence subgroup of level $N$ and $\ell$ is a prime not dividing $N$, then we can define  similarly a canonical Hecke and Galois module by considering the $\ell$-adic realization of the $1$-motive $\delta_{\Gamma}^{\natural, \alg}$. We do not know whether this module is self-dual up to twist. An interesting case would be $\Gamma = \Gamma_1(p)$, since it would simplify some of our arguments in \cite{Lecouturier_MT}.
\end{rem}

\subsection{Acknowledgements}
I would like to thank my former Phd advisor Lo\"ic Merel for his helpful suggestions and continuous support. I would also like to thank Takao Yamazaki for answering some questions about generalized Jacobians, and pointing out the usefulness of the Weil reciprocity law in \S \ref{subsection_comparison}. This research was funded by Tsinghua University and the Yau Mathematical Sciences Center.

\section{Mixed modular symbols}
We keep the notation of \S \ref{section_introduction}.

\subsection{First general properties}\label{section_first_properties}

There is a group homomorphism $i_{\Gamma} : \Gamma \rightarrow \tilde{\mathcal{M}}_{\Gamma}$ given by $\gamma \mapsto \{1, \gamma\}$. We have a surjective group homomorphism $$\Pi_{\Gamma} : \Gamma \rightarrow H_1(Y_{\Gamma}, \mathbf{Z})$$ sending $\gamma \in \Gamma$ to the image in $Y_{\Gamma}$ of the geodesic path in the upper-half plane between $z_0$ and $\gamma(z_0)$ (for any $z_0$ in the upper-half plane). The kernel of $\Pi_{\Gamma}$ is generated by commutators and elliptic elements. Thus, the map $i_{\Gamma}$ factors through $\Pi_{\Gamma}$, so induces a group homomorphism $$\iota_{\Gamma} : H_1(Y_{\Gamma}, \mathbf{Z}) \rightarrow \tilde{\mathcal{M}}_{\Gamma} \text{ .}$$

Recall the following construction due to Stevens \cite[\S 2.3]{Stevens_book}. Let $M_2$ be the $\mathbf{C}$-vector space of weight $2$ modular forms with arbitrary level (in particular, this includes modular forms for non-congruence subgroups). Recall that there is a right action of $\GL_2^+(\mathbf{Q})$ on $M_2$, given by $(f \mid g)(z) = \det(g)\cdot (cz+d)^{-2}\cdot f(\frac{az+b}{cz+d})$ where $g = \begin{pmatrix} a & b \\ c & d \end{pmatrix}$ and $f \in M_2$. For any $f \in M_2$, we denote by $a_0(f)$ the Fourier coefficient at infinity of $f$. 

Fix $z_0$ in the upper-half plane. Following Stevens \cite[Definition 2.3.1 p. 51]{Stevens_book}, we define a map $$\mathcal{S} : \GL_2^+(\mathbf{Q}) \rightarrow \Hom_{\mathbf{C}}(M_2, \mathbf{C})$$
by
$$\mathcal{S}(g)(f) = 2i\pi \cdot \int_{z_0}^{g(z_0)} f(z)dz - 2i\pi\cdot  z_0\cdot (a_0(f\mid g)-a_0(f)) + 2i\pi\cdot \int_{z_0}^{i\infty} \left((f\mid g)(z)-a_0(f\mid g)\right) - \left(f(z) - a_0(f)\right) dz \text{ .}$$

\begin{prop}[Stevens]\label{Stevens_trivial_lemma}
\begin{enumerate}
\item \label{Stevens_trivial_lemma_1} The map $\mathcal{S}$ does not depend on the choice of $z_0$.
\item \label{Stevens_trivial_lemma_2} If $g \in \Gamma$ and $f$ is modular of level $\Gamma$, then 
$$\mathcal{S}(g)(f) = 2i\pi\int_{z_0}^{g(z_0)} f(z)dz \text{ .}$$
\item The map $\mathcal{S}$ is a cocycle, \ie for all $(g,g') \in \GL_2^+(\mathbf{Q})^2$ and $f \in M_2$, we have
$$\mathcal{S}(gg')(f) = \mathcal{S}(g)(f) + \mathcal{S}(g')(f\mid g) \text{ .}$$
\item \label{Stevens_trivial_lemma_3} For all $f \in M_2$, we have $\mathcal{S}\begin{pmatrix} a & b \\ 0 & d \end{pmatrix}(f) = \frac{b}{d}\cdot 2i\pi\cdot a_0(f)$.
\item \label{Stevens_trivial_lemma_4} For all $f \in M_2$, we have $\mathcal{S}\begin{pmatrix} 0 & -1 \\ 1 & 0 \end{pmatrix}(f) = L(f,1)$.
\end{enumerate}
\end{prop}
\begin{proof}
This follows from \cite[p. 51-52]{Stevens_book}.
\end{proof}
Using Proposition \ref{Stevens_trivial_lemma}, we check that there exists a unique group homomorphism
$$\tilde{p}_{\Gamma} : \tilde{\mathcal{M}}_{\Gamma} \rightarrow \Hom_{\mathbf{C}}(M_2(\Gamma), \mathbf{C}) \text{ ,}$$
such that $\tilde{p}_{\Gamma}(\{g,g'\})(f)= \mathcal{S}(g^{-1}g')(f\mid g) = \mathcal{S}(g')(f)-\mathcal{S}(g)(f)$ for all $(g,g') \in \SL_2(\mathbf{Z})^2$ and $f \in M_2(\Gamma)$.

If $c \in \mathcal{C}_{\Gamma}$ is a cusp, let $e_c$ be the ramification index at the cusp $c$ of the map $X_{\Gamma}\rightarrow X_{\SL_2(\mathbf{Z})}$. Let $N$ be the l.c.m of the indices $e_c$; following \cite{general_level} we call $N$ the \textit{general level of $\Gamma$}. If $\Gamma$ is a congruence subgroup of general level $N$, then $\Gamma$ has (usual) level $N$ or $2N$ \cite[Proposition 3]{general_level}.

\begin{prop}\label{rank_computation}
\begin{enumerate}
\item \label{rank_computation_1} The map $\iota_{\Gamma} : H_1(Y_{\Gamma}, \mathbf{Z}) \rightarrow \tilde{\mathcal{M}}_{\Gamma}$ is injective. Thus, we can (and do) indentify $H_1(Y_{\Gamma}, \mathbf{Z})$ with a subgroup of $\tilde{\mathcal{M}}_{\Gamma}$.
\item \label{rank_computation_2} The element $\{g,gT\}$ of $\tilde{\mathcal{M}}_{\Gamma}$ only depends of the cusp $c=\Gamma g\infty$ of $X_{\Gamma}$, and $e_c \cdot \{g,gT\}$ is the image by $\iota_{\Gamma}$ of a small oriented circle around the cusp $c$. In particular, we have in $\tilde{\mathcal{M}}_{\Gamma}$:
$$\sum_{c=\Gamma g\infty \in \mathcal{C}_{\Gamma}} e_c \cdot \{g, gT\}=0 \text{ .}$$
\item \label{rank_computation_3} Consider the boundary map $\partial_{\Gamma} : \tilde{\mathcal{M}}_{\Gamma} \rightarrow \mathbf{Z}[\mathcal{C}_{\Gamma}]^0$ given by $\partial_{\Gamma}(\{g,g'\}) = [\Gamma g'\infty]-[\Gamma g\infty]$. The kernel of $\partial_{\Gamma}$ is spanned by $H_1(Y_{\Gamma}, \mathbf{Z})$ and the elements $\{g,gT\}$ for $g \in \SL_2(\mathbf{Z})$. In particular, the torsion subgroup of the coimage of $\iota_{\Gamma}$ has exponent $N$, and we have an exact sequence
\begin{equation}\label{exact_sequence_1}
0 \rightarrow H_1(Y_{\Gamma}, \mathbf{Z}) \otimes_{\mathbf{Z}} \mathbf{Z}[\frac{1}{N}] \xrightarrow{\iota_{\Gamma} \otimes 1} \tilde{\mathcal{M}}_{\Gamma} \otimes_{\mathbf{Z}} \mathbf{Z}[\frac{1}{N}] \xrightarrow{\partial_{\Gamma} \otimes 1} \mathbf{Z}[\mathcal{C}_{\Gamma}]^0 \otimes_{\mathbf{Z}} \mathbf{Z}[\frac{1}{N}]\rightarrow 0
\end{equation}
\item \label{rank_computation_4} The kernel of $\pi_{\Gamma} : \tilde{\mathcal{M}}_{\Gamma} \rightarrow \mathcal{M}_{\Gamma}$ is spanned by the elements $\{g,gT\}$ for $g \in \SL_2(\mathbf{Z})$. Thus, we have an exact sequence
\begin{equation}\label{exact_sequence_2}
0 \rightarrow \mathbf{Z} \rightarrow \mathbf{Z}[\mathcal{C}_{\Gamma}]\rightarrow \tilde{\mathcal{M}}_{\Gamma} \xrightarrow{\pi_{\Gamma}} H_1(X_{\Gamma}, \mathcal{C}_{\Gamma}, \mathbf{Z}) \rightarrow 0 \text{ .}
\end{equation}
Here, the map $\mathbf{Z} \rightarrow \mathbf{Z}[\mathcal{C}_{\Gamma}]$ is given by $1 \mapsto \frac{1}{d_{\Gamma}}\sum_{c \in \mathcal{C}_{\Gamma}} e_c\cdot [c]$, and the map $\mathbf{Z}[\mathcal{C}_{\Gamma}] \rightarrow \tilde{\mathcal{M}}_{\Gamma}$ is given by $[c] \mapsto \{g,gT\}$ where $g\in SL_2(\mathbf{Z})$ is such that $c=\Gamma g\infty$.
\end{enumerate}
\end{prop}
\begin{proof}
Proof of (\ref{rank_computation_1}). It suffices to show that the map $\tilde{p}_{\Gamma} \circ \iota_{\Gamma} : H_1(Y_{\Gamma}, \mathbf{Z}) \rightarrow  \Hom_{\mathbf{C}}(M_2, \mathbf{C})$ is injective. This follows from Lemma \ref{Stevens_trivial_lemma} (\ref{Stevens_trivial_lemma_2}) and the fact that the period map is injective.

Proof of (\ref{rank_computation_2}). Let $(g,g') \in \SL_2(\mathbf{Z})^2$ such that $\Gamma g\infty = \Gamma g'\infty$. There exists $\gamma\in \Gamma$ and $n \in \mathbf{Z}$ such that $g' = \pm \gamma gT^n$. We have, in $\tilde{\mathcal{M}}_{\Gamma}$: 
\begin{align*}
\{g',g'T\}&=\{\gamma gT^n, \gamma g T^{n+1}\} \\&= \{gT^n, gT^{n+1}\} \\& = \{gT^n, g\}+\{g, gT^{n+1}\} \\&=\{g, gT^{n+1}\}-\{g,gT^n\} \\& = (n+1)\cdot \{g,gT\}-n\cdot \{g,gT\} \\&= \{g,gT\} \text{ .}  
\end{align*}
This proves the first claim. For the second claim, note that the image by $\iota_{\Gamma}$ of a small oriented circle around $c$ is $i_{\Gamma}(gT^{e_c}g^{-1}) = \{1, gT^{e_c}g^{-1}\}$. Furthermore, we have in $\tilde{\mathcal{M}}_{\Gamma}$:
\begin{align*}
e_c\cdot \{g,gT\} &= \{g,gT^{e_c}\} \\& = \{g,\gamma g\} \\& = \{g,1\}+\{1,\gamma\}+ \{\gamma, \gamma g\} \\& =\{g,1\}+\{1,\gamma\}+\{1,g\} \\&=\{1,\gamma\} \\&= \{1, gT^{e_c}g^{-1}\}
\end{align*}
where $\gamma=gT^{e_c}g^{-1} \in \Gamma$. This proves the second claim. The equality $$\sum_{c=\Gamma g\infty \in \mathcal{C}_{\Gamma}} e_c \cdot \{g, gT\}=0 $$
is known to be true in $H_1(Y_{\Gamma}, \mathbf{Z})$, and hence in $\tilde{\mathcal{M}}_{\Gamma}$.

Proof of (\ref{rank_computation_3}). It is clear that $\Ker(\partial_{\Gamma})$ contains $H_1(Y_{\Gamma}, \mathbf{Z})$ and $\{g,gT\}$ for all $g \in \SL_2(\mathbf{Z})$. Note that $\tilde{\mathcal{M}}_{\Gamma}$ is spanned by elements of the form $\{1,g\}$ for $g \in \SL_2(\mathbf{Z})$. Let $\sum_{g \in \SL_2(\mathbf{Z})} \lambda_g \cdot \{1,g\}$ be an element of $\Ker(\partial_{\Gamma})$. For each cusp $c \in \mathcal{C}_{\Gamma}$, fix a $g_c \in \SL_2(\mathbf{Z})$ such that $c = \Gamma g_c \infty$. For each $c \neq \Gamma \infty$, we have $\sum_{g \in \SL_2(\mathbf{Z}), c=\Gamma g\infty} \lambda_g=0$. If $c = \Gamma g\infty$, then we have $g = \pm \gamma g_c T^n$ for some $\gamma \in \Gamma$ and $n \in \mathbf{Z}$. Thus $\{1,g\}=\{1,\gamma\}+\{\gamma, \gamma g_c T^n\}=\{1,\gamma\}+\{1,g_c\}+n\cdot \{g_c, g_c T\}$. Since $\{1,\gamma\} \in H_1(Y_{\Gamma}, \mathbf{Z})$ and $\sum_{g \in \SL_2(\mathbf{Z}), c=\Gamma g\infty} \lambda_g=0$, the element $\sum_{g \in \SL_2(\mathbf{Z}), c=\Gamma g\infty} \lambda_g \cdot \{1,g\}$ is in the span of $H_1(Y_{\Gamma}, \mathbf{Z})$ and the elements $\{g,gT\}$.

Proof of (\ref{rank_computation_4}). By (\ref{rank_computation_2}) and (\ref{rank_computation_3}), the kernel of $\pi_{\Gamma}$ is equal to the image of $\mathbf{Z}[\mathcal{C}_{\Gamma}]$ in $\tilde{\mathcal{M}}_{\Gamma}$. Furthermore, the exact sequence (\ref{exact_sequence_1}) shows that the kernel of $\mathbf{Z}[\mathcal{C}_{\Gamma}]\rightarrow \tilde{\mathcal{M}}_{\Gamma}$  must be free of rank one over $\mathbf{Z}$. Since the kernel of $\mathbf{Z}[\mathcal{C}_{\Gamma}]\rightarrow \tilde{\mathcal{M}}_{\Gamma}$ contains $\sum_{c  \in \mathcal{C}_{\Gamma}} e_c\cdot [c]$ by (\ref{rank_computation_2}) and $\tilde{\mathcal{M}}_{\Gamma}$ is torsion-free, this kernel must be spanned by $\frac{1}{d_{\Gamma}}\sum_{c  \in \mathcal{C}_{\Gamma}} e_c\cdot [c]$.
\end{proof}

\subsection{Image of the period map $\tilde{p}_{\Gamma} \otimes \mathbf{R}$}\label{section_period_iso}
By Proposition \ref{rank_computation}, $\tilde{\mathcal{M}}_{\Gamma}$ is a free $\mathbf{Z}$-module of rank $\dim_{\mathbf{R}} M_2(\Gamma) = 2\cdot g(\Gamma) + 2 \cdot (c(\Gamma)-1)$. It thus makes sense to ask whether the map $\tilde{p}_{\Gamma}\otimes \mathbf{R} : \tilde{\mathcal{M}}_{\Gamma} \otimes_{\mathbf{Z}} \mathbf{R} \rightarrow \Hom_{\mathbf{C}}(M_2(\Gamma), \mathbf{C})$ is an isomorphism (or equivalently a surjective map) of $\mathbf{R}$-vector spaces. Note that if $\Gamma'$ is a subgroup of $\Gamma$, then we have a commutative diagram of $\mathbf{R}$-vector spaces
\[\begin{tikzcd}
\tilde{\mathcal{M}}_{\Gamma'} \otimes_{\mathbf{Z}} \mathbf{R} \arrow{r}{\tilde{p}_{\Gamma'} \otimes \mathbf{R}} \arrow[swap]{d}{} & \Hom_{\mathbf{C}}(M_2(\Gamma'), \mathbf{C}) \arrow{d}{} \\
\tilde{\mathcal{M}}_{\Gamma} \otimes_{\mathbf{Z}} \mathbf{R} \arrow{r}{\tilde{p}_{\Gamma} \otimes \mathbf{R}} &  \Hom_{\mathbf{C}}(M_2(\Gamma), \mathbf{C}) 
\end{tikzcd}
\]
where the two vertical maps are the canonical ones, and are surjective. In particular, if $\tilde{p}_{\Gamma'} \otimes \mathbf{R}$ is surjective, then $\tilde{p}_{\Gamma} \otimes \mathbf{R}$ is also surjective. Thus, if $\Gamma$ is a congruence subgroup containing the principal congruence subgroup $\Gamma(N)$, the map $\tilde{p}_{\Gamma} \otimes \mathbf{R}$ is an isomorphism if the map $\tilde{p}_{\Gamma(N)} \otimes \mathbf{R}$ is surjective. During the rest of this paragraph, we assume that $\Gamma$ is a congruence subgroup.

We have a decomposition of $\mathbf{C}$-vector spaces $M_2(\Gamma) = S_2(\Gamma) \oplus \mathcal{E}_2(\Gamma)$, where $S_2(\Gamma)$ is the subspace of cuspidal modular forms and $\mathcal{E}_2(\Gamma)$ is the subspace of Eisenstein series. Let $\rho_{\Cusp, \Gamma} : \Hom_{\mathbf{C}}(M_2(\Gamma), \mathbf{C}) \rightarrow \Hom_{\mathbf{C}}(S_2(\Gamma), \mathbf{C})$ and $\rho_{\Eis, \Gamma} : \Hom_{\mathbf{C}}(M_2(\Gamma), \mathbf{C}) \rightarrow \Hom_{\mathbf{C}}(\mathcal{E}_2(\Gamma), \mathbf{C})$ be the maps induced by the inclusions $S_2(\Gamma) \hookrightarrow M_2(\Gamma)$ and $\mathcal{E}_2(\Gamma) \hookrightarrow M_2(\Gamma)$ respectively.

\begin{lem}\label{restriction_Eisenstein_surjective}
The map $\rho_{\Cusp, \Gamma}\circ (\tilde{p}_{\Gamma}\otimes \mathbf{R})$ is surjective. Thus, $\tilde{p}_{\Gamma} \otimes \mathbf{R}$ is an isomorphism if and only if $\rho_{\Eis, \Gamma} \circ (\tilde{p}_{\Gamma}\otimes \mathbf{R})$ is surjective. 
\end{lem}
\begin{proof}
By Proposition \ref{rank_computation} (\ref{rank_computation_4}), $\rho_{\Cusp, \Gamma}\circ (\tilde{p}_{\Gamma}\otimes \mathbf{R})$ factors through $\pi_{\Gamma} \otimes_{\mathbf{Z}} \mathbf{R} : \tilde{M}_{\Gamma} \otimes_{\mathbf{Z}} \mathbf{R} \rightarrow H_1(X_{\Gamma}, \mathcal{C}_{\Gamma}, \mathbf{Z}) \otimes_{\mathbf{Z}} \mathbf{R}$. The map $H_1(X_{\Gamma}, \mathcal{C}_{\Gamma}, \mathbf{Z}) \otimes_{\mathbf{Z}} \mathbf{R} \rightarrow \Hom_{\mathbf{C}}(S_2(\Gamma), \mathbf{C})$ is surjective, since its restriction to $H_1(X_{\Gamma}, \mathbf{Z}) \otimes_{\mathbf{Z}} \mathbf{R}$ is known to be an isomorphism.
\end{proof}

We thus need to understand the image of $\rho_{\Eis, \Gamma} \circ (\tilde{p}_{\Gamma}\otimes \mathbf{R})$. We first recall some facts about Eisenstein series. We refer the reader to \cite[\S 1]{Stevens_TAMS} for details. The $\mathbf{C}$-vector space $\mathcal{E}_2$ is spanned by those $f \in M_2(\Gamma)$ such that $f(z)dz$ induces a meromorphic differential form on $X_{\Gamma}$ with integer residues at the cusps -- such a differential is called a \textit{differential of the third kind}. An example of differential of the third kind is the logarithmic derivative $d\log(u)$ of a \textit{modular unit} $u$, \ie a meromorphic function on $X_{\Gamma}$ whose divisor is supported on $\mathcal{C}_{\Gamma}$. The set of $f \in \mathcal{E}_2(\Gamma)$ such that $f(z)dz = d\log(u)$ for some modular unit $u$ is denoted by $\mathcal{E}_2(\Gamma, \mathbf{Z})$; this is a $\mathbf{Z}$-module of rank $\dim_{\mathbf{C}} \mathcal{E}_2(\Gamma)$ by the Manin-Drinfeld theorem. If $M$ is a subgroup of $\mathbf{C}$, we let $\mathcal{E}_2(\Gamma, M)$ be the subgroup of $\mathcal{E}_2(\Gamma)$ generated by the elements $\lambda\cdot E$ where $\lambda \in M$ and $E \in \mathcal{E}_2(\Gamma, \mathbf{Z})$. We have $\mathcal{E}_2(\Gamma) = \mathcal{E}_2(\Gamma, \mathbf{C})$. 

If $\Gamma = \Gamma(N)$ for some $N \geq 1$, Stevens gave a spanning family for the $\mathbf{Q}$-vector space $\mathcal{E}_2(\Gamma, \mathbf{Q})$, together with the set of all possible relations. We refer to \cite[\S 2.4]{Stevens_book} for details. If $(x,y) \in (\mathbf{Q}/\mathbf{Z})^2 \backslash \{(0,0)\}$ has order dividing $N$, then Stevens defined an Eisenstein series $\phi_{(x,y)} \in \mathcal{E}_2(\Gamma(N))$, whose $q$-expansion at the cusp $\Gamma(N)\cdot \infty$ is
$$\phi_{(x,y)}(z)=\frac{1}{2}\cdot\B_2(x) - \sum_{ k \equiv x \text{(modulo }1\text{)} \atop k \in \mathbf{Q}_{>0} } k \cdot \sum_{m=1}^{\infty} q(m(kz+y)) - \sum_{ k \equiv -x \text{(modulo }1\text{)} \atop k \in \mathbf{Q}_{>0}} k \cdot \sum_{m=1}^{\infty} q(m(kz-y))$$
where $q(z) = e^{2i\pi z}$ and $\B_2(x) = (x-E(x))^2-(x-E(x))+\frac{1}{6}$ is the second periodic Bernoulli polynomial function. Note that $\phi_{(-x,-y)} = \phi_{(x,y)}$.
We have $2i\pi \phi_{(x,y)}(z)dz = d\log( g_{(x,y)})$ where $g_{(x,y)}$ is a Siegel unit defined by Kubert-Lang in \cite[Chapter 2 \S1 Formula K$4$]{Kubert_Lang}. By \cite[Chapter 2, Theorem 1.2]{Kubert_Lang}, the function $g_{(x,y)}^{12N}$ is modular of level $\Gamma(N)$. Thus we have $2i\pi \phi_{(x,y)} \in \frac{1}{12N} \cdot \mathcal{E}_2(\Gamma, \mathbf{Z})$, and in particular $2i\pi \phi_{(x,y)} \in \mathcal{E}_2(\Gamma, \mathbf{Q})$. By \cite[Chapter 2, \S 2.4]{Stevens_book}, the Eisenstein series $2i\pi \phi_{(x,y)}$ span $\mathcal{E}_2(\Gamma, \mathbf{Q})$, and the linear relations between them are the so called \textit{distribution relations} \cite[Chapter 2, Remark 2.4.4]{Stevens_book}. 

\begin{thm}\label{thm_iso_tilde_p}
If $\Gamma(p^n) \subset \Gamma$ for some prime $p \geq 2$ and some integer $n \geq 1$, then $\tilde{p}_{\Gamma} \otimes \mathbf{R}$ is an isomorphism. 
\end{thm}
\begin{proof}
Let $\mathcal{E} = \mathcal{E}_2(\Gamma, \mathbf{R}) \subset \mathcal{E}_2(\Gamma)$ and $\mathcal{E}' = \mathcal{E}_2(\Gamma, i\cdot \mathbf{R}) = i \cdot \mathcal{E}\subset \mathcal{E}_2(\Gamma)$. The $\mathbf{R}$-vector space $\mathcal{E}_2(\Gamma)$ is the direct sums its two $\mathbf{R}$-vector subspaces $\mathcal{E}$ and $\mathcal{E}'$. We get a canonical isomorphism of $\mathbf{R}$-vector spaces 
$$\Hom_{\mathbf{C}}(\mathcal{E}_2(\Gamma(p^n)), \mathbf{C}) \xrightarrow{\sim} \Hom_{\mathbf{R}}(\mathcal{E}, \mathbf{R}) \times \Hom_{\mathbf{R}}(\mathcal{E}', \mathbf{R})$$ given by taking restrictions and real parts. We denote by $p : \Hom_{\mathbf{C}}(\mathcal{E}_2(\Gamma(p^n)), \mathbf{C}) \rightarrow \Hom_{\mathbf{R}}(\mathcal{E}, \mathbf{R}) $ and $p' : \Hom_{\mathbf{C}}(\mathcal{E}_2(\Gamma(p^n)), \mathbf{C}) \rightarrow \Hom_{\mathbf{R}}(\mathcal{E}', \mathbf{R})$ the associated projections. By Lemma \ref{restriction_Eisenstein_surjective}, $\tilde{p}_{\Gamma} \otimes \mathbf{R}$ is an isomorphism if and only if $p\circ (\rho_{\Eis, \Gamma} \circ (\tilde{p}_{\Gamma} \otimes \mathbf{R})) : \tilde{\mathcal{M}}_{\Gamma} \otimes_{\mathbf{Z}} \mathbf{R} \rightarrow  \Hom_{\mathbf{R}}(\mathcal{E}, \mathbf{R}) $ and $p'\circ (\rho_{\Eis, \Gamma} \circ (\tilde{p}_{\Gamma} \otimes \mathbf{R})):  \tilde{\mathcal{M}}_{\Gamma} \otimes_{\mathbf{Z}} \mathbf{R} \rightarrow  \Hom_{\mathbf{R}}(\mathcal{E}', \mathbf{R}) $ are surjective $\mathbf{R}$-linear maps. To conclude the proof of Theorem \ref{thm_iso_tilde_p}, it thus suffices to prove the following two lemmas.

\begin{lem}\label{surjectivity_part1}
The map $p'\circ (\rho_{\Eis, \Gamma} \circ (\tilde{p}_{\Gamma} \otimes \mathbf{R}))$ is surjective (we need not assume anything on $\Gamma$, except that it is a congruence subgroup).
\end{lem}
\begin{proof}
We have a $\mathbf{R}$-linear isomorphism $\mathcal{E}' \xrightarrow{\sim} \Div^0(\mathcal{C}_{\Gamma} , \mathbf{R})$ given by $E \mapsto \sum_{c \in \mathcal{C}_{\Gamma}} 2i\pi \cdot \Res_c(E)\cdot (c)$, where $\Res_c(E)$ is the residue at $c$ of the differential form on $X_{\Gamma}$ induced by $E(z)dz$. Here, we have denoted by $\Div^0(\mathcal{C}_{\Gamma} , \mathbf{R})$ the group of degree zero divisors with coefficients in $\mathbf{R}$ supported on the set $\mathcal{C}_{\Gamma}$. We thus get a $\mathbf{R}$-linear isomorphism $f : \Hom_{\mathbf{R}}(\mathcal{E}', \mathbf{R}) \xrightarrow{\sim} \Hom_{\mathbf{R}}(\Div^0(\mathcal{C}_{\Gamma} , \mathbf{R}), \mathbf{R})$. We need to prove that the map $f\circ p'\circ (\rho_{\Eis, \Gamma} \circ (\tilde{p}_{\Gamma} \otimes \mathbf{R})) : \tilde{\mathcal{M}}_{\Gamma} \otimes_{\mathbf{Z}} \mathbf{R} \rightarrow \Hom_{\mathbf{R}}(\Div^0(\mathcal{C}_{\Gamma} , \mathbf{R}), \mathbf{R})$ is surjective. The image of a little oriented circle around the cusp $c$ by the latter map is the restriction to $\Div^0(\mathcal{C}_{\Gamma})$ of the element of $\Hom_{\mathbf{R}}(\Div(\mathcal{C}_{\Gamma} , \mathbf{R}), \mathbf{R})$ sending $[c']$ to $0$ if $c'\neq c$ and $[c]$ to $1$. This concludes the proof of Lemma \ref{surjectivity_part1}.
\end{proof}

\begin{lem}\label{surjectivity_part2}
The map $p\circ (\rho_{\Eis, \Gamma} \circ (\tilde{p}_{\Gamma} \otimes \mathbf{R}))$ is surjective.
\end{lem}
\begin{proof}
For notational simplicity, denote by $\varphi$ the map $p\circ (\rho_{\Eis, \Gamma} \circ (\tilde{p}_{\Gamma} \otimes \mathbf{R}))$.
By the discussion at the beginning of \S \ref{section_period_iso}, we can assume without loss of generality that $\Gamma = \Gamma(p^n)$. The distribution relations show that a spanning family of the $\mathbf{R}$-vector space $\mathcal{E}$ is given by the Eisenstein series $2i\pi \phi_{(\frac{a}{p^n},\frac{b}{p^n})}$ where $(a,b) \in (\mathbf{Z}/p^n\mathbf{Z})^2/\pm 1$ is such that $\gcd(a,b,p)=1$. For simplicity, we write $2i\pi \phi_{(a,b)}$ for $2i\pi \phi_{(\frac{a}{p^n},\frac{b}{p^n})}$. The set of such $(a,b)$ is denoted by $S$. The only linear relation between these Eisenstein series is 
$$\sum_{(a,b) \in S} 2i\pi \phi_{(a,b)}=0 \text{ .}$$
There is a right action of $\SL_2(\mathbf{Z}/p^n\mathbf{Z})$ on $S$, given by $g \mapsto (a,b)\cdot g$. Note that $\SL_2(\mathbf{Z}/p^n\mathbf{Z}) \simeq \Gamma(p^n) \backslash \SL_2(\mathbf{Z})$ also acts on $\mathcal{E}$ via the slash operation, and we have \cite[Chapter 2, Remark 2.4.4]{Stevens_book}:
$$(2i\pi \phi_{(a,b)})\mid g = 2i\pi \phi_{(a,b)\cdot g} \text{ .}$$
By \cite[Chapter 2, Proposition 2.5.4 (b)]{Stevens_book}, for all $(g,g') \in \SL_2(\mathbf{Z})$ and $(a,b)\in S$, we have:
\begin{equation}\label{Stevens_formula_evalution_cusps}
\varphi(\{g,g'\} \otimes 1)(2i\pi \phi_{(a,b)}) = F((a,b)\cdot g') - F((a,b)\cdot g)
\end{equation}
where $F : S \rightarrow \mathbf{R}$ is given by
$$F(a,b) = -\delta_{a}\cdot \log \mid 1-e^{\frac{2i\pi b}{p^n}} \mid \text{ .}$$
Here, $\delta_{a}=0$ if $a \neq 0$ and $\delta_{a}=1$ otherwise. Consider the matrix $M$ whose rows are indexed by $S$ and columns are indexed by $\mathcal{C}_{\Gamma}$, and such that the coefficient of $M$ at position $((a,b), \Gamma g\infty)$ is $F((a,b)\cdot g)$. Note that $M$ is a square matrix. We fix an ordering of $S$ and $\mathcal{C}_{\Gamma}$ as follows. Let $B$ be the subgroup of $\SL_2(\mathbf{Z}/p^n\mathbf{Z})$ consisting of upper-triangular matrices. Write $\SL_2(\mathbf{Z}/p^n\mathbf{Z})/B = \bigcup_{i=1}^k g_i B$ for some fixed elements $g_1$, ..., $g_k$. If $x \in (\mathbf{Z}/p^n\mathbf{Z})^{\times}$, let $\gamma_x = \begin{pmatrix} x & 0 \\ 0 & x^{-1} \end{pmatrix} \in B$. We then write $$\mathcal{C}_{\Gamma} = \bigsqcup_{i=1}^{k} \{g_i\gamma_x\infty, x \in (\mathbf{Z}/p^n\mathbf{Z})^{\times}/\pm1 \} \text{ .}$$ We also write $$S = \bigsqcup_{i=1}^{k} \{(0,1)\gamma_x^{-1}g_i^{-1}, x \in (\mathbf{Z}/p^n\mathbf{Z})^{\times}/\pm1 \} \text{ .}$$ This gives a decomposition of $M$ as a block matrix with $k^2$ blocks of size $\Card((\mathbf{Z}/p^n\mathbf{Z})^{\times}/\pm 1)$. We easily check that the non-diagonal blocks are zero, and the diagonal blocks are all equal to the matrix $M':=(-\log \mid 1 - e^{\frac{2i\pi x^{-1}y}{p^n}} \mid)_{(x,y)\in (\mathbf{Z}/p^n\mathbf{Z})^{\times}/\pm1}$. To conclude the proof of Lemma \ref{surjectivity_part2}, it suffices to prove that the matrices $M'$ and $$M'':=(-\log \mid 1 - e^{\frac{2i\pi x^{-1}y}{p^n}}  \mid + \log \mid 1 - e^{\frac{2i\pi x^{-1}}{p^n}} \mid )_{(x,y)\in (\mathbf{Z}/p^n\mathbf{Z})^{\times}/\pm1\atop x,y \neq \pm1}$$ are invertible. By \cite[Lemma 5.26 (a), (b)]{Washington}, we have
$$\det(M') = \prod_{\chi} \sum_{x \in  (\mathbf{Z}/p^n\mathbf{Z})^{\times}/\pm1} -\chi(x)\cdot \log \mid 1 - e^{\frac{2i\pi {x^{-1}}}{p^n}}\mid $$
$$(\text{resp. } \det(M'') = \prod_{\chi \neq 1} \sum_{x \in  (\mathbf{Z}/p^n\mathbf{Z})^{\times}/\pm1} -\chi(x)\cdot \log \mid 1 - e^{\frac{2i\pi {x^{-1}}}{p^n}}\mid \text{ )}$$
where $\chi$ goes through the (resp. non-trivial) even Dirichlet characters of level $p^n$. By the well-known formula for $L$ functions of primitive even Dirichlet characters \cite[Theorem 4.9]{Washington}, we get:
$$\det(M') = \frac{p}{2} \cdot \prod_{\chi \neq 1}  \frac{f_{\chi}}{2\cdot \tau(\chi)} \cdot L(\chi,1)$$
and
$$\det(M'') = \prod_{\chi \neq 1} \frac{f_{\chi}}{2\cdot \tau(\chi)}\cdot  L(\chi,1) $$
where $f_{\chi}$ is the conductor of $\chi$ and $\tau(\chi)$ is the Gauss sum attached to the primitive Dirichlet character associated to $\chi$.
These two quantities are known to be non-zero \cite[Corollary 4.4]{Washington}.
\end{proof}
\end{proof}
\begin{rem}\label{rem_period_non_iso}
In general, $\tilde{p}_{\Gamma} \otimes \mathbf{R}$ is not an isomorphism, as one can check numerically for instance if $\Gamma = \Gamma(6)$ using the method of the proof of Theorem \ref{thm_iso_tilde_p}. Consequently, for any integer $N \geq 1$, $\tilde{p}_{\Gamma} \otimes \mathbf{R}$ is not an isomorphism if $\Gamma = \Gamma(6N)$.
\end{rem}

\subsection{Hecke operators and the complex conjugation}

\textit{In the rest of this paragraph, we assume that $\Gamma = \Gamma_1(N)$ or $\Gamma=\Gamma_0(N)$ for some integer $N \geq 1$. }

Let $\mathbb{T}$ be the Hecke algebra acting faithfully on $M_2(\Gamma)$, generated by the Hecke operator $T_n$ for $n \geq 1$ and by the diamond operators. The abelian groups $\mathcal{M}_{\Gamma} \simeq H_1(X_{\Gamma}, \mathcal{C}_{\Gamma}, \mathbf{Z})$ and $H_1(Y_{\Gamma}, \mathbf{Z})$ both carry a faithfull action of $\mathbb{T}$ and of the complex conjugation (\cf for instance \cite{Merel_Universal}). The goal of this paragraph is to define a natural action of $\mathbb{T}$ and of the complex conjugation on $\tilde{\mathcal{M}}_{\Gamma}$.

\subsubsection{The complex conjugation}\label{paragraph_complex_conjugation}
If $M$ is an abelian group equipped with an action of an involution $m \mapsto \overline{m}$, we denote by $M(1)$ the abelian group $M$ equipped with the involution $m \mapsto -\overline{m}$. We also let $M^+ = \{m\in M, \overline{m}=m\}$ and $M^-=\{m\in M, \overline{m}=-m\}$. 

The action of the complex conjugation (denoted by a bar) on $\mathcal{M}_{\Gamma}$ and $H_1(Y_{\Gamma}, \mathbf{Z})$ is induced by the map $z \mapsto -\bar{z}$ in the upper-half plane. Thus, we have $\overline{\{\alpha, \beta\}} = \{-\alpha, -\beta\}$ in $\mathcal{M}_{\Gamma}$ for all $(\alpha, \beta) \in \mathbf{P}^1(\mathbf{Q})^2$.  If $g=\begin{pmatrix} a & b \\ c & d \end{pmatrix} \in \GL_2^+(\mathbf{Q})$, we let $\overline{g} = \begin{pmatrix} a & -b \\ -c & d \end{pmatrix} \in \GL_2^+(\mathbf{Q})$. Note that we have $\overline{g\cdot g'} = \overline{g}\cdot \overline{g'}$ if $(g,g') \in \GL_2^+(\mathbf{Q})$. The complex conjugation acts on $\Gamma$ as via $g \mapsto \overline{g}$, and this induces an action on $H_1(Y_{\Gamma}, \mathbf{Z})$ via $\Pi_{\Gamma} : \Gamma \rightarrow H_1(Y_{\Gamma}, \mathbf{Z})$. 

There are canonical exact sequences of abelian groups, which are equivariant for the action of complex conjugation:
\begin{equation}\label{boundary_exact_seq_1}
0 \rightarrow H_1(X_{\Gamma}, \mathbf{Z}) \rightarrow \mathcal{M}_{\Gamma} \rightarrow \mathbf{Z}[\mathcal{C}_{\Gamma}]^0 \rightarrow 0
\end{equation}
and
\begin{equation}\label{boundary_exact_seq_2}
0 \rightarrow \mathbf{Z}(1) \rightarrow \mathbf{Z}[\mathcal{C}_{\Gamma}](1) \rightarrow H_1(Y_{\Gamma}, \mathbf{Z}) \rightarrow H_1(X_{\Gamma}, \mathbf{Z}) \rightarrow 0 \text{ .}
\end{equation}
Here, the action of the complex conjugation on $\mathbf{Z}[\mathcal{C}_{\Gamma}]$ and $\mathbf{Z}[\mathcal{C}_{\Gamma}]^0$ is induced by the natural action on $\mathcal{C}_{\Gamma}$, and the action on $\mathbf{Z}$ is trivial.

The action of the complex conjugation on $\tilde{\mathcal{M}}_{\Gamma}$ is defined by $\overline{\{g,g'\}}:=\{\overline{g}, \overline{g}'\}$. One easily checks that this is well-defined, and that the exact sequences (\ref{exact_sequence_1}) and (\ref{exact_sequence_2}) are equivariant with respect to the complex conjugation.

There is a natural action of the complex conjugation on $\Hom_{\mathbf{C}}(M_2(\Gamma), \mathbf{C})$, given by $\varphi \mapsto \overline{\varphi}$, where $\overline{\varphi} : f \mapsto \overline{\varphi(\overline{f})}$. Here, $\overline{f} \in M_2(\Gamma)$ is defined by $\overline{f}(z) = \overline{f(-\overline{z})}$. Equivalently, if the $q$-expansion of $f$ at $\Gamma\infty$ is $\sum_{n \geq 0} a_nq^n$, then the $q$-expansion of $\overline{f}$ at $\Gamma \infty$ is $\sum_{n\geq 0} \overline{a}_nq^n$.

\begin{prop}\label{equivariance_complex_conjug_period}
The map $\tilde{p}_{\Gamma} : \tilde{\mathcal{M}}_{\Gamma} \rightarrow \Hom_{\mathbf{C}}(M_2(\Gamma), \mathbf{C})$ commutes with the action of the complex conjugation.
\end{prop}
\begin{proof}
Recall that $\tilde{p}_{\Gamma}(\{g,g'\})(f)=\mathcal{S}(g')(f) - \mathcal{S}(g)(f)$ where
$$\mathcal{S}(g)(f) = 2i\pi \cdot \int_{z_0}^{g(z_0)} f(z)dz - 2i\pi\cdot  z_0\cdot (a_0(f\mid g)-a_0(f)) + 2i\pi\cdot \int_{z_0}^{i\infty} \left((f\mid g)(z)-a_0(f\mid g)\right) - \left(f(z) - a_0(f)\right) dz $$
for any $z_0 \in \mathfrak{h}$. It thus suffices to prove that for any $f \in M_2(\Gamma)$ and $g \in \SL_2(\mathbf{Z})$, we have $\overline{\mathcal{S}(g)(\overline{f})} = \mathcal{S}(\tilde{g})(f)$. We shall make use the following straightforward result.
\begin{lem}\label{trivial_lemma_cplx_conjug}
For any $f \in M_2$, $g \in \SL_2(\mathbf{Z})$ and $z \in \mathfrak{h}$, we have $\overline{f \mid g} = \overline{f} \mid \tilde{g}$ and $-\overline{g(z)} = \tilde{g}(-\overline{z})$.
\end{lem}
We analyse separately the three terms in the definition of $\overline{\mathcal{S}(g)(\overline{f})}$. We have:
\begin{align*}
\overline{2i\pi \cdot \int_{z_0}^{g(z_0)} \overline{f}(z)dz} &= -2i\pi \int_{z_0}^{g(z_0)} \overline{\overline{f(-\overline{z})}dz}
\\&=2i\pi \int_{-\overline{z_0}}^{-\overline{g(z_0)}} f(z)dz \\&=2i\pi \int_{-\overline{z_0}}^{\tilde{g}(-\overline{z_0})} f(z)dz \text{ .}
\end{align*}
We have:
\begin{align*}
\overline{2i\pi\cdot  z_0\cdot (a_0(\overline{f}\mid g)-a_0(\overline{f}))} &= -2i\pi \cdot \overline{z_0}\cdot (a_0(\overline{\overline{f}\mid g})-a_0(f)) \\& = 2i\pi \cdot (-\overline{z}_0)\cdot (a_0(f\mid \tilde{g})-a_0(f)) \text{ .}
\end{align*}
Finally, we have:
\begin{align*}
\overline{2i\pi\cdot \int_{z_0}^{i\infty} \left((\overline{f}\mid g)(z)-a_0(\overline{f}\mid g)\right) - \left(\overline{f}(z) - a_0(\overline{f})\right) dz } &= -2i\pi \cdot  \int_{z_0}^{i\infty} \left((f\mid \tilde{g})(-\overline{z})-a_0(f\mid \tilde{g})\right) - \left(f(-\overline{z}) - a_0(f)\right) d\overline{z} \\&= 2i\pi \cdot  \int_{-\overline{z_0}}^{i\infty} \left((f\mid \tilde{g})(z)-a_0(f\mid \tilde{g})\right) - \left(f(z) - a_0(f)\right) dz \text{ .}
\end{align*}
We have thus proved that $\overline{\mathcal{S}(g)(\overline{f})} = \mathcal{S}(\tilde{g})(f)$.
This concludes the proof of Proposition \ref{equivariance_complex_conjug_period}.
\end{proof}

\subsubsection{Hecke operators}\label{paragraph_Hecke}
We are going to define a natural action of $\mathbb{T}$ on $\tilde{\mathcal{M}}_{\Gamma}$. We first briefly recall how $\mathbb{T}$ acts on $\mathcal{M}_{\Gamma}$ and $H_1(Y_{\Gamma}, \mathbf{Z})$. Let $g \in \GL_2^+(\mathbf{Q})$ and consider the double-coset $\Gamma g\Gamma = \bigsqcup_{i \in I} \Gamma g_i$ for some finite set $I$ and $g_i \in  \GL_2^+(\mathbf{Q})$. The Hecke operator $T_g$ acts $X_{\Gamma}$ via the correspondance $(\Gamma z) \mapsto \sum_{i \in I}(\Gamma g_iz)$, where $z \in \mathfrak{H} \cup \mathbf{P}^1(\mathbf{Q})$ and $\mathfrak{H}$ is the upper-half plane. This does not depend on the choice of the elements $g_i$. If $g = \begin{pmatrix} 1&0\\ 0& p \end{pmatrix}$ for some prime $p$, the Hecke operator $T_{g}$ is denoted by $T_p$ if $p$ does not divide $N$ and $U_p$ otherwise. If $g=\begin{pmatrix} a & b \\ c& d\end{pmatrix} \in \Gamma_0(N)$, the Hecke operator $T_g$ is denoted by $\langle d \rangle$ (a so-called diamond operator).

This induces an action of $T_g$ on $H_1(Y_{\Gamma}, \mathbf{Z})$ and $\mathcal{M}_{\Gamma}$, which is explicitly given as follows. For any $\{\alpha, \beta\} \in \mathcal{M}_{\Gamma}$, we have $T_g\{\alpha,\beta\}=\sum_{i \in I} \{g_i\alpha, g_i\beta\}$. For $H_1(Y_{\Gamma}, \mathbf{Z})$, the action of $T_g$ is similarly given in terms of geodesic paths in $\mathfrak{H}$, but it will be more convenient to describe this action using the map $\Pi_{\Gamma}$. For all $i \in I$ and $\gamma \in \Gamma$, we have $$g_i\gamma = t_{i, \Gamma}(\gamma)g_{\sigma_{\gamma}(i)}$$ for unique $t_{i, \Gamma}(\gamma) \in \Gamma$ and $\sigma_{\gamma}(i) \in I$. Note that $t_{i, \Gamma} : \Gamma \rightarrow \Gamma$ and $\sigma_{\gamma} : I \rightarrow I$ are bijective maps. The following result is well-known. 
\begin{prop}\label{Hecke_group_theoretic}
For all $\gamma \in \Gamma$, we have
$$T_g \Pi_{\Gamma}(\gamma) = \sum_{i \in I} \Pi_{\Gamma}(t_{i, \Gamma}(\gamma)) \text{ .}$$
\end{prop}
\begin{proof}
Fix $z_0 \in \mathfrak{H}$. Let $\alpha : \mathfrak{H} \rightarrow Y_{\Gamma}$ be the quotient map. Let $S = \alpha(\GL_2^+(\mathbf{Q})\cdot z_0)$. We have an inclusion $H_1(Y_{\Gamma}, \mathbf{Z}) \subset H_1(Y_{\Gamma}, S, \mathbf{Z})$ where $H_1(Y_{\Gamma}, S, \mathbf{Z})$ is the relative homology group with respect to the pair $(Y_{\Gamma}, S)$. We will do our computations inside $H_1(Y_{\Gamma}, Z, \mathbf{Z})$. If $z_1,z_2 \in \GL_2^+(\mathbf{Q})\cdot z_0$, let $\{z_1,z_2\}$ be the image in $H_1(Y_{\Gamma}, S, \mathbf{Z})$ of the geodesic path between $z_1$ and $z_2$ in $\mathfrak{H}$. For all $\gamma \in \Gamma$ we have in $H_1(Y_{\Gamma}, S, \mathbf{Z})$:
\begin{align*}
T_g \Pi_{\Gamma}(\gamma) &= \sum_{i\in I} \{g_iz_0, g_i\gamma z_0\} \\&=  \sum_{i\in I} \{g_iz_0, t_{i, \Gamma}(\gamma)g_{\sigma_{\gamma}(i)}z_0\} \\& =  \sum_{i\in I} \{g_iz_0, g_{\sigma_{\gamma}(i)}z_0\} + \{g_{\sigma_{\gamma}(i)}z_0, t_{i, \Gamma}(\gamma)g_{\sigma_{\gamma}(i)}z_0\} \\&=\sum_{i\in I} \Pi_{\Gamma}(t_{i, \Gamma}(\gamma))+  \sum_{i\in I} \{g_iz_0, g_{\sigma_{\gamma}(i)}z_0\} \text{,}
\end{align*}
where in the last equality we have used that $$\{g_{\sigma_{\gamma}(i)}z_0, t_{i, \Gamma}(\gamma)g_{\sigma_{\gamma}(i)}z_0\} = \{z_0, t_{i, \Gamma}(\gamma)z_0\}=\Pi_{\Gamma}(t_{i, \Gamma}(\gamma)) \text{ .}$$
To conclude the proof of Proposition \ref{Hecke_group_theoretic}, we have to prove that $\sum_{i\in I} \{g_iz_0, g_{\sigma_{\gamma}(i)}z_0\}=0$. This follows formally from the fact that $\sigma_{\gamma}$ is a permutation of $I$ and the Chasles relations on symbols $\{., .\}$.
\end{proof}

One difficulty to define an action of $\mathbb{T}$ on $\tilde{\mathcal{M}}_{\Gamma}$ comes from the fact that while $\GL_2^+(\mathbf{Q})$ acts on $\mathfrak{H}\cup \mathbf{P}^1(\mathbf{Q})$, it does not act on $\SL_2(\mathbf{Z})$ so it does not make sense to talk about the symbols $\{g_ig,g_ig'\}$ in $\tilde{\mathcal{M}}_{\Gamma}$ for $(g,g') \in \SL_2(\mathbf{Z})^2$. We are going to solve this issue by introducing a $\mathbf{Q}$-vector space $\tilde{\mathcal{M}}_{\mathbf{Q}, \Gamma}$ containing $\tilde{\mathcal{M}}_{\Gamma}$ as a lattice and for which there is a natural action of double cosets Hecke operators. 

Consider the $\mathbf{Q}$-vector space $\tilde{\mathcal{M}}_{\mathbf{Q}}$ generated by the symbols $\{g,g'\}_{\mathbf{Q}}$ where $(g,g') \in \GL_2^+(\mathbf{Q})^2$ with the following relations:
\begin{enumerate}
\item $\{g,g'\}_{\mathbf{Q}}+\{g',g''\}_{\mathbf{Q}}+\{g'',g\}_{\mathbf{Q}}=0$ for all $(g,g',g'') \in \GL_2^+(\mathbf{Q})^3$;
\item $\{g,g'\}_{\mathbf{Q}}-\{\lambda_1g, \lambda_2 g'\}_{\mathbf{Q}}=0$ for all $(g,g') \in \GL_2^+(\mathbf{Q})^2$ and $(\lambda_1,\lambda_2)\in (\mathbf{Q}^{\times})^2$; 
\item $\{g,g\begin{pmatrix}a&b \\ 0 & d\end{pmatrix}\}_{\mathbf{Q}}-\frac{b}{d}\cdot \{g,gT\}_{\mathbf{Q}}=0$ for all $g\in \GL_2^+(\mathbf{Q})$ and $\begin{pmatrix}a&b \\ 0 & d\end{pmatrix} \in \GL_2^+(\mathbf{Q})$.
\end{enumerate}
There is a left action of $\GL_2^+(\mathbf{Q})$ on $\tilde{\mathcal{M}}_{\mathbf{Q}}$, given by $g \cdot \{g',g''\} = \{gg',gg''\}$ for all $(g,g',g'') \in \GL_2^+(\mathbf{Q})^3$. We let $\tilde{\mathcal{M}}_{\mathbf{Q}, \Gamma}$ be the largest quotient of $\tilde{\mathcal{M}}_{\mathbf{Q}}$ on which $\Gamma$ acts trivially. There is a canonical map $$\psi : \tilde{\mathcal{M}}_{\Gamma} \rightarrow \tilde{\mathcal{M}}_{\mathbf{Q}, \Gamma}$$
given by $\{g,g'\} \mapsto \{g,g'\}_{\mathbf{Q}}$ for $(g,g') \in \SL_2(\mathbf{Z})^2$. By Proposition \ref{Stevens_trivial_lemma}, there is a unique well-defined $\mathbf{Q}$-linear map
$$\tilde{p}_{\mathbf{Q}, \Gamma} : \tilde{\mathcal{M}}_{\mathbf{Q}, \Gamma} \rightarrow \Hom_{\mathbf{C}}(M_2(\Gamma), \mathbf{C})$$
such that $\tilde{p}_{\mathbf{Q}, \Gamma}(\{g,g'\}_{\mathbf{Q}}) = \mathcal{S}(g')(f) - \mathcal{S}(g)(f)$ for all $(g,g') \in \GL_2^+(\mathbf{Q})^2$. We have $\tilde{p}_{\Gamma} = \tilde{p}_{\mathbf{Q}, \Gamma} \circ \psi$.

\begin{prop}\label{lattice_Q}
The map $\psi$ induces an isomorphism of $\mathbf{Q}$-vector spaces $\psi \otimes \mathbf{Q} : \tilde{\mathcal{M}}_{\Gamma} \otimes_{\mathbf{Z}} \mathbf{Q} \xrightarrow{\sim} \tilde{\mathcal{M}}_{\mathbf{Q}, \Gamma}$. Thus, we can consider $\tilde{\mathcal{M}}_{\Gamma}$ as a lattice inside $\tilde{\mathcal{M}}_{\mathbf{Q}, \Gamma}$.
\end{prop}
\begin{proof}
We first show that $\psi \otimes \mathbf{Q}$ is surjective. Let $(g,g') \in (\GL_2(\mathbf{Q})^+)^2$. Write $g = \alpha \begin{pmatrix} a & b \\ 0 & d \end{pmatrix}$ and $g' = \alpha' \begin{pmatrix} a' & b' \\ 0 & d' \end{pmatrix}$ for some $(\alpha, \alpha') \in \SL_2(\mathbf{Z})^2$ and $\begin{pmatrix} a & b \\ 0 & d \end{pmatrix}, \begin{pmatrix} a' & b' \\ 0 & d' \end{pmatrix} \in \GL_2^+(\mathbf{Q})$. We then have, in $\tilde{\mathcal{M}}_{\mathbf{Q}, \Gamma}$:
\begin{align*}
\{g,g'\}_{\mathbf{Q}} &= \{g, \alpha\}_{\mathbf{Q}}+\{\alpha, \alpha'\}_{\mathbf{Q}}+\{\alpha', g'\}_{\mathbf{Q}} \\& = \frac{b}{d}\cdot \{\alpha T, \alpha\}_{\mathbf{Q}}+ \{\alpha, \alpha'\}_{\mathbf{Q}} + \frac{b'}{d'}\cdot \{\alpha', \alpha'T\} _{\mathbf{Q}} \\& =
\frac{b}{d}\cdot (\psi \otimes \mathbf{Q})(\{\alpha T, \alpha\}) + (\psi \otimes \mathbf{Q})(\{\alpha, \alpha'\}) + \frac{b'}{d'}\cdot (\psi \otimes \mathbf{Q})(\{\alpha', \alpha' T\}) \text{ .}
\end{align*}
This proves that $\psi \otimes \mathbf{Q}$ is surjective. To prove that $\psi \otimes \mathbf{Q}$ is an isomorphism, it suffices to prove that $\dim_{\mathbf{Q}} \tilde{\mathcal{M}}_{\mathbf{Q}, \Gamma} \geq \dim_{\mathbf{Q}}(\tilde{\mathcal{M}}_{\Gamma} \otimes_{\mathbf{Z}} \mathbf{Q}) = 2 g(\Gamma) +2\cdot (c(\Gamma)-1)$. There is a surjective $\mathbf{Q}$-linear map $$\tilde{\mathcal{M}}_{\mathbf{Q}, \Gamma} \rightarrow \mathcal{M}_{\Gamma} \otimes_{\mathbf{Z}} \mathbf{Q}\text{ ,}$$ given by $\{g,g'\}_{\mathbf{Q}} \mapsto \{g\infty, g'\infty\}$. Its kernel contains the elements $\{g,gT\}_{\mathbf{Q}}$ for all $g \in \SL_2(\mathbf{Z})$, so it suffices to show that the $\mathbf{Q}$-vector space $C$ spanned by these elements has dimension $\geq c(\Gamma)-1$. The map $(\psi \otimes \mathbf{Q}) \circ (\iota_{\Gamma} \otimes \mathbf{Q}) : H_1(Y_{\Gamma}, \mathbf{Z}) \otimes_{\mathbf{Z}} \mathbf{Q} \rightarrow \tilde{\mathcal{M}}_{\mathbf{Q}, \Gamma}$ is injective, since $$\tilde{p}_{\mathbf{Q}, \Gamma} \circ (\psi \otimes \mathbf{Q}) \circ (\iota_{\Gamma} \otimes \mathbf{Q}) : H_1(Y_{\Gamma}, \mathbf{Z}) \otimes_{\mathbf{Z}} \mathbf{Q} \rightarrow \Hom_{\mathbf{C}}(M_2(\Gamma), \mathbf{C})$$ is injective. Thus, we can identify $H_1(Y_{\Gamma}, \mathbf{Z}) \otimes_{\mathbf{Z}} \mathbf{Q}$ with a $\mathbf{Q}$-vector subspace of $\tilde{\mathcal{M}}_{\mathbf{Q}, \Gamma}$. We know that $C$ is the subspace of $H_1(Y_{\Gamma}, \mathbf{Z}) \otimes_{\mathbf{Z}} \mathbf{Q}$ of dimension $c(\Gamma)-1$ spanned by the little circles around the cusps. This concludes the proof of Proposition \ref{lattice_Q}.
 \end{proof}

Let $g \in \GL_2^+(\mathbf{Q})$, and write as before $\Gamma g \Gamma = \bigsqcup_{i=1}^n \Gamma g_i$. There is a well-defined double coset operator
$$T_{g, \mathbf{Q}} : \tilde{\mathcal{M}}_{\Gamma, \mathbf{Q}} \rightarrow \tilde{\mathcal{M}}_{\Gamma, \mathbf{Q}}$$
given by $\{h,h'\}_{\mathbf{Q}} \mapsto \sum_{i=1}^n \{g_ih, g_ih'\}_{\mathbf{Q}}$. If $p$ is a prime and $g = \begin{pmatrix} 1 & 0 \\ 0 & p \end{pmatrix}$, we denote $T_{g, \mathbf{Q}}$ by $T_{p, \mathbf{Q}}$ if $p \nmid N$ and by $U_{p, \mathbf{Q}}$ otherwise. If $d$ is an integer coprime to $N$ and $g \in \Gamma_0(N)$ is such that its lower-right coefficient is congruent to $d$ modulo $N$, we denote $T_{g, \mathbf{Q}}$ by $\langle d \rangle_{\mathbf{Q}}$. Finally, if $g = \begin{pmatrix} 0 & -1 \\ N & 0 \end{pmatrix}$, then we denote $T_{g, \mathbf{Q}}$ by $W_{N, \mathbf{Q}}$.

\begin{thm}\label{Hecke_Q}
There is a unique action of $\mathbb{T}$ on $\tilde{\mathcal{M}}_{\Gamma, \mathbf{Q}}$ such that the following hold.
\begin{enumerate}
\item\label{Hecke_Q_compatible} If $p$ is a prime dividing (resp. not dividing) $N$, then $U_p$ (resp. $T_p$) acts as $U_{p,\mathbf{Q}}$ (resp. $T_{p,\mathbf{Q}}$).  If $d$ is an integer coprime to $N$, then $\langle d \rangle$ acts as $\langle d \rangle_{\mathbf{Q}}$.
\item\label{Hecke_Q_periods} The map $\tilde{p}_{\mathbf{Q}, \Gamma} : \tilde{\mathcal{M}}_{\mathbf{Q}, \Gamma}\rightarrow \Hom_{\mathbf{C}}(M_2(\Gamma), \mathbf{C})$ is $\mathbb{T}$ and $W_{N, \mathbf{Q}}$-equivariant.
\item\label{Hecke_Q_exact_seq} The exact sequences (\ref{exact_sequence_1}) and (\ref{exact_sequence_2}) tensorized by $\mathbf{Q}$ are $\mathbb{T}$ and $W_{N, \mathbf{Q}}$-equivariant after identifying $\tilde{\mathcal{M}}_{\Gamma} \otimes_{\mathbf{Z}} \mathbf{Q}$ with $\tilde{\mathcal{M}}_{\mathbf{Q}, \Gamma}$ via $\psi \otimes \mathbf{Q}$.
\item\label{Hecke_Q_lattice_odd} For all integer $d$ coprime to $N$, the diamond operator $\langle d \rangle$ stabilizes the lattice $\tilde{\mathcal{M}}_{\Gamma}$. For all prime $p$ not dividing $2N$, $T_p$ stabilizes $\tilde{\mathcal{M}}_{\Gamma}$. More precisely, write $$\SL_2(\mathbf{Z}) \begin{pmatrix} 1 & 0 \\ 0 & p \end{pmatrix} \SL_2(\mathbf{Z}) = \SL_2(\mathbf{Z})g_{\infty} \bigcup_{i=-{\frac{p-1}{2}}}^{\frac{p-1}{2}} \SL_2(\mathbf{Z})g_i$$
where $g_{\infty} = \begin{pmatrix} p & 0 \\ 0 & 1 \end{pmatrix}$ and $g_i = \begin{pmatrix} 1 & i \\ 0 & p \end{pmatrix}$. Then we have in $\tilde{\mathcal{M}}_{\Gamma}$, for all $(g,g') \in \SL_2(\mathbf{Z})^2$:
$$T_p \{g,g'\} = \langle p \rangle \{t_{\infty, SL_2(\mathbf{Z})}(g), t_{\infty, SL_2(\mathbf{Z})}(g')\}+\sum_{i=-\frac{p-1}{2}}^{\frac{p-1}{2}}  \{t_{i, SL_2(\mathbf{Z})}(g), t_{i, SL_2(\mathbf{Z})}(g')\} \text{ .}$$
\item\label{Hecke_Q_T_2} If $p$ divides $2N$, the operator $T_p$ (or $U_p$ if $p \mid N$) sends $\tilde{\mathcal{M}}_{\Gamma}$ into $\tilde{\mathcal{M}}_{\Gamma}+\frac{1}{p}\cdot \sum_{c=\Gamma g\infty \in \mathcal{C}_{\Gamma}} \{g,gT\} \subset \frac{1}{p}\cdot \tilde{\mathcal{M}}_{\Gamma}$.
\item \label{Atkin_Lehner_M} The operator $W_{N, \mathbf{Q}}$ induces an endomorphism of $\tilde{\mathcal{M}}_{\Gamma} \otimes_{\mathbf{Z}} \mathbf{Z}[\frac{1}{N}]$, which we denote by $W_N$ and call the Atkin--Lehner involution.
\item\label{hecke_conjugation_M} Assume that all the cusps of $X_{\Gamma}$ are fixed by the complex conjugation. Then the exact sequences (\ref{exact_sequence_1}) and (\ref{exact_sequence_2}) induce a canonical $\mathbb{T}$ and $W_N$-equivariant isomorphism
$$\tilde{\mathcal{M}}_{\Gamma} \otimes_{\mathbf{Z}} \mathbf{Z}[\frac{1}{2N}] \xrightarrow{\sim} H_1(X_{\Gamma}, \mathcal{C}_{\Gamma}, \mathbf{Z}[\frac{1}{2N}])^+ \bigoplus H_1(Y_{\Gamma}, \mathbf{Z}[\frac{1}{2N}])^- \text{ .}$$
\end{enumerate}
\end{thm}
\begin{rem}
If $\Gamma = \Gamma_1(N)$, the cusps of $X_{\Gamma}$ are fixed by the complex conjugation if and only if $N$ divides $2p$ for some prime number $p \geq 2$. If $\Gamma = \Gamma_0(N)$, the cusps of $X_{\Gamma}$ are fixed by the complex conjugation if and only if $N$ is squarefree or four times a squarefree integer.
\end{rem}
\begin{proof}

\begin{lem}\label{Hecke_equivariance_+_-}
\begin{enumerate}
\item \label{Hecke_equivariance_pi}The map $\pi_{\Gamma} \otimes \mathbf{Q} : \tilde{\mathcal{M}}_{\mathbf{Q},\Gamma} \rightarrow H_1(X_{\Gamma}, \mathcal{C}_{\Gamma}, \mathbf{Q})$ is Hecke equivariant. This means the following. If $p$ is a prime dividing (resp. not dividing) $N$, then $(\pi_{\Gamma} \otimes \mathbf{Q}) \circ U_{p, \mathbf{Q}} = U_p \circ (\pi_{\Gamma} \otimes \mathbf{Q})$ (resp.  $(\pi_{\Gamma} \otimes \mathbf{Q}) \circ T_{p, \mathbf{Q}} = T_p \circ (\pi_{\Gamma} \otimes \mathbf{Q})$). Similarly for the diamond operators and the Atkin-Lehner involution.
\item \label{Hecke_equivariance_iota} The map $\iota_{\Gamma} \otimes \mathbf{Q} : H_1(Y_{\Gamma}, \mathbf{Q}) \rightarrow \tilde{\mathcal{M}}_{\mathbf{Q},\Gamma} $ is Hecke equivariant.
\item \label{Hecke_equivariance_periods} The map $\tilde{p}_{\mathbf{Q}, \Gamma} :\tilde{\mathcal{M}}_{\mathbf{Q}, \Gamma} \rightarrow \Hom_{\mathbf{C}}(M_2(\Gamma), \mathbf{C})$ is Hecke equivariant.
\item\label{Hecke_equivariance_produit_injective} The map $$\tilde{p}_{\mathbf{Q}, \Gamma} \times (\tilde{\pi}_{\Gamma} \otimes \mathbf{Q}) : \tilde{\mathcal{M}}_{\mathbf{Q}, \Gamma} \rightarrow  \Hom_{\mathbf{C}}(M_2(\Gamma), \mathbf{C}) \times (\mathcal{M}_{\Gamma} \otimes_{\mathbf{Z}} \mathbf{Q})$$ is injective and Hecke equivariant.
\end{enumerate}
\end{lem}
\begin{proof}
Point (\ref{Hecke_equivariance_pi}) follows by definition of Hecke operators. We now prove (\ref{Hecke_equivariance_iota}). We first check the compatibility of diamond operators. Let $d$ be an integer coprime to $N$. Let $g$ be a matrix whose lower right corner is congruent to $d$ modulo $N$. For all $\gamma \in \Gamma$, we have in $ \tilde{\mathcal{M}}_{\mathbf{Q}, \Gamma}$:
\begin{align*}
\iota_{\Gamma}(\langle d \rangle\Pi_{\Gamma}(\gamma)) &= \iota_{\Gamma}(\Pi_{\Gamma}(g\gamma g^{-1}))  \\& = \{1, g \gamma g^{-1}\}_{\mathbf{Q}} \\& = \{g,g\gamma\}_{\mathbf{Q}} \text{ .}
\end{align*}
The first equality follows from Proposition \ref{Hecke_group_theoretic}. The second equality is justified as follows:
\begin{align*}
\{g, g\gamma\} &= \{g, g\gamma g^{-1}g\} \\&= \{g,1\}+\{1,g\gamma g^{-1}\}+\{g\gamma g^{-1}, g\gamma g^{-1}g\} \\& =  \{g,1\}+\{1,g\gamma g^{-1}\}+\{1,g\} \\&= \{1,g\gamma g^{-1}\} \text{ ,}
\end{align*}
where we have used the fact that $g\gamma g^{-1}\in \Gamma$.
This proves the compatibility of $\iota_{\Gamma} \otimes \mathbf{Q}$ to diamond operators. We now consider the Hecke operator $T_p$ (or $U_p$) for a prime $p$. Write $\Gamma \begin{pmatrix} 1 & 0 \\ 0 & p \end{pmatrix} \Gamma = \bigcup_{i \in I} \Gamma g_i \Gamma$. By Proposition  \ref{Hecke_group_theoretic}, for all $\gamma \in \Gamma$, we have in $ \tilde{\mathcal{M}}_{\mathbf{Q}, \Gamma}$:
\begin{align*}
(\iota_{\Gamma} \otimes \mathbf{Q})(T_p \Pi_{\Gamma}(\gamma)) &= \sum_{i \in I} \Pi_{\Gamma}(t_{i, \Gamma}(\gamma)) 
\\&= \sum_{i \in I} \{1, t_{i, \Gamma}(\gamma)\}_{\mathbf{Q}} \text{ .}
\end{align*}
On the other hand, we have:
\begin{align*}
T_{p, \mathbf{Q}}(\iota_{\Gamma} \otimes \mathbf{Q})(\gamma)&=\sum_{i \in I} \{g_i, g_i\gamma\}_{\mathbf{Q}} \\& = \sum_{i \in I} \{g_i, t_{i,\Gamma}(\gamma)g_{\sigma_{\gamma}(i)}\}_{\mathbf{Q}} \\& = \sum_{i\in I} \{g_i, g_{\sigma_{\gamma}(i)}\}_{\mathbf{Q}}+\{g_{\sigma_{\gamma}(i)}, t_{i,\Gamma}(\gamma)g_{\sigma_{\gamma}(i)}\}_{\mathbf{Q}} \\& = \sum_{i\in I} \{g_i, g_{\sigma_{\gamma}(i)}\}_{\mathbf{Q}} + \sum_{i\in I} \{1, t_{i,\Gamma}(\gamma)\}_{\mathbf{Q}}
\end{align*}
where in the last equality we have used the fact that $t_{i,\Gamma}(\gamma) \in \Gamma$, so $\{g_{\sigma_{\gamma}(i)}, t_{i,\Gamma}(\gamma)g_{\sigma_{\gamma}(i)}\}_{\mathbf{Q}} = \{1, t_{i,\Gamma}(\gamma)\}_{\mathbf{Q}}$. Since $\sigma_{\gamma} : I \rightarrow I$ is a bijection, we have $\sum_{i\in I} \{g_i, g_{\sigma_{\gamma}(i)}\}_{\mathbf{Q}}=0$. This proves (\ref{Hecke_equivariance_iota}).

Lemma \ref{Hecke_equivariance_+_-} (\ref{Hecke_equivariance_periods}) is an immediate consequence of the definition of Hecke operators as double coset operators, and of the formula $\tilde{p}_{\mathbf{Q}, \Gamma}(\{g,g'\}_{\mathbf{Q}})(f) = \mathcal{S}(g^{-1}g')(f\mid g)$ for all $(g,g') \in \GL_2^+(\mathbf{Q})^2$ and $f \in M_2(\Gamma)$. 

We finally prove Lemma \ref{Hecke_equivariance_+_-} (\ref{Hecke_equivariance_produit_injective}). The map $\tilde{p}_{\mathbf{Q}, \Gamma} \times (\tilde{\pi}_{\Gamma} \otimes \mathbf{Q})$ is Hecke equivariant by (\ref{Hecke_equivariance_pi}) and (\ref{Hecke_equivariance_periods}), so we only need to prove that it is injective. This follows from the facts that the map $\tilde{p}_{\Gamma} \otimes \mathbf{Q} : H_1(Y_{\Gamma}, \mathbf{Z}) \otimes_{\mathbf{Z}} \mathbf{Q} \rightarrow \Hom_{\mathbf{C}}(M_2(\Gamma),\mathbf{C})$ is injective and that the kernel of $\tilde{\pi}_{\Gamma} \otimes \mathbf{Q}$ is contained in $H_1(Y_{\Gamma}, \mathbf{Z}) \otimes_{\mathbf{Z}} \mathbf{Q}$ (considered as a subspace of $\tilde{\mathcal{M}}_{\mathbf{Q}, \Gamma}$ via $\iota_{\Gamma} \otimes \mathbf{Q}$). This concludes the proof of Lemma \ref{Hecke_equivariance_+_-}.
\end{proof}

By Lemma \ref{Hecke_equivariance_+_-} (\ref{Hecke_equivariance_produit_injective}), there is an action of $\mathbb{T}$ on $\tilde{\mathcal{M}}_{\mathbf{Q}, \Gamma}$ satisfying Theorem \ref{Hecke_Q} (\ref{Hecke_Q_compatible}). This action is obviously unique. Theorem  \ref{Hecke_Q} (\ref{Hecke_Q_periods}) and (\ref{Hecke_Q_exact_seq}) also follow from Lemma \ref{Hecke_equivariance_+_-}.

We now prove Theorem \ref{Hecke_Q} (\ref{Hecke_Q_lattice_odd}) and (\ref{Hecke_Q_T_2}). The assertion about diamond operators is straightforward. Let $p$ be a prime. If $p \mid N$, then by convention the diamond operator $\langle p \rangle$ is zero. We have:
$$\SL_2(\mathbf{Z})\begin{pmatrix} 1 & 0 \\ 0 & p \end{pmatrix} \SL_2(\mathbf{Z}) = \SL_2(\mathbf{Z})g_{\infty} \bigcup_{i \in I} \SL_2(\mathbf{Z}) g_i \text{ ,}$$
where $I = \{-\frac{p-1}{2}, ..., \frac{p-1}{2}\}$, $g_{\infty} =  \begin{pmatrix} p & 0 \\ 0 & 1 \end{pmatrix}$ and $g_i =  \begin{pmatrix} 1 & i \\ 0 & p \end{pmatrix}$. Recall that for all $g \in \SL_2(\mathbf{Z})$, we have $g_ig = t_{i, \SL_2(\mathbf{Z})}(g)g_{\sigma_g(i)}$ for some $t_{i, \SL_2(\mathbf{Z})} : \SL_2(\mathbf{Z}) \rightarrow \SL_2(\mathbf{Z})$ and $\sigma_{g} : I \cup \{\infty\}\xrightarrow{\sim} I \cup \{\infty\}$.
For all $(g,g') \in \SL_2(\mathbf{Z})^2$, we have in $\tilde{\mathcal{M}}_{\mathbf{Q}, \Gamma}$:
\begin{align*}
T_p\{g,g'\}_{\mathbf{Q}} &= \langle p \rangle \{g_{\infty}g, g_{\infty}g'\}_{\mathbf{Q}}+\sum_{i \in I}  \{g_ig, g_ig'\}_{\mathbf{Q}} \\& 
= \langle p \rangle \{t_{\infty, \SL_2(\mathbf{Z})}(g)g_{\sigma_{g}(\infty)}, t_{\infty, \SL_2(\mathbf{Z})}(g)\}_{\mathbf{Q}} + \langle p \rangle \{t_{\infty, \SL_2(\mathbf{Z})}(g), t_{\infty, \SL_2(\mathbf{Z})}(g')\}_{\mathbf{Q}} \\& + \langle p \rangle \{t_{\infty, \SL_2(\mathbf{Z})}(g'), t_{\infty, \SL_2(\mathbf{Z})}(g')g_{\sigma_{g'}(\infty)}\}_{\mathbf{Q}} \\&+ \sum_{i \in I}\{t_{i, \SL_2(\mathbf{Z})}(g)g_{\sigma_g(i)}, t_{i, \SL_2(\mathbf{Z})}(g)\}_{\mathbf{Q}}+\{t_{i, \SL_2(\mathbf{Z})}(g), t_{i, \SL_2(\mathbf{Z})}(g')\}_{\mathbf{Q}}\\&+\{t_{i, \SL_2(\mathbf{Z})}(g'), t_{i, \SL_2(\mathbf{Z})}(g')g_{\sigma_{g'}(i)}\}_{\mathbf{Q}} \text{ .}
\end{align*}
For all $i \in I \cup \{\infty\}$, we have $\{t_{i, \SL_2(\mathbf{Z})}(g)g_{\sigma_g(i)}, t_{i, \SL_2(\mathbf{Z})}(g)\}_{\mathbf{Q}} \in \frac{1}{p}\mathbf{Z}\cdot \{t_{i, \SL_2(\mathbf{Z})}(g), t_{i, \SL_2(\mathbf{Z})}(g)T\}_{\mathbf{Q}}$ and $\{t_{i, \SL_2(\mathbf{Z})}(g'), t_{i, \SL_2(\mathbf{Z})}(g')g_{\sigma_{g'}(i)}\}_{\mathbf{Q}} \in \frac{1}{p}\mathbf{Z}\cdot \{t_{i, \SL_2(\mathbf{Z})}(g'), t_{i, \SL_2(\mathbf{Z})}(g')T\}_{\mathbf{Q}}$. This proves Theorem \ref{Hecke_Q} (\ref{Hecke_Q_T_2}). Assume now that $p \nmid 2N$. To conclude the proof of Theorem \ref{Hecke_Q} (\ref{Hecke_Q_lattice_odd}), it suffices to show that for all $g \in \SL_2(\mathbf{Z})$, we have
\begin{equation}\label{lemma_integrality}
\langle p \rangle \{t_{\infty, \SL_2(\mathbf{Z})}(g)g_{\sigma_{g}(\infty)}, t_{\infty, \SL_2(\mathbf{Z})}(g)\}_{\mathbf{Q}} + \sum_{i \in I}\{t_{i, \SL_2(\mathbf{Z})}(g)g_{\sigma_g(i)}, t_{i, \SL_2(\mathbf{Z})}(g)\}_{\mathbf{Q}} =0 \text{ .}
\end{equation}
Write $g = \begin{pmatrix} a & b \\ c & d \end{pmatrix}$. We consider two cases. First, assume that $p \mid c$. Then $\sigma_g(\infty)=\infty$ and $\sigma_g$ induces a permutation of $I$. We have:
\begin{align*}
\langle p \rangle \{t_{\infty, \SL_2(\mathbf{Z})}(g)g_{\sigma_{g}(\infty)}, t_{\infty, \SL_2(\mathbf{Z})}(g)\}_{\mathbf{Q}} &+ \sum_{i \in I}\{t_{i, \SL_2(\mathbf{Z})}(g)g_{\sigma_g(i)}, t_{i, \SL_2(\mathbf{Z})}(g)\}_{\mathbf{Q}} \\& = \frac{1}{p}\cdot \sum_{i \in I} \sigma_g(i) \cdot \{t_{i, \SL_2(\mathbf{Z})}(g)T, t_{i, \SL_2(\mathbf{Z})}(g)\}_{\mathbf{Q}}
\end{align*}
Since $p\neq 2$, we have $ \sum_{i \in I} \sigma_g(i) = \sum_{i \in I} i = 0$, so to prove (\ref{lemma_integrality}) it suffices to prove that $\{t_{i, \SL_2(\mathbf{Z})}(g)T, t_{i, \SL_2(\mathbf{Z})}(g)\}_{\mathbf{Q}}$ is independant of $i$, \ie that the cusp $\Gamma t_{i, \SL_2(\mathbf{Z})}(g) \infty$ is independant of $i$. We have $t_{i, \SL_2(\mathbf{Z})}(g) \infty = g_i g \infty = \frac{a+ic}{pc}$. Since $p \mid c$ and $\gcd(a,c)=1$, the fraction $\frac{a+ic}{pc}$ is irreducible. Since $\Gamma = \Gamma_1(N)$ or $\Gamma=\Gamma_0(N)$ and $p \nmid N$, we see that $\Gamma \frac{a+ic}{pc} = \Gamma \frac{a}{pc}$ is independant of $i$.

Assume now that $p \nmid c$. Let $j \in I$ such that $p \mid a+jc$. Then $\sigma_g(j) = \infty$ and $\sigma_g(\infty)=j'$ where $j' \in I$ is such that $j' \equiv dc^{-1}\text{ (modulo }p\text{)}$. We have:
\begin{align*}
\langle p \rangle \{t_{\infty, \SL_2(\mathbf{Z})}(g)g_{\sigma_{g}(\infty)}, t_{\infty, \SL_2(\mathbf{Z})}(g)\}_{\mathbf{Q}} &+ \sum_{i \in I}\{t_{i, \SL_2(\mathbf{Z})}(g)g_{\sigma_g(i)}, t_{i, \SL_2(\mathbf{Z})}(g)\}_{\mathbf{Q}} \\& = \frac{\sigma_g(\infty)}{p}\cdot \langle p \rangle \{t_{\infty, \SL_2(\mathbf{Z})}(g)T, t_{\infty, \SL_2(\mathbf{Z})}(g)\}_{\mathbf{Q}} \\&+\sum_{i\in I \atop i \neq j} \frac{\sigma_g(i)}{p}\cdot  \{t_{i, \SL_2(\mathbf{Z})}(g)T, t_{i, \SL_2(\mathbf{Z})}(g)\}_{\mathbf{Q}} \text{ .}
\end{align*}
Since $p \neq 2$, we have $$\sigma_{g}(\infty) + \sum_{i \in I \atop i \neq j} \sigma_g(i) = \sum_{i \in I} i = 0 \text{ .}$$
Thus, to prove (\ref{lemma_integrality}) it suffices to prove that the elements $\langle p \rangle \{t_{\infty, \SL_2(\mathbf{Z})}(g)T, t_{\infty, \SL_2(\mathbf{Z})}(g)\}_{\mathbf{Q}}$ and $\{t_{i, \SL_2(\mathbf{Z})}(g)T, t_{i, \SL_2(\mathbf{Z})}(g)\}_{\mathbf{Q}}$ for all $i \in I \backslash \{j\}$ coincide. Equivalently, it suffices to prove that the cusps $\langle p \rangle \Gamma t_{\infty, \SL_2(\mathbf{Z})}(g) \infty$ and $\Gamma t_{i, \SL_2(\mathbf{Z})}(g) \infty$ are the same, for all $i \in I \backslash \{j\}$. We have $\langle p \rangle \Gamma t_{\infty, \SL_2(\mathbf{Z})}(g) \infty = \langle p \rangle \Gamma \frac{pa}{c}$ and $\Gamma t_{i, \SL_2(\mathbf{Z})}(g) \infty = \Gamma \frac{a+ic}{pc}$. The facts that $p \nmid a+ic$ if $i \neq j$, $\Gamma = \Gamma_1(N)$ or $\Gamma=\Gamma_0(N)$ and $p \nmid N$ imply that these cusps are all the same. This concludes the proof of (\ref{lemma_integrality}).

Theorem \ref{Hecke_Q} (\ref{Atkin_Lehner_M}) follows from the fact that for any $g \in \SL_2(\mathbf{Z})$, we can write $\begin{pmatrix} 0 & -1 \\ N & 0 \end{pmatrix} g = g' \cdot \begin{pmatrix} a & b \\ 0 & d \end{pmatrix}$ for some $g' \in \SL_2(\mathbf{Z})$ and $a,b,d \in \mathbf{Z}$ with $ad=N$. Theorem \ref{Hecke_Q} (\ref{hecke_conjugation_M}) follows from (\ref{Hecke_Q_exact_seq}). This concludes the proof of Theorem \ref{Hecke_Q} .
\end{proof}

\subsection{Manin symbols}\label{paragraph_Manin}
Following Manin \cite{Manin_parabolic}, consider the map
$$\xi_{\Gamma} : \mathbf{Z}[\Gamma \backslash \PSL_2(\mathbf{Z})] \rightarrow \mathcal{M}_{\Gamma}$$
defined by $\Gamma g \mapsto \{g0, g\infty\}$. Let $S = \begin{pmatrix} 0 & -1 \\ 1 & 0 \end{pmatrix}$ and $U = \begin{pmatrix} 1 & -1 \\ 1 & 0 \end{pmatrix}$. We have $S^2=U^3=\begin{pmatrix} -1 & 0 \\ 0 & -1 \end{pmatrix}$. 

\begin{thm}[Manin]\label{Manin_thm}
\begin{enumerate}
\item\label{Manin_surjectivity} The map $\xi_{\Gamma}$ is surjective.
\item\label{Manin_kernel} We have $$\Ker(\xi_{\Gamma}) = \mathbf{Z}[\Gamma \backslash \PSL_2(\mathbf{Z})]^U + \mathbf{Z}[\Gamma \backslash \PSL_2(\mathbf{Z})]^S$$
where for any $g \in \SL_2(\mathbf{Z})$, $\mathbf{Z}[\Gamma \backslash \PSL_2(\mathbf{Z})]^g$ denotes the subgroup of elements fixed by the right multiplication by $g$.
\end{enumerate}
\end{thm}

Recall that in \S \ref{intro_Manin_symbols} we defined a map $\tilde{\xi}_{\Gamma} : \mathbf{Z}[\Gamma\backslash \PSL_2(\mathbf{Z})]\rightarrow \tilde{\mathcal{M}}_{\Gamma}$ by $\tilde{\xi}_{\Gamma}(\Gamma g) = \{g,gS\}$ for all $g\in \PSL_2(\mathbf{Z})$. We have $\xi_{\Gamma} = \pi_{\Gamma} \circ \tilde{\xi}_{\Gamma}$. Let $\mathbf{Z}[\mathcal{C}_{\Gamma}]_0$ be the quotient of $\mathbf{Z}[\mathcal{C}_{\Gamma}]$ by the element $\frac{1}{d_{\Gamma}}\sum_{c \in \mathcal{C}_{\Gamma}} e_c\cdot [c]$ where $e_c$ is the width of $c$. If $c \in \mathcal{C}_{\Gamma}$, we denote by $(c)$ the image of $[c]$ in $\mathbf{Z}[\mathcal{C}_{\Gamma}]_0$. We have a canonical injective group homomorphism $\mathbf{Z}[\mathcal{C}_{\Gamma}]_0 \hookrightarrow \tilde{\mathcal{M}}_{\Gamma}$ coming from the exact sequence (\ref{exact_sequence_2}).

Theorem \ref{intro_thm_Manin} is a consequence of the following result.
\begin{thm}\label{generalized_Manin_thm}
Let $\varphi_{\Gamma} : \mathbf{Z}[\Gamma \backslash \PSL_2(\mathbf{Z})] \rightarrow \mathbf{Z}[\mathcal{C}_{\Gamma}]_0$ be the group homomorphism given by $\varphi_{\Gamma}([\Gamma g]) = (\Gamma g \infty)$.
\begin{enumerate}
\item \label{surjectivity_generalized_Manin_general} The coimage of $\tilde{\xi}_{\Gamma}$ is canonically isomorphic to the coimage of the restriction of $\varphi_{\Gamma}$ to $ \mathbf{Z}[\Gamma \backslash \PSL_2(\mathbf{Z})]^U$. 
\item \label{surjectivity_generalized_Manin_special} 
Let $p\geq 5$ be a prime and $n \in \mathbf{N}$. If $\Gamma=\Gamma_1(p^n)$ or $\Gamma = \Gamma_0(p^n)$, then the image of $\tilde{\xi}_{\Gamma}$ has index dividing $3$, the divisibility being strict of and only if $p \equiv 1 \text{ (modulo 3}\text{)}$ and $\Gamma = \Gamma_0(p^n)$.
\item \label{kernel_generalized_Manin} The kernel of $\tilde{\xi}_{\Gamma}$ is spanned by the following two subgroups of  $\mathbf{Z}[\Gamma \backslash \PSL_2(\mathbf{Z})]$:
\begin{itemize}
\item The subgroup $\mathbf{Z}[\Gamma \backslash \PSL_2(\mathbf{Z})]^S$.
\item The kernel of the restriction of $\varphi_{\Gamma}$ to $ \mathbf{Z}[\Gamma \backslash \PSL_2(\mathbf{Z})]^U$.
\end{itemize}
\end{enumerate}
\end{thm}
\begin{proof}
We prove (\ref{surjectivity_generalized_Manin_general}). By Theorem \ref{Manin_thm}, $\pi_{\Gamma}\circ \tilde{\xi}_{\Gamma}$ is surjective and the coimage of $\tilde{\xi}_{\Gamma}$ in $\tilde{\mathcal{M}}_{\Gamma}$ equals the coimage in $\Ker(\pi_{\Gamma})$ of the restriction of $\tilde{\xi}_{\Gamma}$ to $\mathbf{Z}[\Gamma \backslash \SL_2(\mathbf{Z})]^U + \mathbf{Z}[\Gamma \backslash \SL_2(\mathbf{Z})]^S$. Note that the restriction of $\tilde{\xi}_{\Gamma}$ to $\mathbf{Z}[\Gamma \backslash \SL_2(\mathbf{Z})]^S$ is zero. 
\begin{lem}\label{lemma_Manin_relations_cusps}
For all $g \in \SL_2(\mathbf{Z})$, we have in $\tilde{\mathcal{M}}_{\Gamma}$:
$$\tilde{\xi}_{\Gamma}([\Gamma g]+ [\Gamma gU]+[\Gamma gU^2]) = \{g,gT\}+\{gU,gUT\}+\{gU^2,gU^2T\} \text{ .}$$
\end{lem}
\begin{proof}
Using the following equalities: $US=-T$, $U^2S=-UT$ and $U^2T=S$, we have:
\begin{align*}
\tilde{\xi}_{\Gamma}([\Gamma g]+ [\Gamma gU]+[\Gamma gU^2]) &= \{g,gS\}+\{gU,gUS\}+\{gU^2, gU^2S\} \\& 
= \{g,gS\}+\{gU,gT\}+\{gU^2, gUT\} \\& = \{g,gT\}+\{gT,gS\}+\{gU,gUT\}+\{gUT,gT\}+\{gU^2, gU^2T\}+\{gU^2T, gUT\} \\&
=\{g,gT\}+\{gU,gUT\}+\{gU^2,gU^2T\}+\{gT,gS\}+\{gUT,gT\}+\{gS,gUT\} \\& = \{g,gT\}+\{gU,gUT\}+\{gU^2,gU^2T\} \text{ .}
\end{align*}
\end{proof}
By Lemma \ref{lemma_Manin_relations_cusps}, the restrictions of $\tilde{\xi}_{\Gamma}$ and $\varphi_{\Gamma}$ to $\mathbf{Z}[\Gamma \backslash \SL_2(\mathbf{Z})]^U$ are equal. This proves (\ref{surjectivity_generalized_Manin_general}). We now prove (\ref{surjectivity_generalized_Manin_special}). Let $p \geq 5$ be a prime and $n \in \mathbf{N}$. 
\begin{lem}\label{lemma_g0_g1_goo}
Assume $\Gamma = \Gamma_1(p^n)$ or $\Gamma = \Gamma_0(p^n)$. Consider the map 
$f_{\Gamma} : \mathbf{Z}[\Gamma \backslash \PSL_2(\mathbf{Z})] \rightarrow \mathbf{Z}[\mathcal{C}_{\Gamma}]$ sending $[\Gamma g]$ to $[\Gamma g 0]+[\Gamma g 1]+[\Gamma g \infty]$. The image of $f_{\Gamma}$ consists of the elements of degree divisible by $3$ in $\mathbf{Z}[\mathcal{C}_{\Gamma}]$.
\end{lem}
\begin{proof}
Since $f_{\Gamma}$ is functorial in $\Gamma$, it suffices to prove Lemma \ref{lemma_g0_g1_goo} when $\Gamma = \Gamma_1(p^n)$. Recall that if $(a,b,a',b') \in \mathbf{Z}^4$ and $\gcd(a,b)=\gcd(a',b')=1$, then we have $\Gamma \cdot \frac{a}{b} = \Gamma \cdot \frac{a'}{b'}$ if and only if $\begin{pmatrix} a' \\ b'\end{pmatrix} \equiv \pm \begin{pmatrix} a+kb \\ b\end{pmatrix} \text{ (modulo }p^n\text{)}$ for some $k \in \mathbf{Z}$. The class of the vector $\begin{pmatrix} a \\ b\end{pmatrix}$ modulo the previous equivalence relation is denoted by $\begin{Bmatrix} a  \\ b \end{Bmatrix}$, and such a symbol is uniquely identified with an element of $\mathcal{C}_{\Gamma}$. If $r\geq 0$ is an integer, we denote by $\mathcal{C}_{\Gamma}^{(r)} \subset \mathcal{C}_{\Gamma}$ the set of $\begin{Bmatrix} a  \\ b \end{Bmatrix}$ with $\gcd(b,p^n)=p^r$. We have $\mathcal{C}_{\Gamma} = \bigsqcup_{r=0}^n \mathcal{C}_{\Gamma}^{(r)}$.

We first show that the elements of $\mathbf{Z}[\mathcal{C}_{\Gamma}^{(0)}]$ of degree divisible by $3$ are contained in $\text{Im}(f_{\Gamma})$. First note that there is a bijection $(\mathbf{Z}/p^n\mathbf{Z})^{\times}/\pm1 \xrightarrow{\sim} \mathcal{C}_{\Gamma}^{(0)}$ given by $x \mapsto \begin{Bmatrix} 0  \\ x \end{Bmatrix}$. For any $x,y \in (\mathbf{Z}/p^n\mathbf{Z})^{\times}$ such that $x+y$ is prime to $p$, we have $[\begin{Bmatrix} 0  \\ x \end{Bmatrix}]+[\begin{Bmatrix} 0  \\ y \end{Bmatrix}]+[\begin{Bmatrix} 0  \\ x+y \end{Bmatrix}] \in \text{Im}(f_{\Gamma})$. Note in particular that $2\cdot [\begin{Bmatrix} 0  \\ x \end{Bmatrix}]+[\begin{Bmatrix} 0  \\ 2x \end{Bmatrix}] \in \text{Im}(f_{\Gamma})$. If $x-y$ is prime to $p$, we also have $[\begin{Bmatrix} 0  \\ x \end{Bmatrix}]+[\begin{Bmatrix} 0  \\ y \end{Bmatrix}]+[\begin{Bmatrix} 0  \\ x-y \end{Bmatrix}] \in \text{Im}(f_{\Gamma})$. Thus, for any  $x,y \in (\mathbf{Z}/p^n\mathbf{Z})^{\times}$ with $x \pm y$ prime to $p$, we have $[\begin{Bmatrix} 0  \\ x+y \end{Bmatrix}] - [\begin{Bmatrix} 0  \\ x-y \end{Bmatrix}] \in \text{Im}(f_{\Gamma})$. Since $p$ is odd, for any  $x,y \in (\mathbf{Z}/p^n\mathbf{Z})^{\times}$ with $x \pm y$ prime to $p$ we have $[\begin{Bmatrix} 0  \\ x \end{Bmatrix}] - [\begin{Bmatrix} 0  \\ y \end{Bmatrix}] \in \text{Im}(f_{\Gamma})$. Since $p>3$, by letting $x=2$ and $y=1$ we get $[\begin{Bmatrix} 0  \\ 2 \end{Bmatrix}] - [\begin{Bmatrix} 0  \\ 1 \end{Bmatrix}] \in \text{Im}(f_{\Gamma})$. We also have $2\cdot [\begin{Bmatrix} 0  \\ 1\end{Bmatrix}]+ [\begin{Bmatrix} 0  \\ 2 \end{Bmatrix}] = f_{\Gamma}(\Gamma \begin{pmatrix} 1 & 0 \\ 1 & 1 \end{pmatrix}) \in \text{Im}(f_{\Gamma})$. We thus get $3\cdot [\begin{Bmatrix} 0  \\ 1 \end{Bmatrix}] \in \text{Im}(f_{\Gamma})$. It remains to show that for all $x \in (\mathbf{Z}/p^n\mathbf{Z})^{\times}$, we have $[\begin{Bmatrix} 0  \\ x \end{Bmatrix}] - [\begin{Bmatrix} 0  \\ 1 \end{Bmatrix}] \in \text{Im}(f_{\Gamma})$. We already know this if $x \not\equiv \pm 1 \text{ (modulo }p\text{)}$. If $x \equiv \pm 1 \text{ (modulo }p\text{)}$ then since $p>3$ we have $\frac{x}{2} \equiv \pm \frac{1}{2} \not\equiv \pm 1 \text{ (modulo }p\text{)}$, so we get $[\begin{Bmatrix} 0  \\ \frac{x}{2} \end{Bmatrix}] - [\begin{Bmatrix} 0  \\ 1 \end{Bmatrix}] \in \text{Im}(f_{\Gamma})$. Since we know that $[\begin{Bmatrix} 0  \\ x \end{Bmatrix}] + 2\cdot [\begin{Bmatrix} 0  \\ \frac{x}{2}\end{Bmatrix}] \in \text{Im}(f_{\Gamma})$, we get $[\begin{Bmatrix} 0  \\ x \end{Bmatrix}] -  [\begin{Bmatrix} 0  \\ 1 \end{Bmatrix}] \in \text{Im}(f_{\Gamma})$.

We now conclude the proof of Lemma \ref{lemma_g0_g1_goo}. Let $r$ be an integer such that $1 \leq r \leq n$. Let $\begin{Bmatrix} v  \\ up^r \end{Bmatrix} \in \mathcal{C}_{\Gamma}^{(r)}$, where $(v,u) \in \mathbf{Z}^2$ is such that $\gcd(v,up^r)=1$ and $u \not\equiv 0 \text{ (modulo }p\text{)}$. Let $g = \begin{pmatrix} a & b \\ c & d \end{pmatrix} \in \SL_2(\mathbf{Z})$ be such that $\gcd(c,p)=\gcd(d,p)=1$, $c+d \equiv up^r\text{ (modulo }p^n\text{)}$ and $d \equiv v^{-1}\text{ (modulo }p^r\text{)}$. Then we have $f_{\Gamma}([\Gamma g]) = [\begin{Bmatrix} 0 \\ c\end{Bmatrix}]+[\begin{Bmatrix} 0 \\ d\end{Bmatrix}]+[\begin{Bmatrix} v \\ up^r\end{Bmatrix}]$. Thus, we have $[\begin{Bmatrix} v  \\ up^r \end{Bmatrix}] - [\begin{Bmatrix} 0  \\ 1\end{Bmatrix}] \in \text{Im}(f_{\Gamma})$. 
\end{proof}

By Theorem \ref{generalized_Manin_thm} (\ref{surjectivity_generalized_Manin_general}) and Lemma \ref{lemma_g0_g1_goo}, the index of the image of $\tilde{\xi}_{\Gamma}$ in $\mathbf{Z}[\mathcal{C}_{\Gamma}]_0$ is divisible by $3$. Moreover, this index is one if and only if either there exists an element of $\mathbf{Z}[\Gamma \backslash \PSL_2(\mathbf{Z})]^U$ of degree $1$ or $\sum_{c \in \mathcal{C}_{\Gamma}} e_c$ is relatively prime to $3$. If $\Gamma = \Gamma_1(p^n)$, we check that these conditions never hold. If $\Gamma = \Gamma_0(p^n)$, we check that these two conditions hold if $p\equiv 1 \text{ (modulo }3\text{)}$ and none of them hold otherwise. This proves Theorem \ref{generalized_Manin_thm} (\ref{surjectivity_generalized_Manin_special})

Theorem \ref{generalized_Manin_thm} (\ref{Manin_kernel}) follows from Theorem \ref{Manin_thm} (\ref{Manin_kernel}) and Lemma \ref{lemma_Manin_relations_cusps}. This concludes the proof of Theorem \ref{generalized_Manin_thm}.
\end{proof}

\begin{rem}\label{Manin_rmk}
It would be interesting to compute the index of $\tilde{\xi}_{\Gamma}$ if $\Gamma = \Gamma_0(N)$ or $\Gamma = \Gamma_1(N)$ for any $N\geq 1$. If $N$ is even, then this index is infinite (it suffices to check that for $N=2$). If $N$ is odd, we do not know if this index finite (the numerical computations seem to indicate that this index divides $3$). The proof would rely on a generalization of Lemma \ref{lemma_g0_g1_goo}, which seems to be hard if $N$ is highly composite.
\end{rem}

\subsection{Self-duality and Zagier's Petersson inner product}\label{paragraph_Petersson}
In this paragraph, we do not assume anything about $\Gamma$, which could in particular be a non-congruence subgroup. We denote by $N$ the general level of $\Gamma$.

By taking duals, the map $\tilde{p}_{\Gamma} : \tilde{\mathcal{M}}_{\Gamma} \rightarrow \Hom_{\mathbf{C}}(M_2(\Gamma), \mathbf{C})$ induces a map $\tilde{p}_{\Gamma}^*  : M_2(\Gamma) \rightarrow \Hom_{\mathbf{Z}}(\tilde{M}_{\Gamma} , \mathbf{C})$. If $f \in M_2(\Gamma)$ and $g \in \SL_2(\mathbf{Z})$, recall that by construction we have $\tilde{p}_{\Gamma}^*(f)(\{g,gS\}) = L(f\mid g, 1)$ and $\tilde{p}_{\Gamma}^*(f)(\{g,gT\}) = 2i\pi \cdot a_0(f \mid g)$, where $S = \begin{pmatrix} 0 & -1 \\ 1 & 0 \end{pmatrix}$ and $T = \begin{pmatrix}1 & 1 \\ 0 & 1 \end{pmatrix}$. 

Following Zagier \cite{Zagier_Rankin} we can extend the usual Petersson Hermitian pairing on cuspidal forms to the whole space of modular forms $M_2(\Gamma)$. We now recall the definition. For any real number $T>1$, define the truncated fundamental domain $\mathcal{F}_T = \{z \in \mathfrak{h} \text{ such that } \abs{z} \geq 1, \abs{ \Re(z) }\leq \frac{1}{2} \text{ and } \Im(z)<T\}$, where $\mathfrak{h}$ is the upper-half plane. Then according to Zagier, if $f_1,f_2 \in M_2(\Gamma)$, we define the Petersson pairing of $f_1$ and $f_2$ to be
$$(f_1, f_2) := \frac{1}{[\SL_2(\mathbf{Z}) : \Gamma]} \cdot \lim_{T \rightarrow +\infty} \sum_{g \in \Gamma \backslash \SL_2(\mathbf{Z})} \int_{\mathcal{F}_T} (f_1\mid g)(z) \cdot \overline{(f_2\mid g)(z)}dxdy - T\cdot a_0(f_1\mid g)\cdot \overline{a_0(f_2 \mid g)} \text{ .}$$
We get a Hermitian bilinear pairing on $M_2(\Gamma)$. Although the restriction of the Petersson pairing to cuspidal modular forms is non-degenerate, it can be degenerate on $M_2(\Gamma)$. Nevertheless, Pasol and Popa proved the following result \cite[Theorems 4.4 and 4.5]{Popa_Petersson}, which seems related to Remark \ref{rem_period_non_iso}.

\begin{thm}[Pasol and Popa]\label{non_degeneracy_petersson} 
\begin{enumerate}
\item Assume $\Gamma = \Gamma_0(N)$ or $\Gamma = \Gamma_1(N)$. The Petersson pairing $M_2(\Gamma) \times M_2(\Gamma) \rightarrow \mathbf{C}$ is non degenerate if $N$ is a prime power. 
\item The Petersson pairing is degenerate if $\Gamma=\Gamma_1(N)$ and $N$ is divisible by $pq$ with $p\neq q$ primes such that $q$ is not a primitive residue modulo $p$ (\eg if $N$ is divisible by $6$). It is also degenerate if $\Gamma=\Gamma_1(N)$ and $N$ is divisible by $p^2q$ with $p\neq q$ primes.
\item The Petersson pairing is degenerate if $\Gamma=\Gamma_0(N)$ and $N$ is square-free and not prime.
\end{enumerate}
\end{thm}

It is natural to ask whether the Petersson pairing can be expressed in terms of a pairing on $\Hom_{\mathbf{Z}}(\tilde{M}_{\Gamma}, \mathbf{C})$ \textit{via} the map $\tilde{p}_{\Gamma}^*$. This was done by Merel when restricting to cuspidal modular forms \cite[Th\'eor\`eme 2]{Merel_Manin_L}. In the general case, this was done by Pasol and Popa. More precisely, the following result is a reformulation of \cite[Theorem 8.6 a)]{Popa_Haberland}.

\begin{thm}[Pasol and Popa] \label{Haberland}
For any $f_1, f_2 \in M_2(\Gamma)$, we have:
\begin{align*}
48 i \pi^2\cdot (f_1, f_2) = \frac{1}{[\SL_2(\mathbf{Z}) : \Gamma]} \sum_{g \in \Gamma \backslash \SL_2(\mathbf{Z})} &\tilde{p}_{\Gamma}^*(f_1)(\{gTS,gT\})\cdot \overline{\tilde{p}_{\Gamma}^*(f_2)(\{gS,g\})} \\& - \tilde{p}_{\Gamma}^*(f_1)(\{gS,g\}) \cdot \overline{\tilde{p}_{\Gamma}^*(f_2)(\{gTS,gT\})} \\& + 4\cdot \tilde{p}_{\Gamma}^*(f_1)(\{g,gT\})\cdot \overline{ \tilde{p}_{\Gamma}^*(f_2)(\{g,gS\})} \\& - 4\cdot \tilde{p}_{\Gamma}^*(f_1)(\{g,gS\})\cdot \overline{ \tilde{p}_{\Gamma}^*(f_2)(\{g,gT\})} \text{ .}
\end{align*}
\end{thm}

Recall that in \S \ref{intro_duality}, we have defined an anti-symmetric pairing $\langle \cdot, \cdot \rangle$ on the $\mathbf{Z}$-dual $\tilde{\mathcal{M}}_{\Gamma}^*$ of $\tilde{\mathcal{M}}_{\Gamma}$. We easily check that $\langle \cdot, \cdot \rangle$ is anti-invariant for the complex conjugation $c$ acting on $\tilde{\mathcal{M}}_{\Gamma}$, namely for all $\varphi_1, \varphi_2 \in \tilde{\mathcal{M}}_{\Gamma}$, we have
\begin{equation}\label{complex_conjug_compatibility_pairing}
\langle \varphi_1 \circ c , \varphi_2 \rangle = -\langle \varphi_1 , \varphi_2 \circ c \rangle \text{ .}
\end{equation}

By tensoring with $\mathbf{C}$, $\langle \cdot , \cdot \rangle$ extends to an anti-symmetric bilinear pairing on $\Hom_{\mathbf{Z}}(\tilde{\mathcal{M}}_{\Gamma}, \mathbf{C})$. If $\varphi \in \Hom_{\mathbf{Z}}(\tilde{\mathcal{M}}_{\Gamma}, \mathbf{C})$, we define $\overline{\varphi} \in \Hom_{\mathbf{Z}}(\tilde{\mathcal{M}}_{\Gamma}, \mathbf{C})$ by the formula $\overline{\varphi}(m) = \overline{\varphi(m)}$ for all $m \in \tilde{\mathcal{M}}_{\Gamma}$.

We can restate Theorem \ref{Haberland} as
\begin{equation}\label{reformulation_Haberland}
-8 i \pi^2 (f_1, f_2) = \frac{1}{[\SL_2(\mathbf{Z}):\Gamma]}\cdot  \langle \tilde{p}_{\Gamma}^*(f_1) , \overline{\tilde{p}_{\Gamma}^*(f_2)} \rangle \text{ .}
 \end{equation}

We expect that two properties hold:
 \begin{enumerate}
 \item\label{property_degenerate_pairing} The pairing $\langle \cdot , \cdot \rangle$ takes values in $\mathbf{Z}$ and has determinant $\frac{1}{d_{\Gamma}} \prod_{c \in \mathcal{C}_{\Gamma}} e_c$.
\item\label{property_Hecke_pairing}  The pairing $\langle \cdot , \cdot \rangle$ exchanges the Hecke operators with their duals, \ie for all $\varphi_1, \varphi_2 \in \tilde{\mathcal{M}}_{\Gamma}^*$ and $T \in \mathbb{T}$, we have $$\langle T\varphi_1, \varphi_2 \rangle = \langle \varphi_1, W_N T W_N^{-1} \varphi_2 \rangle$$
after inverting $2N$. 
\end{enumerate}

If $N$ is a prime power, then by Theorem \ref{thm_iso_tilde_p} and Theorem \ref{non_degeneracy_petersson} the pairing $\langle \cdot, \cdot \rangle$ is non-degenerate and (\ref{property_Hecke_pairing}) is satisfied. Even if in general the Petersson pairing $(\cdot , \cdot)$ can be degenerate, this does not contradicts the conjecture that $\langle \cdot, \cdot \rangle$ should be non-degenerate. Indeed, we have seen that the map $\tilde{p}_{\Gamma}^*$ can have a kernel (when $N$ is not a prime power), and by Theorem \ref{Haberland}, the kernel of the Petersson pairing $(\cdot, \cdot)$ contains $\Ker(\tilde{p}_{\Gamma}^*)$. We do not know whether this inclusion is always an equality.

We now prove Theorem \ref{intro_thm_Petersson} 
\begin{proof}
Recall that by Theorem \ref{Hecke_Q} (\ref{hecke_conjugation_M}), we have a canonical $\mathbb{T}$-equivariant isomorphism
\begin{equation}\label{complex_conjug_dec_duality}
\tilde{\mathcal{M}}_{\Gamma} \otimes_{\mathbf{Z}} \mathbf{Z}[\frac{1}{2N}] \xrightarrow{\sim} H_1(X_{\Gamma}, \mathcal{C}_{\Gamma}, \mathbf{Z}[\frac{1}{2N}])^+ \bigoplus H_1(Y_{\Gamma}, \mathbf{Z}[\frac{1}{2N}])^- \text{ .}
\end{equation}
There is a canonical perfect bilinear pairing, namely the \textit{intersection pairing}
$$H_1(X_{\Gamma}, \mathcal{C}_{\Gamma}, \mathbf{Z}[\frac{1}{2}])^+ \times H_1(Y_{\Gamma}, \mathbf{Z}[\frac{1}{2}])^-  \rightarrow \mathbf{Z}[\frac{1}{2}] \text{ .}$$
This pairing exchanges the Hecke operators with their duals. Using (\ref{complex_conjug_dec_duality}), we get a canonical perfect bilinear pairing
$$ \cdot \bullet \cdot : \tilde{\mathcal{M}}_{\Gamma} \otimes_{\mathbf{Z}} \mathbf{Z}[\frac{1}{2N}] \times \tilde{\mathcal{M}}_{\Gamma} \otimes_{\mathbf{Z}} \mathbf{Z}[\frac{1}{2N}] \rightarrow \mathbf{Z}[\frac{1}{2N}] $$
exchanging the Hecke operators with their duals. This gives us a linear isomorphism
$$F : \tilde{\mathcal{M}}_{\Gamma} \otimes_{\mathbf{Z}} \mathbf{Z}[\frac{1}{6N}]  \rightarrow \tilde{\mathcal{M}}_{\Gamma}^*  \otimes_{\mathbf{Z}} \mathbf{Z}[\frac{1}{6N}] \text{ .} $$
The pairing $\langle \cdot, \cdot \rangle$ induces a linear map
$$G :  \tilde{\mathcal{M}}_{\Gamma}^*  \otimes_{\mathbf{Z}} \mathbf{Z}[\frac{1}{2N}]  \rightarrow  \tilde{\mathcal{M}}_{\Gamma}  \otimes_{\mathbf{Z}} \mathbf{Z}[\frac{1}{6N}] \text{ .}$$
Theorem \ref{intro_thm_Petersson} thus follows from the following result.
\begin{prop}\label{prop_intersection_inverse_pairing}
We have $F \circ G = \Id_{ \tilde{\mathcal{M}}_{\Gamma}^*  \otimes_{\mathbf{Z}} \mathbf{Z}[\frac{1}{2N}]}$.
\end{prop}
\begin{proof}
By (\ref{complex_conjug_dec_duality}), $G$ is the direct sum of two maps 
$$G^+ : \Hom(H_1(X_{\Gamma}, \mathcal{C}_{\Gamma}, \mathbf{Z}[\frac{1}{2N}])^+, \mathbf{Z}[\frac{1}{2N}]) \rightarrow  H_1(Y_{\Gamma}, \mathbf{Z}[\frac{1}{6N}])^- $$
and
$$G^- : \Hom(H_1(Y_{\Gamma}, \mathbf{Z}[\frac{1}{2N}])^-, \mathbf{Z}[\frac{1}{2N}]) \rightarrow  H_1(X_{\Gamma}, \mathcal{C}_{\Gamma}, \mathbf{Z}[\frac{1}{6N}])^+ \text{ .}$$
Similarly $F$ is the direct sum of
$$F^+ : H_1(X_{\Gamma}, \mathcal{C}_{\Gamma}, \mathbf{Z}[\frac{1}{6N}])^+ \rightarrow \Hom(H_1(Y_{\Gamma}, \mathbf{Z}[\frac{1}{6N}])^-, \mathbf{Z}[\frac{1}{6N}])$$
and
$$F^- : H_1(Y_{\Gamma}, \mathbf{Z}[\frac{1}{6N}])^-\rightarrow  \Hom(H_1(X_{\Gamma}, \mathcal{C}_{\Gamma}, \mathbf{Z}[\frac{1}{2N}])^+, \mathbf{Z}[\frac{1}{6N}]) \text{ .}$$
To prove Proposition \ref{prop_intersection_inverse_pairing}, it suffices to prove that 
\begin{equation}\label{eq1_prop_intersection_inverse_pairing}
F^- \circ G^+ = \Id_{\Hom(H_1(X_{\Gamma}, \mathcal{C}_{\Gamma}, \mathbf{Z}[\frac{1}{2N}])^+, \mathbf{Z}[\frac{1}{2N}])}
\end{equation}
and
\begin{equation}\label{eq2_prop_intersection_inverse_pairing}
F^+ \circ G^- = \Id_{\Hom(H_1(Y_{\Gamma}, \mathbf{Z}[\frac{1}{2N}])^-,\mathbf{Z}[\frac{1}{2N}])} \text{ .}
\end{equation}
Since $G^-$ (resp. $F^+$) is the dual of $G^+$ (resp. $F^-$), it suffices to prove (\ref{eq2_prop_intersection_inverse_pairing}).

We use an explicit formula to compute the intersection product, due to Merel \cite{Merel_HK}. Let $\rho=e^{\frac{2i\pi}{3}} \in \mathfrak{h}$ and $i = e^{\frac{i\pi}{2}} \in \mathfrak{h}$. Let $R$ (resp. $I$) be the image in $Y_{\Gamma}$ of $SL_2(\mathbf{Z})\cdot \rho$ (resp. $\SL_2(\mathbf{Z})\cdot i$). If $g \in \SL_2(\mathbf{Z})$, we denote by $\{gi, g\rho\}$ the image in $H_1(Y_{\Gamma}, R\cup I, \mathbf{Z})$ of the geodesic path in $\mathfrak{h}$ between $gi$ and $g\rho$. Merel proved \cite[Th\'eor\`emes 2 and 3]{Merel_HK} that $H_1(Y_{\Gamma}, R\cup I, \mathbf{Z})$ is generated by the element $\{gi, g\rho\}$ and that $$H_1(Y_{\Gamma}, \mathbf{Z}) = \big\{\sum_{\Gamma g \in \SL_2(\mathbf{Z})} \lambda_g \cdot \{gi, g\rho\} \text{ such that } \lambda_g \in \mathbf{Z} \text{ and } \lambda_{g}+\lambda_{gS}=\lambda_g+\lambda_{g\tau}+\lambda_{g\tau^2}=0 \text{ for all }g\in \SL_2(\mathbf{Z})\big\}$$
where $\tau = \begin{pmatrix} 0 & -1 \\ 1 & 1 \end{pmatrix} = ST$ stabilizes $\rho$. Furthermore, Merel proved the following formula for the intersection product $\bullet : H_1(X_{\Gamma}, \mathcal{C}_{\Gamma}, \mathbf{Z})\times H_1(Y_{\Gamma}, \mathbf{Z}) \rightarrow \mathbf{Z}$. If $x = \sum_{g \in \SL_2(\mathbf{Z})} \mu_g\cdot \{g0, g\infty\} \in H_1(X_{\Gamma}, \mathcal{C}_{\Gamma}, \mathbf{Z})$ and $y = \sum_{g \in \SL_2(\mathbf{Z})} \lambda_g\cdot \{gi, g\rho\} \in H_1(Y_{\Gamma}, \mathbf{Z})$
then \cite[Corollaire 3]{Merel_HK} we have 
\begin{equation}\label{Merel_formula_intersection}
x\bullet y =  \sum_{g \in \Gamma\backslash \SL_2(\mathbf{Z})} \lambda_g\cdot \mu_g \text{ .}
\end{equation}
To prove (\ref{eq1_prop_intersection_inverse_pairing}), it thus suffices to prove the following result, whose statement and proof are generalization of the first formula in \cite[Th\'eor\`eme 2]{Merel_Manin_L}.
\begin{prop}\label{expression_G^+_rho_i}
We do not make any assumption on $\Gamma$ (which can thus be a non-congruence subgroup). For all $\varphi \in \Hom(H_1(X_{\Gamma}, \mathcal{C}_{\Gamma}, \mathbf{Z}), \mathbf{Z}) \subset \tilde{\mathcal{M}}_{\Gamma}^*$, we have in $H_1(Y_{\Gamma}, \mathbf{Z}) \subset \tilde{\mathcal{M}}_{\Gamma}$:
$$G(\varphi) = \sum_{g \in \Gamma\backslash \SL_2(\mathbf{Z})} \varphi(\{g0, g\infty\})\cdot \{gi, g\rho\} \text{ .}$$
In other words, the pairing $\langle \cdot, \cdot \rangle$ induces by restriction a map $\Hom(H_1(X_{\Gamma}, \mathcal{C}_{\Gamma}, \mathbf{Z}), \mathbf{Z})  \rightarrow H_1(Y_{\Gamma}, \mathbf{Z})$ which is inverse to the intersection pairing.
\end{prop}
\begin{proof}
Let $N$ be the general level of $\Gamma$. We let $\tilde{\mathcal{M}}_{\Gamma, R\cup I}$ be the pushout of the two (injective) morphisms $H_1(Y_{\Gamma}, \mathbf{Z}[\frac{1}{6N}]) \rightarrow \tilde{\mathcal{M}}_{\Gamma} \otimes_{\mathbf{Z}} \mathbf{Z}[\frac{1}{6N}]$ and $H_1(Y_{\Gamma}, \mathbf{Z}[\frac{1}{6N}]) \rightarrow H_1(Y_{\Gamma}, R\cup I, \mathbf{Z}[\frac{1}{6N}])$. Let $\tilde{\mathcal{M}}_{\Gamma, R\cup I}' = \tilde{\mathcal{M}}_{\Gamma, R\cup I} \oplus \mathbf{Z}[\frac{1}{6N}]\cdot \{1, i\}$ (where $\{1, i\}$ is a formal symbol). For any $(g,g') \in \SL_2(\mathbf{Z})$, we define the symbol $\{g,g'i\}$ in $\tilde{\mathcal{M}}_{\Gamma, R\cup I}'$ to be $\{g,1\}+\{1,i\}+\{i,g'i\} \in\tilde{\mathcal{M}}_{\Gamma, R\cup I}'$. Similarly, we define the symbol $\{g,g'\rho\}$ in $\tilde{\mathcal{M}}_{\Gamma, R\cup I}'$ to be $\{g,1\}+\{1,i\}+\{i,g'\rho\} \in\tilde{\mathcal{M}}_{\Gamma, R\cup I}'$. Finally, we define $\{g'i,g\}$ to be $-\{g,g'i\}$ and $\{g'\rho,g\}$ to be $-\{g,g'\rho\}$. We can thus talk about symbols $\{\alpha, \beta\}$ in $\tilde{\mathcal{M}}_{\Gamma, R\cup I}'$, where $\alpha, \beta \in \{\SL_2(\mathbf{Z}), \SL_2(\mathbf{Z})\cdot i,\SL_2(\mathbf{Z})\cdot \rho\}$. We easily check that the Chasles relation are satisfied and that for any $\gamma \in \Gamma$, we have $\{\gamma \alpha, \gamma \beta\} = \{\alpha, \beta\}$. Since $\tilde{\mathcal{M}}_{\Gamma} \otimes_{\mathbf{Z}} \mathbf{Z}[\frac{1}{6N}]$ embeds into $\tilde{\mathcal{M}}_{\Gamma, R\cup I}'$, it is legitimate to do all our computations in $\tilde{\mathcal{M}}_{\Gamma, R\cup I}'$.

On the one hand, by definition of $G$, we have in $\tilde{\mathcal{M}}_{\Gamma}$:
\begin{align*}
6\cdot G(\varphi) &= \sum_{g \in \Gamma \backslash \SL_2(\mathbf{Z})} \varphi(\{gS\infty,g\infty\})\cdot \{gTS, gT\}-  \varphi(\{gTS\infty, gT\infty\}) \cdot \{gS,g\} \\& + 4\cdot  \varphi(\{g\infty,gS\infty\}) \cdot \{g,gT\} 
\\&= \sum_{g \in \Gamma \backslash \SL_2(\mathbf{Z})} \varphi(\{g\infty,g0\})\cdot \{gSTS, gST\}-\varphi(\{gSTS\infty, gST\infty\}) \cdot \{g,gS\} \\&- 4\cdot  \varphi(\{g0,g\infty\}) \cdot \{g,gT\} 
\\&=\sum_{g \in \Gamma \backslash \SL_2(\mathbf{Z})} \varphi(\{g\infty,g0\})\cdot \{g\tau S, g\tau\} - \varphi(\{g\tau S\infty, g\tau\infty\}) \cdot \{g,gS\}\\& - 4\cdot  \varphi(\{g0,g\infty\}) \cdot \{g,gT\} \\&
=\sum_{g \in \Gamma \backslash \SL_2(\mathbf{Z})} \varphi(\{g\tau0,g\tau\infty\})\cdot \{g\tau^2, g\tau^2S\}-\varphi(\{g\tau 0, g\tau\infty\}) \cdot \{g,gS\} \\& - 4\cdot  \varphi(\{g0,g\infty\}) \cdot \{g,gT\} \text{ .}
\end{align*}
On the other hand, we have in $\tilde{\mathcal{M}}_{\Gamma, R\cup I}'$:
\begin{align*}
\sum_{g \in \Gamma \backslash \SL_2(\mathbf{Z})} \varphi(\{g0, g\infty\})\cdot \{gi, g\rho\}&=
\sum_{g \in \Gamma \backslash \SL_2(\mathbf{Z})} \varphi(\{g0, g\infty\})\cdot \left(\{gi, g\}+\{g, g\rho\}\right) \text{ .}
\end{align*}
We have, using the fact that $Si=i$:
\begin{align*}
\sum_{g \in \Gamma \backslash \SL_2(\mathbf{Z})} \varphi(\{g0, g\infty\})\cdot \{gi, g\}&=
\frac{1}{2}\cdot \sum_{g \in \Gamma \backslash \SL_2(\mathbf{Z})} \varphi(\{g0, g\infty\})\cdot \{gi, g\} + \varphi(\{gS0, gS\infty\})\cdot \{gSi, gS\} \\&= \frac{1}{2}\cdot \sum_{g \in \Gamma \backslash \SL_2(\mathbf{Z})} \varphi(\{g0, g\infty\})\cdot \{gS,g\} \text{ .}
\end{align*}

We have:
\begin{align*}
\sum_{g \in \Gamma \backslash \SL_2(\mathbf{Z})} \varphi(\{g0, g\infty\})\cdot \{g, g\rho\}&=
\frac{1}{3}\cdot \sum_{g \in \Gamma \backslash \SL_2(\mathbf{Z})} \varphi(\{g0, g\infty\})\cdot \{g, g\rho\} \\& + \varphi(\{g\tau0, g\tau\infty\})\cdot \{g\tau, g\tau\rho\} +\varphi(\{g\tau^20, g\tau^2\infty\})\cdot \{g\tau^2, g\tau^2\rho\} \text{ .}
\end{align*}
Since $\tau \rho = \rho$ and $ \varphi(\{g0, g\infty\}) = - \varphi(\{g\tau0, g\tau\infty\})-\varphi(\{g\tau^20, g\tau^2\infty\})$, we get
\begin{align*}
\sum_{g \in \Gamma \backslash \SL_2(\mathbf{Z})} \varphi(\{g0, g\infty\})\cdot \{g, g\rho\}&=
\frac{1}{3}\cdot \sum_{g \in \Gamma \backslash \SL_2(\mathbf{Z})} \varphi(\{g\tau0, g\tau\infty\})\cdot \{g\tau, g\} \\& + \varphi(\{g\tau^20, g\tau^2\infty\})\cdot \{g\tau^2, g\} \text{ .}
\end{align*}
Note that 
\begin{align*}
\{g\tau^2, g\} = \{g\tau^2, g\tau^2\tau\}= \{g\tau^2, g\tau^2ST\} = \{g\tau^2, g\tau^2S\} +  \{g\tau^2S, g\tau^2ST\} \text{ .}\end{align*}
Thus, we have:
\begin{align*}
\sum_{g \in \Gamma \backslash \SL_2(\mathbf{Z})} \varphi(\{g0, g\infty\})\cdot \{g, g\rho\}&=
\frac{1}{3}\cdot \sum_{g \in \Gamma \backslash \SL_2(\mathbf{Z})} \varphi(\{g\tau0, g\tau\infty\})\cdot \{g\tau, g\} \\& + \varphi(\{g0, g\infty\})\cdot \{g, gS\}+ \varphi(\{g0, g\infty\})\cdot \{gS,gST\} \\&= \frac{1}{3}\cdot \sum_{g \in \Gamma \backslash \SL_2(\mathbf{Z})} \varphi(\{g\tau0, g\tau\infty\})\cdot \{g\tau, g\} \\& + \varphi(\{g0, g\infty\})\cdot \{g, gS\}- \varphi(\{g0, g\infty\})\cdot \{g,gT\}  \text{ .}
\end{align*}
Combining the two computations above, we get:
\begin{align*}
\sum_{g \in \Gamma \backslash \SL_2(\mathbf{Z})} \varphi(\{g0, g\infty\})\cdot \{gi, g\rho\}&=\sum_{g \in \Gamma \backslash \SL_2(\mathbf{Z})} \frac{1}{6}\cdot \varphi(\{g0, g\infty\})\cdot \{gS,g\}+\frac{1}{3}\cdot  \varphi(\{g\tau0, g\tau\infty\})\cdot \{g\tau, g\} \\& - \frac{1}{3}\cdot \varphi(\{g0, g\infty\})\cdot \{g,gT\}  
\\&= \sum_{g \in \Gamma \backslash \SL_2(\mathbf{Z})} \frac{1}{6}\cdot \varphi(\{g0, g\infty\})\cdot \{gS,g\}\\&+\frac{1}{3}\cdot  \varphi(\{g\tau0, g\tau\infty\})\cdot (\{gST, gS\} +\{gS,g\})\\& - \frac{1}{3}\cdot \varphi(\{g0, g\infty\})\cdot \{g,gT\}  \\&=
\sum_{g \in \Gamma \backslash \SL_2(\mathbf{Z})} \frac{1}{6}\cdot \varphi(\{g0, g\infty\})\cdot \{gS,g\}\ +\frac{1}{3}\cdot  \varphi(\{g\tau0, g\tau\infty\})\cdot\{gS,g\} \\& - \frac{2}{3}\cdot \varphi(\{g0, g\infty\})\cdot \{g,gT\}
\end{align*}
where in the last equality, we have used the fact that
$\{gST,gS\} = \{g\tau, g\tau T^{-1}\} = -\{g\tau, g\tau T\}$.
Using again the fact that  $ \varphi(\{g0, g\infty\}) + \varphi(\{g\tau0, g\tau\infty\})+\varphi(\{g\tau^20, g\tau^2\infty\})=0$, we have:
\begin{align*}
\sum_{g \in \Gamma \backslash \SL_2(\mathbf{Z})} \varphi(\{g0, g\infty\})\cdot \{gi, g\rho\}&=\sum_{g \in \Gamma \backslash \SL_2(\mathbf{Z})} \frac{1}{6}\cdot \varphi(\{g\tau0, g\tau\infty\})\cdot \{gS,g\}-\frac{1}{6}\cdot  \varphi(\{g\tau^20, g\tau^2\infty\})\cdot \{gS, g\} \\& - \frac{2}{3}\cdot \varphi(\{g0, g\infty\})\cdot \{g,gT\}  
\\&=\sum_{g \in \Gamma \backslash \SL_2(\mathbf{Z})} \frac{1}{6}\cdot \varphi(\{g\tau0, g\tau\infty\})\cdot \{gS,g\}-\frac{1}{6}\cdot  \varphi(\{g\tau0, g\tau\infty\})\cdot \{g\tau^2S, g\tau^2\} \\& - \frac{2}{3}\cdot \varphi(\{g0, g\infty\})\cdot \{g,gT\}  
\\&=G(\varphi) \text{ .}
\end{align*}
This concludes the proof of Proposition \ref{expression_G^+_rho_i}, and thus the proof of Proposition \ref{prop_intersection_inverse_pairing} and Theorem \ref{intro_thm_Petersson}.
\end{proof}

\end{proof}
\end{proof}
\begin{rem}\label{rem_generality_conjecture_G}
Proposition \ref{expression_G^+_rho_i} is the strongest general result we were able to prove toward Properties (\ref{property_degenerate_pairing}) and (\ref{property_Hecke_pairing}). Note that the pairing $\langle \cdot,  \cdot \rangle$ must have determinant divisible by the index of $H_1(Y_{\Gamma}, \mathbf{Z})$ in $\tilde{\mathcal{M}}_{\Gamma}$, \ie $\frac{1}{d_{\Gamma}}\prod_{c \in \mathcal{C}_{\Gamma}} e_c$.
\end{rem}

\section{Relation with the generalized Jacobian}\label{generalized_cuspidal}
We remind the reader that we keep the notation of \S \ref{intro_generalized_cuspidal_motive}. We first define the two maps  $\delta_{\Gamma}^{\natural, \alg} : \mathbf{Z}[\mathcal{C}_{\Gamma}]^0 \rightarrow J_{\Gamma}^{\natural}$ and $\delta_{\Gamma}^{\natural, \an} : \mathbf{Z}[\mathcal{C}_{\Gamma}]^0 \rightarrow J_{\Gamma}^{\natural}$ such that $$\beta_{\Gamma}^{\natural} \circ \delta_{\Gamma}^{\natural, \alg} = \beta_{\Gamma}^{\natural} \circ \delta_{\Gamma}^{\natural, \an} = \delta_{\Gamma} \text{ .}$$

\subsection{Definitions and comparison result}

\subsubsection{Algebraic definition}\label{subsection_alg_def}
We first recall an alternative algebraic description of $J_{\Gamma}^{\#}$, which the author has learned in \cite[\S 2.2]{Yamazaki}. Let $F$ be a field extension of $k$ containing the compositum of the fields $k(c)$ for $c \in X_{\Gamma}\backslash Y_{\Gamma}$. Let $K$ be the function field of $X_{\Gamma} \times_k F$. If $P$ is a closed point in $X_{\Gamma} \times_k F$, let $K_P$ be the completion of $K$ at $P$ and $U_P \subset K_P^{\times}$ be the group of principal unit. Let $$\Div(X_{\Gamma}, \mathcal{C}_{\Gamma})(F):=\Div(Y_{\Gamma})(F) \oplus \bigoplus_{c \in \mathcal{C}_{\Gamma}} K_{c}^{\times}/U_c^{\times}\text{ ,}$$
where $\Div(Y_{\Gamma})(F)$ is the group of divisors of $Y_{\Gamma}$ defined over $F$. We denote by $\Div^0(X_{\Gamma}, \mathcal{C}_{\Gamma})(F)$ the kernel of the degree map $\Div(X_{\Gamma}, \mathcal{C}_{\Gamma})(F) \rightarrow \mathbf{Z}$ given by $$D \oplus (f_c \text{ modulo }U_c)_{c \in \mathcal{C}_{\Gamma}} \mapsto \deg(D) + \sum_{c \in \mathcal{C}_{\Gamma}} \ord_c(f_c) \text{ .}$$

There is a canonical map $K^{\times} \rightarrow \Div^0(X_{\Gamma}, \mathcal{C}_{\Gamma})(F)$, given by
$$f \mapsto \text{div}_{Y_{\Gamma}}(f) \oplus (f \text{ modulo }U_{c})_{c \in \mathcal{C}_{\Gamma}}$$
where $\text{div}_{Y_{\Gamma}}(f)$ is the divisor of the restriction of $f$ to $Y_{\Gamma}$.
Then there is a canonical $\Gal(F/k)$-equivariant group isomorphism
\begin{equation}\label{iso_gen_jac_alg}
 J_{\Gamma}^{\#}(F)  \xrightarrow{\sim} \Div^0(X_{\Gamma}, \mathcal{C}_{\Gamma})(F)/K^{\times} 
\end{equation}
sending the class of a divisor $D$ supported on $Y_{\Gamma}$ to the image of $(D \oplus 0)$ in  $\Div^0(X_{\Gamma}, \mathcal{C}_{\Gamma})(F)/K^{\times}$.
Under this identification, the map $J_{\Gamma}^{\#}(F) \rightarrow J_{\Gamma}(F)$ corresponds to the map $\Div^0(X_{\Gamma}, \mathcal{C}_{\Gamma})(F)/K^{\times} \rightarrow J_{\Gamma}(F)$ given by
$$D \oplus (f_c \text{ modulo }U_c)_{c \in \mathcal{C}_{\Gamma}} \mapsto  \text{class of the divisor }D + \sum_{c \in \mathcal{C}_{\Gamma}} \ord_c(f_c) \cdot (c) \text{ .}$$

We now define the map $\delta_{\Gamma}^{\natural, \alg} : \mathbf{Z}[\mathcal{C}_{\Gamma}]^0 \rightarrow J_{\Gamma}^{\natural}(\mathbf{C})$. We apply the discussion above to $F = \mathbf{C}$. For any $c \in \mathcal{C}_{\Gamma}$, let $\frac{p_c}{q_c} \in \mathbf{Q}$ such that $c = \Gamma \cdot \frac{p_c}{q_c}$ and let $g_c \in \SL_2(\mathbf{Z})$ such that $g_c( \frac{p_c}{q_c}) = \infty$. Then the function $t_c : z\mapsto j(\frac{g_c z}{e_c})^{-1}$ of the upper-half plane induces a uniformizer at $c$, still denoted by $t_c \in K_c^{\times}$ (we recall that $e_c$ is the width of $c$). The map 
$\delta_{\Gamma}^{\natural, \alg}$ sends a degree zero divisor $\sum_{c \in \mathcal{C}_{\Gamma}} n_c \cdot (c)$ to the image of $0 \oplus (t_c^{n_c} \text{ modulo }U_c)_{c \in \mathcal{C}_{\Gamma}}$ in $J_{\Gamma}^{\natural}(\mathbf{C})$ via the identification (\ref{iso_gen_jac_alg}). 

The following result motivates the definition of $J_{\Gamma}^{\natural}$.
\begin{prop}\label{beta_alg_lemma} 
\begin{enumerate}
\item \label{i_commutative_beta}We have $\beta_{\Gamma}^{\natural} \circ \delta_{\Gamma}^{\natural, \alg} = \delta_{\Gamma}$.
\item \label{ii_well_defined_beta}The map $\delta_{\Gamma}^{\natural, \alg}$ is independant of the choice of the elements $\frac{p_c}{q_c}$ and of $g_c$.
\item \label{iii_rationality_alg_beta} The map $\delta_{\Gamma}^{\natural, \alg}$ takes values in $J_{\Gamma}^{\natural}(\mathbf{Q}(\zeta_N))$.
\item \label{iv_galois_equivariance} Assume that $\Gamma = \Gamma_0(N)$ or $\Gamma = \Gamma_1(N)$ (in which case we have $k=\mathbf{Q}$). Then the map $\delta_{\Gamma}^{\natural, \alg}$ is $\Gal(\mathbf{Q}(\zeta_N)/\mathbf{Q})$-equivariant.
\end{enumerate}
\end{prop}
\begin{proof}
Point (\ref{i_commutative_beta}) is straightforward. We prove point (\ref{ii_well_defined_beta}). First, fix the choice of $\frac{p_c}{q_c}$. Recall that locally at $\infty$, we have $j(z)^{-1} \sim q(z) = e^{2i\pi z}$ modulo principal units. Thus, if we replace $g_c$ by $T^ng_c$ for some $n \in \mathbf{Z}$, then $t_c$ is replaced by $e^{\frac{2i \pi n}{e_c}}\cdot t_c$. By construction, the image of $\prod_{c \in \mathcal{C}_{\Gamma}} \mu_{e_c}$ in $J_{\Gamma}^{\natural}(\mathbf{C})$ is trivial, so our construction does not depend on the choice of the elements $g_c$. The fact that $t_c$ does not depend on the choice of $\frac{p_c}{q_c}$ is obvious. Point (\ref{iii_rationality_alg_beta}) follows from the following result.
\begin{lem}
For all $c \in \mathcal{C}_{\Gamma}$, $t_c$ is defined over $\mathbf{Q}(\zeta_N)$.
\end{lem}
\begin{proof}
Let $c \in \mathcal{C}_{\Gamma}$ and $\Gamma' = \Gamma \cap \Gamma(e_c)$, where $\Gamma(e_c)$ is the principal congruence subgroup of level $e_c$. Note that $\Gamma(N)$ is a subgroup of $\Gamma'$ since $e_c$ divides $N$ and $\Gamma(N) \subset \Gamma$. Let $c' = \Gamma' \cdot \frac{p_c}{q_c}$; it is a cusp of $X_{\Gamma'}$. The covering $X_{\Gamma'} \rightarrow X_{\Gamma}$ is unramified at $c'$ (where $X_{\Gamma'}$ and $X_{\Gamma}$ are considered over $\mathbf{Q}(\zeta_N)$), since the ramification index of both $X_{\Gamma'} \rightarrow X_{\SL_2(\mathbf{Z})}$ and $X_{\Gamma} \rightarrow X_{\SL_2(\mathbf{Z})}$ is $e_c$ at $c'$ and $c$ respectively. The function $j(\frac{g_c z}{e_c})$ is modular for the group $\Gamma'$, and so it suffices to prove that it is defined over $\mathbf{Q}(\zeta_N)$. This is true at the level of the modular curve $X_{\Gamma(N)}$, so it is true for $X_{\Gamma'}$ since the covering $X_{\Gamma(N)} \rightarrow X_{\Gamma'}$ is defined over $\mathbf{Q}(\zeta_N)$.
\end{proof}
We now prove point (\ref{iv_galois_equivariance}).  The action of $\Gal(\mathbf{Q}(\zeta_N)/\mathbf{Q})$ on $\mathbf{Z}[\mathcal{C}_{\Gamma}]^0$ comes from its action on $\mathcal{C}_{\Gamma}$. If $c \in \mathcal{C}_{\Gamma}$ and $g \in\Gal(\mathbf{Q}(\zeta_N)/\mathbf{Q})$, we denote by $g(c) \in \mathcal{C}_{\Gamma}$ the cusp corresponding to the action of $g$ on $c$. Under (\ref{iso_gen_jac_alg}), the action of $\Gal(\mathbf{Q}(\zeta_N)/\mathbf{Q})$ on $\Div^0(X_{\Gamma}, \mathcal{C}_{\Gamma})(F)/K^{\times}$ is given as follows. The action of $\Gal(\mathbf{Q}(\zeta_N)/\mathbf{Q})$ on $K^{\times}$ and $\Div(Y_{\Gamma})(F)$ is the obvious one. If $g \in \Gal(\mathbf{Q}(\zeta_N)/\mathbf{Q})$, then $g$ acts on $\bigoplus_{c \in \mathcal{C}_{\Gamma}} K_{c}^{\times}/U_c^{\times}$ by $g \cdot (f_c \text{ modulo }U_c)_{c \in \mathcal{C}_{\Gamma}} = (g\cdot f_c)_{g(c), c \in \mathcal{C}_{\Gamma}}$. To prove (\ref{iv_galois_equivariance}), it suffices to prove that for all $c \in \mathcal{C}_{\Gamma}$ and $g \in \Gal(\mathbf{Q}(\zeta_N)/\mathbf{Q})$, we have $(g\cdot f_c)^{e_{g(c)}} \equiv f_{g(c)}^{e_{g(c)}} \text{ (modulo }U_{g(c)}\text{)}$. We have $f_{g(c)}^{e_{g(c)}} \equiv j(z) \text{ (modulo }U_{g(c)}\text{)}$. Since $e_{g(c)} = e_c$, we have $(g\cdot f_c)^{e_{g(c)}} = g\cdot (f_c^{e_c}) \equiv g\cdot j(z) \equiv j(z)  \text{ (modulo }U_{g(c)}\text{)}$. We have used the fact that $j : X_{\Gamma} \rightarrow X_{\SL_2(\mathbf{Z})}$ is defined over $\mathbf{Q}$.
\end{proof}

\subsubsection{Analytic definition}

There is a well-known analytic uniformization of the generalized Jacobian $J_{\Gamma}^{\#}(\mathbf{C})$. Namely, there is an isomorphism of complex abelian varieties (called the generalized Abel-Jacobi map)
$$\AJ_{\Gamma} : J_{\Gamma}^{\#}(\mathbf{C}) \xrightarrow{\sim} \Hom_{\mathbf{C}}(M_2(\Gamma), \mathbf{C})/H_1(Y_{\Gamma}, \mathbf{Z}) \text{ ,}$$
sending the class of a divisor $D$ to the class of the morphism $\varphi : M_2(\Gamma) \rightarrow \mathbf{C}$ given by $f \mapsto \int_{\gamma} 2i\pi f(z)dz$ where $\gamma$ is a $1$-chain in $Y_{\Gamma}$ whose boundary is $D$.

Recall that we have a group homomorphism $\tilde{p}_{\Gamma} : \tilde{\mathcal{M}}_{\Gamma} \rightarrow \Hom_{\mathbf{C}}(M_2(\Gamma),\mathbf{C})$. By Proposition \ref{rank_computation} (\ref{rank_computation_2}) and (\ref{rank_computation_3}), the compositum
$$\tilde{\mathcal{M}}_{\Gamma} \xrightarrow{\tilde{p}_{\Gamma}} \Hom_{\mathbf{C}}(M_2(\Gamma), \mathbf{C}) \rightarrow \Hom_{\mathbf{C}}(M_2(\Gamma), \mathbf{C})/H_1(Y_{\Gamma}, \mathbf{Z}) \xrightarrow{\AJ_{\Gamma}^{-1}} J_{\Gamma}^{\#}(\mathbf{C}) \rightarrow J_{\Gamma}^{\natural}(\mathbf{C})$$
factors through $\partial_{\Gamma} : \tilde{\mathcal{M}}_{\Gamma} \rightarrow \mathbf{Z}[\mathcal{C}_{\Gamma}]^0$, so we have defined a map
$$\delta_{\Gamma}^{\natural, \an} : \mathbf{Z}[\mathcal{C}_{\Gamma}]^0 \rightarrow J_{\Gamma}^{\natural}(\mathbf{C}) \text{ .}$$
By Theorem \ref{Hecke_Q}, the map $\delta_{\Gamma}^{\natural, \an}$ is $\mathbb{T}$-equivariant.
\begin{lem}\label{beta_an_lemma}
We have $\beta_{\Gamma}^{\natural} \circ \delta_{\Gamma}^{\natural, \an} = \delta_{\Gamma}$.
\end{lem}
\begin{proof}
This follows from the fact that for any cuspidal modular form $f$ and any $g \in \SL_2(\mathbf{Z})$, we have $\mathcal{S}(g)(f) = 2i\pi\int_{\infty}^{g(\infty)} f(z)dz$.
\end{proof}

\subsubsection{Comparison between the algebraic and analytic definitions}\label{subsection_comparison}
We now prove Theorem \ref{intro_main_thm_comparison}.
\begin{proof}
We first prove point (\ref{comparison_i}). By Proposition \ref{beta_alg_lemma} (\ref{i_commutative_beta}) and Lemma \ref{beta_an_lemma}, the map $\delta_{\Gamma}^{\natural, \alg} - \delta_{\Gamma}^{\natural, \an}$ takes values in the image of $\prod_{c \in \mathcal{C}_{\Gamma}} \mathbf{C}^{\times}/\mu_{N}$ inside $J_{\Gamma}^{\natural}(\mathbf{C})$. If $(\lambda_c)_{c \in \mathcal{C}_{\Gamma}} \in  \mathbf{C}^{\times}$, we abusively use the same notation for its image in $J_{\Gamma}^{\natural}(\mathbf{C})$.

Recall (\cf paragraph \ref{subsection_alg_def}) that for any $c \in \mathcal{C}_{\Gamma}$, we write $c = \Gamma \cdot \frac{p_c}{q_c}$ and we fix some matrix $g_c \in \SL_2(\mathbf{Z})$ such that $g_c(\frac{p_c}{q_c}) = \infty$. We then let $t_c(z) = j(\frac{g_c(z)}{e_c})^{-1}$ and $u_c(z) = e^{\frac{2i\pi g_c(z)}{e_c}}$. Note that $u_c(z) \sim t_c(z)$ near $\frac{p_c}{q_c}$.

\begin{lem}\label{main_thm_comparison_lemma_1}
Let $D = \sum_{c \in \mathcal{C}_{\Gamma}} n_c \cdot (c) \in \mathbf{Z}[\mathcal{C}_{\Gamma}]^0$, $m_D$ be the order of the image of $D$ in $J_{\Gamma}$ and $u$ be a modular unit with divisor $m_D\cdot D$. We have:
$$m_D \cdot  \delta_{\Gamma}^{\natural, \alg}(D) = ((\frac{t_c^{m_D\cdot n_c}}{u})(c))_{c \in \mathcal{C}_{\Gamma}} \text{ .}$$
\end{lem}
\begin{proof}
By definition, $m_D \cdot  \delta_{\Gamma}^{\natural, \alg}(D)$ is the image of $0 \oplus (t_c^{m_D\cdot n_c})_{c\in \mathcal{C}_{\Gamma}}\in \Div^0(X_{\Gamma}, \mathcal{C}_{\Gamma})(\mathbf{C})$ in $J_{\Gamma}^{\natural}(\mathbf{C})$. Thus, $m_D \cdot  \delta_{\Gamma}^{\natural, \alg}(D)$ is also the image of $0 \oplus ((\frac{t_c^{m_D\cdot n_c}}{u})(c))_{c \in \mathcal{C}_{\Gamma}} \in \Div^0(X_{\Gamma}, \mathcal{C}_{\Gamma})(\mathbf{C})$ in $J_{\Gamma}^{\natural}(\mathbf{C})$. This concludes the proof of Lemma \ref{main_thm_comparison_lemma_1}.
\end{proof}

\begin{lem}\label{main_thm_comparison_lemma_2}
Let $D = \sum_{c \in \mathcal{C}_{\Gamma}} n_c \cdot (c) \in \mathbf{Z}[\mathcal{C}_{\Gamma}]^0$. Lift $\delta_{\Gamma}^{\natural, \an}(D)$ to an element $ \varphi_D \in \Hom_{\mathbf{C}}(M_2(\Gamma), \mathbf{C})$. Let $v$ be a modular unit of level $\Gamma$, and $d\log(v) \in M_2(\Gamma)$ denote the Eisenstein series $\frac{v'}{v}$. The quantity $\varphi_D(\frac{1}{2i\pi}d\log(v))$ is canonical in $\mathbf{C}/\frac{2i \pi}{N} \mathbf{Z}$. We have the following equality in $\mathbf{C}^{\times}/\mu_{N}$:
$$\exp(\varphi_D(\frac{1}{2i\pi}d\log(v))) = \prod_{c \in \mathcal{C}_{\Gamma}} (\frac{v}{t_c^{\ord_c(v)}})^{n_c}(c)\text{ .}$$
\end{lem}
\begin{proof}
By construction of $\delta_{\Gamma}^{\natural, \an}$, we can choose $\varphi_D$ so that for any $f \in M_2(\Gamma)$, we have
$$\varphi_D(f) = \sum_{c \in \mathcal{C}_{\Gamma}} n_c \cdot \mathcal{S}(g_c^{-1})(f) \text{ ,}$$
where we recall that $$\mathcal{S}(g)(f) = 2i\pi \cdot \int_{z_0}^{g(z_0)} f(z)dz - 2i\pi\cdot  z_0\cdot (a_0(f\mid g)-a_0(f)) + 2i\pi\cdot \int_{z_0}^{i\infty} \left((f\mid g)(z)-a_0(f\mid g)\right) - \left(f(z) - a_0(f)\right) dz $$
for any $g \in \SL_2(\mathbf{Z})$ and $f \in M_2(\Gamma)$.

We can simplify this formula when $f = \frac{1}{2i\pi}\frac{v'}{v}$, where we view $v$ as a function on the upper-half plane $\mathfrak{h}$. For any $c \in \mathcal{C}_{\Gamma}$, let $v_c = \frac{v}{u_c^{\ord_c(v)}} \circ g_c^{-1}$. Fix a logarithm of $v_c$ on $\mathfrak{h}$, denoted by $\log(v_c)$ (so by definition we have $\exp(\log(v_c))=v_c$ on $\mathfrak{h}$). We also fix a logarithm $\log(v)$ of $v$ on $\mathfrak{h}$; note that $\log(v)' = 2i\pi f$. For all $c \in \mathcal{C}_{\Gamma}$, we easily see that $a_0(f\mid g_c^{-1}) = \frac{1}{e_{c}}\cdot \ord_{c}(v)$ and $2i \pi \cdot (f\mid g_c^{-1})-2i\pi \cdot a_0(f\mid g_c^{-1}) = \log(v_c)'$. Thus, we have:
\begin{align*}
\mathcal{S}(g_c^{-1})(f) &= \log(v)(g_c^{-1}(z_0)) - \log(v)(z_0) + 2i\pi \cdot z_0\cdot (\frac{1}{e_c}\cdot \ord_c(v)-\frac{1}{e_{\Gamma \infty}}\cdot \ord_{\Gamma \infty}(v)) \\& + \log(v_c)(i\infty)-\log(v_c)(z_0)-\log(v)(i\infty) + \log(v)(z_0) \text{ .}
\end{align*}
Since $\sum_{c \in \mathcal{C}_{\Gamma}} n_c = 0$, we have
\begin{align*}
\sum_{c \in \mathcal{C}_{\Gamma}} n_c\cdot \mathcal{S}(g_c^{-1})(f) = \sum_{c \in \mathcal{C}_{\Gamma}} n_c\cdot \left( \log(v)(g_c^{-1}(z_0))  + 2i\pi \cdot z_0\cdot \frac{1}{e_c}\cdot \ord_c(v) + \log(v_c)(i\infty)-\log(v_c)(z_0) \right) \text{ .}
\end{align*}
Thus, we have in $\mathbf{C}^{\times}/\mu_N$:
\begin{align*}
\exp(\varphi_D(\frac{1}{2i\pi}d\log(v))) &= \prod_{c \in \mathcal{C}_{\Gamma}} \left(v(g_{c}^{-1}(z_0))\cdot \exp(2i\pi \cdot z_0\cdot \frac{1}{e_c}\cdot \ord_c(v))\cdot v_c(i\infty)\cdot v_c(z_0)^{-1}\right)^{n_c} \\&
=  \prod_{c \in \mathcal{C}_{\Gamma}} v_c(i\infty)^{n_c} \\&= \prod_{c \in \mathcal{C}_{\Gamma}} (\frac{v}{t_c^{\ord_c(v)}})^{n_c}(c) \text{ .}
\end{align*}
This concludes the proof of Lemma \ref{main_thm_comparison_lemma_2}.
\end{proof}
\begin{lem}\label{main_thm_comparison_lemma_3}
Let $(\lambda_c)_{c \in \mathcal{C}_{\Gamma}} \in \prod_{c \in \mathcal{C}_{\Gamma}} \mathbf{C}^{\times}$. Let $\varphi \in \Hom_{\mathbf{C}}(M_2(\Gamma),\mathbf{C})$ whose class in $J_{\Gamma}^{\#}(\mathbf{C})$ corresponds to $(\lambda_c)_{c \in \mathcal{C}_{\Gamma}}$ via the generalized Abel-Jacobi isomorphism $AJ_{\Gamma}$. For any modular unit $v$ of $X_{\Gamma}$, we have
$$\exp(\varphi(\frac{1}{2i\pi}d\log(v))) = \prod_{c \in \mathcal{C}_{\Gamma}} \lambda_c^{-\ord_c(v)} \text{ .}$$
\end{lem}
\begin{proof}
Let $f$ be a meromorphic function on $X_{\Gamma}$ whose divisor $\text{div}(f)$ is supported on $Y_{\Gamma}$, and such that for all $c \in \mathcal{C}_{\Gamma}$ we have $f(c) = \lambda_c$. We write $\text{div}(f) = \sum_{i=1}^{m} m_i(Q_i)$ for some $Q_i\in Y_{\Gamma}$ and $m_i \in \mathbf{Z}$. By the construction of $J_{\Gamma}^{\#}$ in paragraph \ref{subsection_alg_def}, we have in $\mathbf{C}/2i\pi \mathbf{Z}$:
$$\varphi(\frac{1}{2i\pi}d\log(v)) = -\int_{\text{div}(f)} d\log(v)(z)dz \text{ ,}$$
where the integral is over any $1$-chain with boundary $\text{div}(f)$. The following computation was given to us by the mathoverflow user abx in his answer \cite{MO_abel-jacobi}. Such a $1$-chain can be written as a linear combination of paths $\gamma_j$ from $Q_{j_1}$ to $Q_{j_2}$ for some indexes $j_1$ and $j_2$. By choosing a determination of $\log(v)$ along the path $\gamma_j$, we get:
$$\exp(\int_{\gamma_j} d\log(v)) = \frac{v(Q_{j_1})}{v(Q_{j_2})} \text{ .}$$
Consequently, we have:
$$\exp(\int_{\text{div}(f)} d\log(v)(z)dz) = \prod_{i=1}^m v(Q_i)^{m_i} \text{ .}$$
By Weil's reciprocity law, we have:
$$ \prod_{i=1}^m v(Q_i)^{m_i} = \prod_{c \in \mathcal{C}_{\Gamma}} f(c)^{\ord_c(v)} = \prod_{c \in \mathcal{C}_{\Gamma}} \lambda_c^{\ord_c(v)} \text{ .}$$
This concludes the proof of Lemma \ref{main_thm_comparison_lemma_3}.
\end{proof}

Let $u$ be a modular unit of $Y_{\Gamma}$. Using the generalized Abel-Jacobi isomorphism, lift $\delta_{\Gamma}^{\natural, \alg}(\text{div}(u))$ (resp. $\delta_{\Gamma}^{\natural, \an}(\text{div}(u))$)  to an element $\varphi_u^{\alg}$ (resp. $\varphi_u^{\an}$) of $\Hom_{\mathbf{C}}(M_2(\Gamma), \mathbf{C})$.
By Lemmas \ref{main_thm_comparison_lemma_1} and \ref{main_thm_comparison_lemma_3}, for any modular unit $v$ we have in $\mathbf{C}^{\times}/\mu_{N}$:
$$\exp(\varphi_u^{\alg}(\frac{1}{2i\pi}d\log(v))) = \prod_{c \in \mathcal{C}_{\Gamma}} (\frac{u}{t_c^{\ord_c(u)}}(c))^{\ord_c(v)} \text{ .}$$
On the other hand, by Lemma \ref{main_thm_comparison_lemma_2} we have in $\mathbf{C}^{\times}/\mu_{N}$:
$$\exp(\varphi_u^{\an}(\frac{1}{2i\pi}d\log(v))) = \prod_{c \in \mathcal{C}_{\Gamma}} (\frac{v}{t_c^{\ord_c(v)}}(c))^{\ord_c(u)} \text{ .}$$
By Weil reciprocity law, we have in $\mathbf{C}^{\times}/\mu_{N}$:
$$\exp(\varphi_u^{\alg}(\frac{1}{2i\pi}d\log(v)))=\exp(\varphi_u^{\an}(\frac{1}{2i\pi}d\log(v))) \text{ .}$$
By Lemma \ref{main_thm_comparison_lemma_3}, the image of $\varphi_u^{\alg}-\varphi_u^{\an}$ in $J_{\Gamma}^{\natural}(\mathbf{C})$ is equal to the image of $(\lambda_c)_{c \in \mathcal{C}_{\Gamma}}$ in $J_{\Gamma}^{\natural}(\mathbf{C})$ for some $\lambda_c \in \mathbf{C}^{\times}/\mu_{N}$ such that for all modular unit $v$, we have in $\mathbf{C}^{\times}/\mu_{N}$:
$$\prod_{c \in \mathcal{C}_{\Gamma}} \lambda_c^{\ord_c(v)}=1 \text{ .}$$
In particular, there exists $\lambda \in \mathbf{C}^{\times}/\mu_{N}$ such that for all $c \in \mathcal{C}_{\Gamma}$ we have $\lambda_c^n=\lambda$.
This proves that for all modular unit $u$, we have
$$n\cdot \delta_{\Gamma}^{\natural, \alg}(\text{div}(u)) = n\cdot \delta_{\Gamma}^{\natural, \an}(\text{div}(u)) \text{ .}$$
In particular, for all $D \in \mathbf{Z}[\mathcal{C}_{\Gamma}]^0$, we have 
$$n^2\cdot \delta_{\Gamma}^{\natural, \alg}(D) = n^2\cdot \delta_{\Gamma}^{\natural, \an}(D) \text{ .}$$
This concludes the proof of point (\ref{comparison_i}). 

We now prove point (\ref{comparison_ii}). The map $\delta_{\Gamma}^{\natural, \alg} - \delta_{\Gamma}^{\natural, \an}$ takes values in the image of $\prod_{c \in \mathcal{C}_{\Gamma}} \mathbf{C}^{\times}/\mu_{N}$ in $J_{\Gamma}^{\natural}(\mathbf{C})$, and has finite order by (\ref{comparison_i}). The map $\delta_{\Gamma}^{\natural, \alg} - \delta_{\Gamma}^{\natural, \an}$ is invariant by the action of the complex conjugation by assumption, by Proposition \ref{beta_alg_lemma} (\ref{iv_galois_equivariance}) and by Proposition \ref{equivariance_complex_conjug_period}. Since $N$ is assumed to be odd, an element of finite order in $\mathbf{C}^{\times}/\mu_{N}$ fixed by the complex conjugation is equal to $\pm1$. This concludes the proof of Theorem \ref{intro_main_thm_comparison}.
\end{proof}

\begin{rems}\label{remarks_after_comparison}
\begin{enumerate}
\item\label{remarks_after_comparison_i} By Theorem \ref{intro_main_thm_comparison} (\ref{comparison_i}), the map $\delta_{\Gamma}^{\natural, \alg} : \mathbf{Z}[\mathcal{C}_{\Gamma}]^0 \rightarrow J_{\Gamma}^{\natural}(\mathbf{C})$ is injective if and only if $\delta_{\Gamma}^{\natural, \an}$ is injective. We easily see that $\delta_{\Gamma}^{\natural, \an}$ is injective if $\tilde{p}_{\Gamma} \otimes \mathbf{R}$ is an isomorphism. By Theorem \ref{thm_iso_tilde_p}, this is the case if $\Gamma$ has prime power level.
\item It would be interesting to remove the hypotheses of Theorem \ref{intro_main_thm_comparison} (\ref{comparison_ii}). It is not clear to us whether they are necessary.
\item\label{remarks_after_comparison_ii} Assume that $\Gamma$ satisfies the assumptions of Theorem \ref{intro_main_thm_comparison} (\ref{comparison_ii}). Let $J_{\Gamma}^{\flat}$ be the semi-abelian variety over $k$ defined by $J_{\Gamma}^{\flat} = J_{\Gamma}^{\natural}/\prod_{c \in \mathcal{C}_{\Gamma}} \Res_{k(c)/k}(\mu_2)$. The map $\delta_{\Gamma}^{\natural, \alg}$ (or equivalently $\delta_{\Gamma}^{\natural, \an}$) gives a natural group homomorphism $\delta_{\Gamma}^{\flat} : \mathbf{Z}[\mathcal{C}_{\Gamma}]^0 \rightarrow J^{\flat}_{\Gamma}(\mathbf{Q}(\zeta_N))$. If $\Gamma = \Gamma_1(N)$ or $\Gamma = \Gamma_0(N)$ (with $N$ necessarily an odd prime given our assumptions), then by Theorem \ref{intro_main_thm_comparison} (\ref{comparison_ii}) and Theorem \ref{Hecke_Q}, the map $\delta_{\Gamma}^{\flat}$ is $\mathbb{T}$-equivariant. 
\end{enumerate}
\end{rems}

\subsection{Applications to the modular curve $X_0(p)$}\label{paragraph_applications_Gamma_0(p)}
Let $p$ be an odd prime and $\Gamma = \Gamma_0(p)$. Let $w_p$ be the Atkin-Lehner involution acting on $X_0(p) = X_{\Gamma_0(p)}$. Note that we have $\mathcal{C}_{\Gamma_0(p)} = \{\Gamma_0(p) 0, \Gamma_0(p) \infty\}$. These two cusps are defined over $\mathbf{Q}$ and there is a canonical uniformizer at $\Gamma_0(p) \infty$ (resp. $\Gamma_0(p) 0$) given by $j^{-1}$ (resp. $(j \circ w_p)^{-1}$). The order of $(\Gamma_0(p) \infty) - (\Gamma_0(p) 0)$ in $J_{\Gamma_0(p)}$ is $n:=\frac{p-1}{d}$ where $d = \gcd(p-1,12)$. More precisely, $n\cdot((\Gamma_0(p) \infty) - (\Gamma_0(p) 0))$ is the divisor of the modular unit $$u(z) := \left(\frac{\Delta(pz)}{\Delta(z)}\right)^{\frac{1}{d}}$$ where $\Delta \in S_{12}(\SL_2(\mathbf{Z}))$ is Ramanujan's Delta function and $z \in \mathfrak{H}$.

We have an exact sequence of semi-abelian schemes over $\mathbf{Q}$:
$$1 \rightarrow \mathbf{G}_m \rightarrow \mathbf{G}_m \times \mathbf{G}_m \rightarrow J_{\Gamma_0(p)}^{\#} \rightarrow J_{\Gamma_0(p)} \rightarrow 1$$
which induces (by Hilbert 90) an exact sequence of abelian groups
$$1 \rightarrow \mathbf{Q}^{\times} \rightarrow \mathbf{Q}^{\times}\times \mathbf{Q}^{\times}  \rightarrow J_{\Gamma_0(p)}^{\#}(\mathbf{Q}) \rightarrow J_{\Gamma_0(p)}(\mathbf{Q}) \rightarrow 1 \text{ .}$$
By convention, the first copy of $\mathbf{G}_m$ in $\mathbf{G}_m \times \mathbf{G}_m$ corresponds to the cusp $\Gamma_0(p) \infty$, while the second copy corresponds to $\Gamma_0(p) 0$. Although we shall not use it, in contrast with the case of $J_{\Gamma_0(p)}(\mathbf{Q})$, the torsion subgroup of $J_{\Gamma_0(p)}^{\#}(\mathbf{Q})$ is the image of $\{\pm 1\} \times \{\pm 1\}$ \cite[Theorem 1.1.3]{Yamazaki}.

\subsubsection{The generalized cuspidal $1$-motive}
In this case, we can be a little more precise both on the algebraic and analytic sides. 
Let $$\delta_{\Gamma_0(p)}^{\#, \alg} : \mathbf{Z}[\mathcal{C}_{\Gamma_0(p)}]^0 \rightarrow J_{\Gamma_0(p)}^{\#}(\mathbf{Q})$$ be the map sending $(\Gamma_0(p) \infty)-(\Gamma_0(p) 0)$ to the image of $0 \oplus (j^{-1} \text{ modulo } U_{\Gamma_0(p) \infty}, j \circ w_p \text{ modulo } U_{\Gamma_0(p) 0})$ in $J_{\Gamma_0(p)}^{\#}(\mathbf{Q})$ via (\ref{iso_gen_jac_alg}).
Let  $$\delta_{\Gamma_0(p)}^{\#, \an} : \mathbf{Z}[\mathcal{C}_{\Gamma_0(p)}]^0 \rightarrow J_{\Gamma_0(p)}^{\#}(\mathbf{C})$$ be the map sending $(\Gamma_0(p) \infty)-(\Gamma_0(p) 0)$ to the image of $(f \mapsto -L(f,1)) \in \Hom_{\mathbf{C}}(M_2(\Gamma_0(p)), \mathbf{C})$ via the Abel-Jacobi map $AJ_{\Gamma_0(p)}$ (where $L(f,1)$ is the special value at $s=1$ of the L-function of $f$). 
We now prove Theorem \ref{intro_main_comparison_Gamma_0_case}.

\begin{proof}
Proof of (\ref{main_comparison_Gamma_0_case_i}). Note that $\mathbf{Z}[\mathcal{C}_{\Gamma_0(p)}]$ is annihilated by Mazur's \textit{Eisenstein ideal} $I$, generated by the Hecke operators $T_{\ell}-\ell-1$ for primes $\ell \neq p$ and by $U_p-1$. Thus, we only need to show that $\delta_{\Gamma_0(p)}^{\#, \alg}$ and $\delta_{\Gamma_0(p)}^{\#, \an}$ are annihilated by $I$. 

We first consider the map $\delta_{\Gamma_0(p)}^{\#, \alg}$. Let $\ell \neq p$ be a prime number, and consider the usual double coset $\Gamma_0(p) \begin{pmatrix}1 & 0\\ 0 & p \end{pmatrix} \Gamma_0(p) =\Gamma_0(p)g_{\infty} \bigcup_{i=0}^{\ell-1} \Gamma_0(p)\cdot g_i$ where $g_i = \begin{pmatrix} 1 & i \\ 0 & \ell \end{pmatrix}$ and $g_{\infty} = \begin{pmatrix} \ell & 0 \\ 0 & 1 \end{pmatrix}$.  Since in our case the Hecke operators are self-dual and fix the two cusps, the action of $T_{\ell}$ on the image of a point $0 \oplus (f_c)_{c \in \mathcal{C}_{\Gamma_0(p)}} \in \Div^0(X_0(p),\mathcal{C}_{\Gamma_0(p)})$ in $J_{\Gamma_0(p)}^{\#}$ is the image of $$0 \oplus (\prod_{i \in \{0, ..., \ell-1, \infty\}}f_c\circ g_i)_{c \in \mathcal{C}_{\Gamma_0(p)}} \text{ .}$$ If $\ell > 2$, then we are done since 
$$\prod_{i \in \{0, ..., \ell-1, \infty\}}(j \circ g_i)(z) \sim e^{2i\pi \ell z}\cdot \prod_{i=0}^{\ell-1} e^{2i \pi \frac{z+i}{\ell}} = (e^{2i \pi z})^{\ell+1} \sim j(z)^{\ell+1}$$
and similarly for $j\circ w_p$. If $\ell=2$, we find
$$\prod_{i \in \{0, ..., \ell-1, \infty\}}(j \circ g_i)(z) \sim -j(z)^{\ell+1}$$ and 
$$\prod_{i \in \{0, ..., \ell-1, \infty\}}(j\circ w_p \circ g_i)(z) \sim -(j\circ w_p)(z)^{\ell+1} \text{ .}$$
But the image of $0 \oplus (f_c)_{c \in \mathcal{C}_{\Gamma_0(p)}}$ in $J_{\Gamma_0(p)}^{\#}$ is the same as the image of $0 \oplus (\lambda \cdot f_c)_{c \in \mathcal{C}_{\Gamma_0(p)}}$ for any scalar $\lambda$. Thus, we have proved that $T_{\ell}-\ell-1$ annihilates $\delta_{\Gamma_0(p)}^{\#}$. The case $\ell=p$ is similar.

We now consider the map $\delta_{\Gamma_0(p)}^{\#, \an}$. Let $T$ be a Hecke operator in the Eisenstein ideal $I$.  We need to show that the map $f \mapsto -L(T f, 1) \in \Hom_{\mathbf{C}}(M_2(\Gamma_0(p)), \mathbf{C})$ comes from the integration of an element in $H_1(Y_0(p), \mathbf{Z})$. If $f=E$ is the unique Eisenstein series of $M_2(\Gamma_0(p))$, then $T f =0$ by definition. If $f$ is a cusp form, then there is a unique element $e$ in $H_1(Y_0(p), \mathbf{Q})^+$ (the so-called \textit{winding element} of Mazur) such that $-L(f,1) = \int_{e} 2i\pi f(z)dz$. We know that $Te \in H_1(Y_0(p), \mathbf{Z})^+$. In particular, we have $\int_{Te} 2i\pi E(z)dz = 0$ (this is true for any cycle in $H_1(Y_0(p), \mathbf{Z})^+$ since $2i\pi E(z)dz$ is the logarithmic derivative of the modular unit $u(z)$). Thus, the map $f \mapsto -L(T f, 1)$ of $\Hom_{\mathbf{C}}(M_2(\Gamma_0(p)), \mathbf{C})$ coincides with the map $f \mapsto \int_{Te} 2i\pi f(z)dz$, which concludes the proof of (\ref{main_comparison_Gamma_0_case_i}) since $Te \in H_1(Y_0(p), \mathbf{Z})$.

The proof of (\ref{main_comparison_Gamma_0_case_ii}) and (\ref{main_comparison_Gamma_0_case_iii}) is essentially a particular case of the proof of Theorem \ref{intro_main_thm_comparison}, but we give the details for the convenience of the reader.

Proof of (\ref{main_comparison_Gamma_0_case_ii}). By definition, $n \cdot \delta_{\Gamma_0(p)}^{\#, \alg}((\Gamma_0(p) \infty)-(\Gamma_0(p) 0))$ is the image of $((\frac{j^{-n}}{u})(\Gamma_0(p) \infty), (\frac{(j\circ w_p)^{n}}{u})(\Gamma_0(p) 0))$ in $J_{\Gamma_0(p)}^{\#}(\mathbf{Q})$. By the $q$-expansion product formula for $u$, we see that $(\frac{j^{-n}}{u})(\Gamma_0(p) \infty) =1$. On the other hand, it is well-known that $u \circ w_p = p^{-\frac{12}{d}} \cdot u^{-1}$ (this is noted for instance in \cite[p. 471]{deShalit}). Thus, $(\frac{(j\circ w_p)^{n}}{u})(\Gamma_0(p) 0) = (\frac{j^{n}}{u\circ w_p})(\Gamma_0(p) \infty) = p^{\frac{12}{d}}\cdot (\frac{j^{-n}}{u})^{-1}(\Gamma_0(p) \infty) = p^{\frac{12}{d}}$.

Proof of (\ref{main_comparison_Gamma_0_case_iii}). It suffices to prove that $n \cdot  \delta_{\Gamma_0(p)}^{\#, \alg} = n \cdot  \delta_{\Gamma_0(p)}^{\#, \an}$. The element $n \cdot  \delta_{\Gamma_0(p)}^{\#, \an}((\Gamma_0(p) \infty)-(\Gamma_0(p) 0))$ of $J_{\Gamma_0(p)}^{\#}(\mathbf{C})$ is the image of $(\lambda_{\infty}, \lambda_0) \in \mathbf{R}^{\times} \times \mathbf{R}^{\times}$. By Lemma \ref{main_thm_comparison_lemma_3}, we have $$\lambda_{\infty}^{n} \cdot \lambda_0^{-n} = \exp(n\cdot L(d\log(u), 1)) \text{ .}$$ We know that $d\log(u)$ is the Eisenstein series $$E = \frac{p-1}{d} + \frac{24}{d}\cdot \sum_{k=0}^{\infty} \left(\sum_{m \mid k\atop \gcd(m,p)=1} m \right)q^k \in M_2(\Gamma_0(p)) \text{ .}$$ Thus, we have $L(d\log(u), s) = \frac{24}{d}\cdot (1-p^{1-s})\zeta(s-1)\zeta(s)$, so $L(d\log(u), 1) = -\frac{12}{d}\cdot \log(p)$ and $\lambda_{\infty}^{n} \cdot \lambda_0^{-n} = p^{-n\cdot \frac{12}{d}}$. Thus, we have $\lambda_{\infty} \cdot \lambda_0^{-1} = \epsilon \cdot p^{-\frac{12}{d}}$ where $\epsilon \in \{-1,1\}$ and $\epsilon=1$ if $n$ is odd. This concludes the proof of (\ref{main_comparison_Gamma_0_case_ii}).
\end{proof}
\begin{rem}
It would be interested to remove the sign ambiguity in Theorem \ref{intro_main_comparison_Gamma_0_case} (\ref{main_comparison_Gamma_0_case_iii}).
\end{rem}

\subsubsection{Comparison with de Shalit's extended $p$-adic period pairing}\label{section_deShalit}
We now prove Theorem \ref{intro_main_thm_comparison_p_adic}.
\begin{proof}
The proof makes use of de Shalit's explicit construction of $Q$, so we will use the results and notation of \cite{deShalit} (especially \S 1.5) without recalling them. Let $f$ be a meromorphic function on $X_0(p)$, defined over $\mathbf{Q}$, and such that $f \sim j^{-1}$ at the cusp $\Gamma_0(p) \infty$ and $f \sim j\circ w_p$ at the cusp $\Gamma_0(p) 0$. Write 
$$\text{div}(f) = D + (\Gamma_0(p) \infty) - (\Gamma_0(p) 0)$$
where $D = \sum_{i=1}^m n_i (P_i)$ is a divisor supported on $Y_0(p)$. Then, by definition, $\delta^{\#, \alg}( (\Gamma_0(p) \infty) - (\Gamma_0(p) 0))$ is the class of $-D$ in $J_{\Gamma_0(p)}^{\#}(\mathbf{Q})$. Recall that de Shalit denotes by $\Gamma$ a $p$-adic Schottky group uniformizing $X_0(p)$, $\mathfrak{H}_{\Gamma}$ the associated $p$-adic upper-half plane and $\tau : \mathfrak{H}_{\Gamma} \rightarrow X_0(p)(\mathbf{C}_p)$ the $p$-adic uniformization map. Thus, as a $\Gamma$-invariant function on $\mathfrak{H}_{\Gamma}$, we have 
$$f(z)= \lambda \cdot \Theta(\textbf{D}; z) \cdot \Theta(z_{\infty}^{(0)}, z_0^{(0)}; z)$$
where $\lambda \in \mathbf{C}_p^{\times}$, $\textbf{D}$ is some lift of $D$ to $\mathfrak{H}_{\Gamma}$ and $\Theta(\textbf{D}; z)$,  $\Theta(z_{\infty}^{(0)}, z_0^{(0)}; z)$ are the theta functions defined in \cite[\S 0.2]{deShalit}. We can assume that $\textbf{D}$ is disjoint from $\tau^{-1}(\mathcal{C}_{\Gamma_0(p)})$.
By \cite[\S 1.4]{deShalit}, under (\ref{p-adic_uniformization_generalized}), the class of $-D$ corresponds to the group homomorphism $\psi_D : \mathbf{Z}[S]\rightarrow K^{\times}$ given by 
$$\psi_D(e_i)=\frac{\Theta(\textbf{D}; z_0^{(i)})}{\Theta(\textbf{D}; z_{\infty}^{(i)})} \text{ .}$$
On the other hand, by definition \cite[\S 1.5]{deShalit} for all $i,j \in \{0, ..., g\}$ we have 
$$Q(e_j, e_i) = \lim_{z \rightarrow z_0^{(j)} \atop z' \rightarrow z_{\infty}^{(j)}} (j\circ w_p(z))^2\cdot \frac{\Theta(z',z;z_{\infty}^{(i)})}{\Theta(z',z;z_0^{(i)})} $$
where $z$ and $z'$ satisfy the constraint $\tau(z') = w_p(\tau(z))$.
By \cite[\S 0.2 Properties (e)]{deShalit}, we also have
\begin{align*}
Q(e_j, e_i) &= \lim_{z \rightarrow z_0^{(j)} \atop z' \rightarrow z_{\infty}^{(j)}} (j\circ w_p(z))^2\cdot \frac{\Theta(z_{\infty}^{(i)},z_0^{(i)};z')}{\Theta(z_{\infty}^{(i)},z_0^{(i)};z)} \\& = \lim_{z \rightarrow z_0^{(j)} \atop z' \rightarrow z_{\infty}^{(j)}}  \frac{j(z')\cdot \Theta(z_{\infty}^{(i)},z_0^{(i)};z')}{(j \circ w_p)^{-1}(z)\cdot \Theta(z_{\infty}^{(i)},z_0^{(i)};z)} \text{ .}
\end{align*}
On the other hand, we have by assumption:
\begin{align*}
1 &=  \lim_{z \rightarrow z_0^{(i)} \atop z' \rightarrow z_{\infty}^{(i)}} \frac{j(z') \cdot f(z')}{(j\circ w_p)^{-1}(z)\cdot f(z)}
\\& = \psi_D(e_i)^{-1}\cdot Q(e_i,e_0) \text{ .}
\end{align*}
Thus, we have $\psi_D(e_i) = Q(e_i,e_0)$. To conclude the proof of Theorem \ref{intro_main_thm_comparison_p_adic}, it suffices to prove the following result (expected in \cite[\S 1.5]{deShalit}, altough no proof is given).
\begin{prop}\label{symmetry_Q}
The pairing $Q$ is symmetric. 
\end{prop}
\begin{proof}
Recall (\cf \cite[\S 1.2 and 1.3]{deShalit}) that $\Gamma$ is a free group with generators denoted by $\{\alpha_1, ..., \alpha_g\}$, and that for all $i \in \{0, ..., g\}$ we have the relation $z_{\infty}^{(i)} = \alpha_i^{-1}(z_{\infty}^{(0)})$ (with the convention $\alpha_0=1$). To keep track of indices, let $z_i$ (resp. $z_i'$) be what we called $z$ (resp. $z'$) above (so $z_i$ goes to $z_0^{(i)}$ and $z_i'$ goes to $z_{\infty}^{(i)}$). As de Shalit did, we can and do assume that $z_0^{(i)}$ and $z_i$ do not depend on $i$.  By the argument of \cite[\S 3.1]{deShalit}, we have $z_i' = \alpha_i^{-1}(z_0')$. Thus, we have $z_i' = \alpha_i^{-1}\alpha_j(z_j')$ and $z_{\infty}^{(i)} = \alpha_i^{-1}\alpha_j(z_{\infty}^{(j)})$ for all $(i,j) \in \{0, ..., g\}^2$. We have:
\begin{align*}
Q(e_i, e_j) &= \lim_{z_i \rightarrow z_{0}^{(i)} \atop z_i' \rightarrow z_{\infty}^{(i)}} ((j\circ w_p)(z_i))^2 \cdot \frac{\Theta(z_i', z_i, z_{\infty}^{(j)})}{\Theta(z_i', z_i, z_{0}^{(j)})} 
\\&=  \lim_{z_j \rightarrow z_{0}^{(j)} \atop z_j' \rightarrow z_{\infty}^{(j)}} ((j\circ w_p)(z_j))^2 \cdot \frac{\Theta(\alpha_i^{-1}\alpha_j(z_j'), z_j, \alpha_j^{-1}\alpha_i(z_{\infty}^{(i)}))}{\Theta(\alpha_i^{-1}\alpha_j(z_j'), z_j, z_{0}^{(i)})} 
\\&=  \lim_{z_j \rightarrow z_{0}^{(j)} \atop z_j' \rightarrow z_{\infty}^{(j)}} ((j\circ w_p)(z_j))^2 \cdot \frac{\Theta(z_j', z_j, \alpha_j^{-1}\alpha_i(z_{\infty}^{(i)}))}{\Theta(z_j', z_j, z_{0}^{(i)})} \cdot \frac{\Theta(\alpha_i^{-1}\alpha_j(z_j'), z_j', \alpha_j^{-1}\alpha_i(z_{\infty}^{(i)}))}{\Theta(\alpha_i^{-1}\alpha_j(z_j'), z_j', z_{0}^{(i)})}
\\& =  \lim_{z_j \rightarrow z_{0}^{(j)} \atop z_j' \rightarrow z_{\infty}^{(j)}} ((j\circ w_p)(z_j))^2 \cdot \frac{\Theta(z_j', z_j, z_{\infty}^{(i)})}{\Theta(z_j', z_j, z_{0}^{(i)})} \cdot \frac{\Theta(\alpha_i^{-1}\alpha_j(z_j'), z_j', \alpha_j^{-1}\alpha_i(z_{\infty}^{(i)}))}{\Theta(\alpha_i^{-1}\alpha_j(z_j'), z_j', z_{0}^{(i)})} \cdot \frac{\Theta(z_j', z_j, \alpha_j^{-1}\alpha_i(z_{\infty}^{(i)}))}{\Theta(z_j', z_j, z_{\infty}^{(i)})}
\\& = Q(e_j, e_i) \cdot \lim_{z_j \rightarrow z_{0}^{(j)} \atop z_j' \rightarrow z_{\infty}^{(j)}}  \frac{\Theta(\alpha_i^{-1}\alpha_j(z_j'), z_j', \alpha_j^{-1}\alpha_i(z_{\infty}^{(i)}))}{\Theta(\alpha_i^{-1}\alpha_j(z_j'), z_j', z_{0}^{(i)})} \cdot \frac{\Theta(z_j', z_j, \alpha_j^{-1}\alpha_i(z_{\infty}^{(i)}))}{\Theta(z_j', z_j, z_{\infty}^{(i)})} \text{ .}
\end{align*}
By \cite[\S 0.2 Properties]{deShalit}, for any $a$, $b$ $\in \mathfrak{H}_{\Gamma}$ and any $\alpha \in \Gamma$, there exists a constant $c(a,b;\alpha)$ such that for all $z \in \mathfrak{H}_{\Gamma}$ not in $\Gamma \cdot a \cup \Gamma \cdot b$, we have $\Theta(a,b;z) = c(a,b;\alpha) \cdot \Theta(a,b;\alpha z)$.
Thus, we have:
\begin{align*}
\frac{\Theta(\alpha_i^{-1}\alpha_j(z_j'), z_j', \alpha_j^{-1}\alpha_i(z_{\infty}^{(i)}))}{\Theta(\alpha_i^{-1}\alpha_j(z_j'), z_j', z_{0}^{(i)})} &= c(\alpha_i^{-1}\alpha_j(z_j'),z_j'; \alpha_j^{-1}\alpha_i)^{-1} \cdot \frac{\Theta(\alpha_i^{-1}\alpha_j(z_j'), z_j', z_{\infty}^{(i)})}{\Theta(\alpha_i^{-1}\alpha_j(z_j'), z_j', z_{0}^{(i)})}
\end{align*}
and
$$\frac{\Theta(z_j', z_j, \alpha_j^{-1}\alpha_i(z_{\infty}^{(i)}))}{\Theta(z_j', z_j, z_{\infty}^{(i)})}=c(z_j',z_j, \alpha_j^{-1}\alpha_i)^{-1} \text{ .}$$
By construction, we have
$$Q(e_i-e_j, e_i) =  \frac{\Theta(\alpha_i^{-1}\alpha_j(z_j'), z_j', z_{\infty}^{(i)})}{\Theta(\alpha_i^{-1}\alpha_j(z_j'), z_j', z_{0}^{(i)})} \text{ .}$$
Thus we have:
\begin{align*}
Q(e_i,e_j) &= Q(e_j,e_i) \cdot Q(e_i-e_j,e_i)\cdot  \lim_{z_j \rightarrow z_{0}^{(j)} \atop z_j' \rightarrow z_{\infty}^{(j)}} c(\alpha_i^{-1}\alpha_j(z_j'),z_j'; \alpha_j^{-1}\alpha_i)^{-1} \cdot c(z_j',z_j, \alpha_j^{-1}\alpha_i)^{-1}
\\& = Q(e_j,e_i) \cdot Q(e_i-e_j,e_i) \cdot  \lim_{z_j \rightarrow z_{0}^{(j)} \atop z_j' \rightarrow z_{\infty}^{(j)}}  c(\alpha_i^{-1}\alpha_j(z_j'), z_j, \alpha_j^{-1}\alpha_i)^{-1} \\& = Q(e_j,e_i) \cdot Q(e_i-e_j,e_i) \cdot c(z_{\infty}^{(i)}, z_0^{(i)}, \alpha_j^{-1}\alpha_i)^{-1} \text{ .}
\end{align*}
But by construction we have $c(z_{\infty}^{(i)}, z_0^{(i)}, \alpha_j^{-1}\alpha_i)  = Q(e_i-e_j,e_i)$.
Thus, we have proved that $Q(e_i,e_j)=Q(e_j,e_i)$. 
\end{proof}
\end{proof}

We conclude this paragraph by two important properties of $Q$, which follow easily from the work of de Shalit \cite{deShalit}. 

\begin{prop}\label{Q_T_equiv}
\begin{enumerate}
\item\label{Q_T_equiv_i} The pairing $Q : \mathbf{Z}[S] \times \mathbf{Z}[S] \rightarrow K^{\times}$ is $\mathbb{T}$-equivariant, \ie for all $T \in \mathbb{T}$ and $(x,y) \in  \mathbf{Z}[S] \times \mathbf{Z}[S]$, we have $Q(Tx, y) =  Q(x,Ty)$.

\item\label{Q_T_equiv_ii}  Modulo the principal units of $K^{\times}$, the pairing $Q$ takes values in $\mathbf{Q}_p^{\times}$ and is $\Gal(\mathbf{F}_{p^2}/\mathbf{F}_p)$-equivariant, \ie for all  $h \in \Gal(\mathbf{F}_{p^2}/\mathbf{F}_p)$ and $(x,y) \in  \mathbf{Z}[S] \times \mathbf{Z}[S]$, we have $Q(hx, hy) = h(Q(x,y))$ (modulo principal units).
\end{enumerate}
\end{prop}
\begin{proof}
Proof of (\ref{Q_T_equiv_i}). The restriction of $Q \otimes_{\mathbf{Z}} \mathbf{Z}_{\ell}$ to $\mathbf{Z}[S]^0 \times \mathbf{Z}[S]$ (and thus to $\mathbf{Z}[S] \times \mathbf{Z}[S]^0$ by symmetry of $Q$) is known to be $\mathbb{T}$-equivariant, since it has an interpretation in termes of the generalized Jacobian $J^{\sharp}_{\Gamma_0(p)}$ (\cf \cite[\S 2.3]{deShalit_X_1}). It follows that for any $T \in \mathbb{T}$, the quantity $\lambda_{i,j}:=\frac{Q(Te_i, e_j)}{Q(e_i, Te_j)}$ does not depend on $(i,j) \in \{0,.., g\}^2$ (recall that we have denoted $S = \{e_0, ..., e_g\}$). We have $\lambda_{i,i}=1$ by symmetry of $Q$, so for all $(i,j) \in \{0,.., g\}^2$ we have $Q(Te_i, e_j) = Q(e_i, Te_j)$. By bilinearity of $Q$, this proves (\ref{Q_T_equiv_i}).

Proof of (\ref{Q_T_equiv_ii}). The fact that $Q$ takes values in $\mathbf{Q}_p^{\times}$ modulo principal units is \cite[Lemma 1.7]{deShalit}. Let $h$ be the non-trivial element of $\Gal(\mathbf{F}_{p^2}/\mathbf{F}_p)$. For any $i \in \{0, ... g\}$, we have $h(e_i) = U_p(e_i)$ where $U_p$ is the Hecke operator of index $p$. Thus, 
$$Q(h(e_i), h(e_j)) = Q(U_p(e_i), U_p(e_j)) = Q(e_i, U_p^2(e_j)) = Q(e_i, e_j)$$
where we have used (\ref{Q_T_equiv_i}) and the fact that $U_p^2=1$. Since $Q$ takes values in  $\mathbf{Q}_p^{\times}$ modulo principal units, we have $h(Q(e_i,e_j)) = Q(e_i,e_j)$ modulo principal units.
\end{proof}

\begin{rems}
\begin{enumerate}
\item The pairing $Q$ itself should be $\mathbf{Q}_p^{\times}$-valued (and thus automatically $\Gal(\mathbf{F}_{p^2}/\mathbf{F}_p)$-equivariant).
 \item The Hecke-equivariance property is specific to the level $\Gamma_0(p)$. Indeed, an analogue of Oesterl\'e's conjecture at level $\Gamma(2) \cap \Gamma_0(p)$ (basically replacing the $j$-invariant by Legendre $\lambda$ invariant) was proved in \cite{Lecouturier_Betina}. It appears that $Q$ does not commute with the Hecke operator $U_2$ (although it commutes with $T_{\ell}$ if $\ell \neq 2,p$ and with $U_p$).
\end{enumerate}
\end{rems}

\subsubsection{Application to Galois representations}\label{application_Galois_representations}
To conclude this paper, we give an application of our results to the construction of Galois representations. 

Recall that $\mathbb{T}$ is the Hecke algebra acting on $M_2(\Gamma_0(p))$, generated by the Hecke operators $T_{q}$ for primes $q \neq p$ and by $U_p$. Let $\mathbb{T}^0$ be the cuspidal Hecke algebra, acting faithfully on the cuspidal modular forms $S_2(\Gamma_0(p))$. Let $I \subset \mathbb{T}$ be the Eisenstein ideal, generated by the operators $T_{q} - q-1$ ($q\neq p$ prime) and $U_p-1$. A maximal ideal $\mathfrak{m}$ of $\mathbb{T}$ is said to be \textit{Eisenstein} if $I \subset \mathfrak{m}$. We know that $\mathbb{T}/I = \mathbf{Z}$ and that $\mathbb{T}^0/I = \mathbf{Z}/n\mathbf{Z}$ \cite[Proposition II.9.7]{Mazur_Eisenstein} (recall that $n$ is the numerator of $\frac{p-1}{12}$). In particular, the residue characteristic $\ell$ of a maximal Eisenstein ideal divides $\frac{p-1}{12}$, and such a maximal ideal is unique.

Let $\mathcal{G} = J_{\Gamma_0(p)}^{\#}(\overline{\mathbf{Q}})/\text{Im}(\delta^{\#, \alg})$, where $\text{Im}(\delta^{\#, \alg})$ is the image of $\delta^{\#, \alg}$. If $\ell$ is a prime, let $\mathcal{V}_{\ell}$ be the  $\ell$-adic Tate module of $\mathcal{G}$, \ie $\mathcal{V}_{\ell} := \varprojlim_{n} \mathcal{G}[\ell^n]$. This is a $\mathbb{T}[\Gal(\overline{\mathbf{Q}}/\mathbf{Q})]$-module. Note that $\mathcal{V}_{\ell}$ is also the $\ell$-adic Tate module of $J_{\Gamma_0(p)}^{\#}(\mathbf{C})/\text{Im}(\delta^{\#, \an})$ if $\ell \neq 2$ and of $J_{\Gamma_0(p)}^{\#}(\overline{\mathbf{Q}}_p)/\text{Im}(\delta^{\#, p-\text{adic}})$ (for all prime $\ell$) by Theorems \ref{intro_main_comparison_Gamma_0_case} and \ref{intro_main_thm_comparison_p_adic} respectively. If $\mathfrak{m}$ is a maximal ideal of residue characteristic $\ell$ of $\mathbb{T}$, let $\mathbb{T}_{\mathfrak{m}}$ and $\mathcal{V}_{\mathfrak{m}}:=\mathcal{V}_{\ell} \otimes_{\mathbb{T} \otimes_{\mathbf{Z}} \mathbf{Z}_{\ell}} \mathbb{T}_{\mathfrak{m}}$ be the $\mathfrak{m}$-adic completion of $\mathbb{T}$ and $\mathcal{V}_{\ell}$ respectively. 

\begin{prop}\label{multiplicity_one_V}
The $\mathbb{T}_{\mathfrak{m}}$-module $\mathcal{V}_{\mathfrak{m}}$ is free of rank $2$ if and only if $J_0(p)[\mathfrak{m}]$ is free of rank $2$ over $\mathbb{T}^0/\mathfrak{m}$ (where we view abusively $\mathfrak{m}$ as an ideal of $\mathbb{T}^0$). By Mazur \cite[Lemma II.15.1 and Corollary II.15.2]{Mazur_Eisenstein}, the latter assertion is always true, except possibly if $\mathfrak{m}$ is a non-Eisenstein maximal ideal of characteristic $2$ and $\mathfrak{m}$ is ordinary (\ie the image of $U_p$ in $\mathbb{T}^0/\mathfrak{m}$ is non-zero).
\end{prop}
\begin{proof}
If $\mathfrak{m}$ is not Eisenstein, then this follows from Theorem \ref{intro_main_comparison_Gamma_0_case} (\ref{main_comparison_Gamma_0_case_i}) and the fact that the $\mathbf{Z}[\mathcal{C}_{\Gamma_0(p)}]$ is annihilated by $I$. 

Assume that $\mathfrak{m}$ is Eisenstein, of residue characteristic $\ell$. By Theorem \ref{intro_main_thm_comparison_p_adic}, $\mathcal{V}_{\ell}$ is the $\ell$-adic Tate module of the $\mathbb{T}$-module $\Hom(\mathbf{Z}[S], \overline{\mathbf{Q}}_p^{\times})/q(\mathbf{Z}[S])$. Since we know that $q$ is injective, we get an exact sequence of $\mathbb{T} \otimes_{\mathbf{Z}} \mathbf{Z}_{\ell}$-modules
$$0 \rightarrow \Hom(\mathbf{Z}[S], \mathbf{Z}_{\ell}) \rightarrow \mathcal{V}_{\ell} \rightarrow \mathbf{Z}_{\ell}[S] \rightarrow 0 \text{ .}$$
We get an exact sequence of $\mathbb{T}_{\mathfrak{m}}$-modules
\begin{equation}\label{ordinary_exact_sequence_supersingular}
0 \rightarrow \Hom(\mathbf{Z}[S], \mathbf{Z}) \otimes_{\mathbb{T}} \mathbb{T}_{\mathfrak{m}} \rightarrow \mathcal{V}_{\mathfrak{m}} \rightarrow \mathbf{Z}[S] \otimes_{\mathbb{T}} \mathbb{T}_{\mathfrak{m}} \rightarrow 0 \text{ .}
\end{equation}
By \cite[Theorem 0.5]{Emerton_supersingular} and \cite[Corollary II.16.3]{Mazur_Eisenstein}, the $\mathbb{T}_{\mathfrak{m}}$-modules $\mathbf{Z}[S] \otimes_{\mathbb{T}} \mathbb{T}_{\mathfrak{m}}$ and $ \Hom(\mathbf{Z}[S], \mathbf{Z}) \otimes_{\mathbb{T}} \mathbb{T}_{\mathfrak{m}} $ are free of rank $1$. Thus, $\mathcal{V}_{\mathfrak{m}}$ is free free of rank $2$ over $\mathbb{T}_{\mathfrak{m}}$. This concludes the proof of Proposition \ref{multiplicity_one_V} (another approach if $\ell \neq 2$ would have been to use modular symbols via Theorem \ref{intro_main_comparison_Gamma_0_case} instead of the supersingular module). 
\end{proof}

\begin{prop}\label{galois_rep_duality}
There is a Hecke and Galois equivariant perfect $\mathbf{Z}_{\ell}$-bilinear pairing $\langle \cdot, \cdot \rangle : \mathcal{V}_{\ell} \times \mathcal{V}_{\ell} \rightarrow \mathbf{Z}_{\ell}(1)$ (the equivariance means that for any $(x,y) \in  \mathcal{V}_{\ell} \times \mathcal{V}_{\ell}$, $T \in \mathbb{T}$ and $g \in \Gal(\overline{\mathbf{Q}}/\mathbf{Q})$, we have $\langle gx, gy \rangle = \chi_{\ell}(g)\cdot \langle x, y \rangle$ and $\langle T x, y \rangle = \langle x, Ty \rangle$).
\end{prop}
\begin{proof}
It suffices indeed to prove the analogous statement for $\mathcal{V}_{\mathfrak{m}}$ for all maximal ideal $\mathfrak{m}$ of $\mathbb{T}$ containing $\ell$. If $\mathfrak{m}$ is not Eisenstein, the pairing $\mathcal{V}_{\mathfrak{m}} \times \mathcal{V}_{\mathfrak{m}} \rightarrow \mathbf{Z}_{\ell}(1)$ comes from the Weil pairing on $J_0(p)$. Assume now that $\mathfrak{m}$ is Eisenstein. Let $\mathcal{V}_{\mathfrak{m}}^* = \Hom_{\mathbf{Z}_{\ell}}(\mathcal{V}_{\mathfrak{m}}, \mathbf{Z}_{\ell}(1))$, with the action of $\Gal(\overline{\mathbf{Q}}/\mathbf{Q})$ given by $(g\cdot \varphi)(x) = \chi_{\ell}(g)\cdot \varphi(g^{-1}\cdot x)$ for all $\varphi \in \mathcal{V}_{\mathfrak{m}}^*$ and $x \in \mathcal{V}_{\mathfrak{m}}$. By Proposition \ref{multiplicity_one_V}, we can choose a basis $(e_1,e_2)$ of the $\mathbb{T}_{\mathfrak{m}}$-module $\mathcal{V}_{\mathfrak{m}}$. Let $(e_1^*, e_2^*)$ be the dual basis in $\mathcal{V}_{\mathfrak{m}}^*$, where $e_i^* \in \Hom_{\mathbf{Z}_{\ell}}(\mathbb{T}_{\mathfrak{m}}, \mathbf{Z}_{\ell})$. By \cite[Corollary II.15.2]{Mazur_Eisenstein}, the $\mathbb{T}_{\mathfrak{m}}$-module $\Hom_{\mathbf{Z}_{\ell}}(\mathbb{T}_{\mathfrak{m}}, \mathbf{Z}_{\ell})$ is free of rank one, so that $(-e_2^*, e_1^*)$ is a basis of the $\mathbb{T}_{\mathfrak{m}}$-module $\mathcal{V}_{\mathfrak{m}}^*$. The Galois representation $ \Gal(\overline{\mathbf{Q}}/\mathbf{Q}) \rightarrow \GL_2(\mathbb{T}_{\mathfrak{m}})$ in that basis is equal to $\rho$, so we get an isomorphism of $\mathbb{T}_{\mathfrak{m}}[\Gal(\overline{\mathbf{Q}}/\mathbf{Q})]$-modules $\mathcal{V}_{\mathfrak{m}}^* \simeq \mathcal{V}_{\mathfrak{m}}$.
\end{proof}

Assume from now on, and until the end of the paper, that $\ell$ is a prime dividing the numerator of $\frac{p-1}{12}$. Fix an embedding $\overline{\mathbf{Q}} \hookrightarrow \overline{\mathbf{Q}}_p$. This fixes the choice of a decomposition group $G_p$ at $p$ in $\Gal(\overline{\mathbf{Q}}/\mathbf{Q})$. We denote by $I_p \subset G_p$ the inertia subgroup.

Fix a cocycle $b \in Z^1(\Gal(\overline{\mathbf{Q}}/\mathbf{Q}), \mathbf{Z}_\ell(1))$ whose class in $H^1(\Gal(\overline{\mathbf{Q}}/\mathbf{Q}), \mathbf{Z}_\ell(1))$ corresponds to the class on $p^{\frac{12}{d}}$ in $\mathbf{Q}^{\times}/(\mathbf{Q}^{\times})^{\ell}$ via Kummer theory. Let $\overline{b} : \Gal(\overline{\mathbf{Q}}/\mathbf{Q}) \rightarrow \mathbf{F}_\ell(1)$ be the reduction of $b$ modulo $\ell$. Let $\overline{\rho} : \Gal(\overline{\mathbf{Q}}/\mathbf{Q})  \rightarrow \GL_2(\mathbf{F}_\ell)$ be given by
$$ \overline{\rho} = \begin{pmatrix} \overline{\chi}_{\ell} & \overline{b} \\ 0 & 1 \end{pmatrix} \text{ .}$$
Let $\overline{L}$ be the line in $\mathbf{F}_{\ell}^{2}$ spanned by the vector $(1,0)$. Note that $\overline{L}$ is the unique line fixed (pointwise) by $\overline{\rho}(I_p)$

We consider the following classical deformation problem. Let $\mathcal{C}$ be the category of local Artinian rings with residue field $\mathbf{F}_{\ell}$. Let $\Def : \mathcal{C} \rightarrow \underline{\text{Set}}$ be the functor such that if $A \in \mathcal{C}$, then $\Def(A)$ is the set of strict-equivalence classes of morphisms $\rho : \Gal(\overline{\mathbf{Q}}/\mathbf{Q}) \rightarrow \GL_2(A)$ such that the following conditions hold:
\begin{enumerate}[label=(\alph*)]
\item\label{def_condition_reduction} The reduction of $\rho$ modulo the maximal ideal of $A$ is $\overline{\rho}$.
\item \label{def_condition_ramification} The representation $\rho$ is unramified outside $p$ and $\ell$.
\item \label{def_condition_determinant} The determinant of $\rho$ if $\chi_p$ (where we abusively view $\chi_p$ as $A^{\times}$ valued via the ring homomorphism $\mathbf{Z}_p \rightarrow A$).
\item \label{def_condition_finite_flat} The representation $\rho$ is finite flat at $\ell$ (meaning that the restriction of $\rho$ to a decomposition group $\Gal(\overline{\mathbf{Q}}_{\ell}/\mathbf{Q}_{\ell})$ arises from the $\overline{\mathbf{Q}}_{\ell}$-points of a finite flat group scheme over $\mathbf{Z}_{\ell}$).
\item \label{def_condition_semistable} There is a line $L$ in $A^2$ stable by $\rho(I_p)$.
\end{enumerate}

Note that $\overline{\rho}$ gives an element of $\Def(\mathbf{F}_{\ell})$ so our deformation problem makes sense.
Since the endomorphisms of $\mathbf{F}_{\ell}^2$ commuting with $\overline{\rho}$ are the scalars, we know that $\Def$ is pro-representable by a local Noetherian $\mathbf{Z}_{\ell}$-algebra $R$. We now prove Theorem \ref{intro_thm_galois_rep}.
\begin{proof}
Let $\overline{\mathcal{V}} := \mathcal{V}_{\mathfrak{m}}/\mathfrak{m}\cdot  \mathcal{V}_{\mathfrak{m}}$. This is a $\mathbf{F}_{\ell}$-vector space of rank $2$, with an action of $\Gal(\overline{\mathbf{Q}}/\mathbf{Q})$. The kernel of the projection $\mathbb{T}\rightarrow \mathbb{T}^0$ is $\mathbf{Z}\cdot T_0$ for some $T_0 \in \mathbb{T}$. One can choose $T_0$ so that $T_0-n\in I$ \cite[Proposition 1.8]{Emerton_supersingular}. In particular, we have $T_0 \in \mathfrak{m}$. Let $\mathcal{V}_{\mathfrak{m}}^0 := \mathcal{V}_{\mathfrak{m}}/T_0\cdot \mathcal{V}_{\mathfrak{m}}$ and $\mathcal{V}_{\mathfrak{m}}^{\Eis} := \mathcal{V}_{\mathfrak{m}}/I\cdot  \mathcal{V}_{\mathfrak{m}}$; these are naturally $\Gal(\overline{\mathbf{Q}}/\mathbf{Q})$-modules. Note that $\mathcal{V}_{\mathfrak{m}}^{\Eis}$ is free of rank $2$ over $\mathbf{Z}_{\ell}$.
\begin{lem}\label{lemma_galois_rep}
The projection map $f : \mathcal{V}_{\mathfrak{m}} \rightarrow \mathcal{V}_{\mathfrak{m}}^{\Eis} \times \mathcal{V}_{\mathfrak{m}}^{0}$ is injective, $\Gal(\overline{\mathbf{Q}}/\mathbf{Q})$-equivariant, and its image has finite index in the fiber product $\mathcal{V}_{\mathfrak{m}}^{\Eis} \times_{\overline{\mathcal{V}}} \mathcal{V}_{\mathfrak{m}}^{0}$. Furthermore, there is a $\mathbf{Z}_{\ell}$-basis of $\mathcal{V}_{\mathfrak{m}}^{\Eis}$ such that the action of $\Gal(\overline{\mathbf{Q}}/\mathbf{Q})$ is given by $\begin{pmatrix} \chi_{\ell} & b \\ 0 & 1 \end{pmatrix}$. In particular, the $\Gal(\overline{\mathbf{Q}}/\mathbf{Q})$-module $\overline{\mathcal{V}}$ is isomorphic to $\overline{\rho}$.
\end{lem}
\begin{proof}
We have $I \cdot T_0 = 0$ in $\mathbb{T}$, so $I\cdot \mathcal{V}_{\mathfrak{m}} \cap T_0\cdot \mathcal{V}_{\mathfrak{m}}$ is annihilated by $I$ and by $T_0$, and hence by $n$ since $T_0-n \in I$. Since $\mathcal{V}_{\mathfrak{m}}$ is a free $\mathbf{Z}_{\ell}$-module, we get $I\cdot \mathcal{V}_{\mathfrak{m}} \cap T_0\cdot \mathcal{V}_{\mathfrak{m}} = 0$. Thus $f$ is injective and takes values in $\mathcal{V}_{\mathfrak{m}}^{\Eis} \times_{\overline{\mathcal{V}}} \mathcal{V}_{\mathfrak{m}}^{0}$ since $I+(T_0) \subset \mathfrak{m}$. Furthermore, the image of $f$ into $\mathcal{V}_{\mathfrak{m}}^{\Eis} \times_{\overline{\mathcal{V}}} \mathcal{V}_{\mathfrak{m}}^{0}$ has finite index since $\rk_{\mathbf{Z}_{\ell}}(\mathcal{V}_{\mathfrak{m}}) = 2\cdot \rk_{\mathbf{Z}_{\ell}}(\mathbb{T}_{\mathfrak{m}}) = \rk_{\mathbf{Z}_{\ell}}( \mathcal{V}_{\mathfrak{m}}^{\Eis} \times_{\overline{\mathcal{V}}} \mathcal{V}_{\mathfrak{m}}^{0})$. 

There is a $\mathbf{Z}_{\ell}$ linear isomorphism $T_0 \cdot \mathbb{T}_{\mathfrak{m}} \xrightarrow{\sim} \mathbb{T}_{\mathfrak{m}}/I\cdot  \mathbb{T}_{\mathfrak{m}}$ sending $T_0\cdot T$ to the class of $T$ modulo $I$. This induces a $\Gal(\overline{\mathbf{Q}}/\mathbf{Q})$-equivariant isomorphism $T_0\cdot \mathcal{V}_{\mathfrak{m}} \xrightarrow{\sim} \mathcal{V}_{\mathfrak{m}}^{\Eis}$. The facts that $T_0 \cdot J_{\Gamma_0(p)}^{\#}$ is the image of $\mathbf{G}_m\times \mathbf{G}_m$ in $J_{\Gamma_0(p)}^{\#}$, $T_0 - n \in I$ and $n\cdot \delta^{\#, \alg}((\infty)-(0)) = (1, p^{\frac{12}{d}}) \in \mathbf{Q}^{\times} \times \mathbf{Q}^{\times}$ imply that $T_0 \cdot \mathcal{V}_{\mathfrak{m}}$ is isomorphic to the $\ell$-adic Tate module of the $1$-motive $\mathbf{Z} \rightarrow \mathbf{G}_m$ sending $1$ to the class of $p^{\frac{12}{d}}$, which is given by $\begin{pmatrix} \chi_{\ell} & b \\ 0 & 1 \end{pmatrix}$. This concludes the proof of Lemma \ref{lemma_galois_rep}. 
\end{proof}

We now prove that $\mathcal{V}_{\mathfrak{m}}$ gives rise to a continuous homomorphism $R \rightarrow \mathbb{T}_{\mathfrak{m}}$ of local $\mathbf{Z}_{\ell}$-algebras. By Lemma \ref{lemma_galois_rep} , there is a $\mathbb{T}_{\mathfrak{m}}$ basis of $\mathcal{V}_{\mathfrak{m}}$ giving rise to a representation $\rho : \Gal(\overline{\mathbf{Q}}/\mathbf{Q}) \rightarrow \GL_2(\mathbb{T}_{\mathfrak{m}})$ such that the reduction of $\rho$ modulo $\mathfrak{m}$ is $\overline{\rho}$, which is the deformation condition \ref{def_condition_reduction}. The conditions \ref{def_condition_ramification}, \ref{def_condition_determinant} and \ref{def_condition_finite_flat}  follow from the analogous statement for the $\Gal(\overline{\mathbf{Q}}/\mathbf{Q})$-modules $\mathcal{V}_{\mathfrak{m}}^{\Eis}$ (by Lemma \ref{lemma_galois_rep}) and $\mathcal{V}_{\mathfrak{m}}^{0}$ (which is well-known). Condition \ref{def_condition_semistable}  follows from (\ref{ordinary_exact_sequence_supersingular}). Thus, we get a canonical continuous homomorphism $u : R \rightarrow \mathbb{T}_{\mathfrak{m}}$ of local $\mathbf{Z}_{\ell}$-algebras. It remains to show that $u$ is an isomorphism if $\ell \geq 5$. This follows from \cite[Corollary 7.1.3]{WWE}: the authors construct a universal pseudo-deformation ring $R'$ with an isomorphism $v:R' \xrightarrow{\sim} \mathbb{T}_{\mathfrak{m}}$. Obviously, $v$ is the composition of a map $R' \rightarrow R$ and $u$. Since $v$ is an isomorphism, so is $u$.
\end{proof}

\bibliography{biblio}
\bibliographystyle{plain}
\newpage

\end{document}